\documentclass[12pt]{article}

\usepackage{amssymb}
\usepackage{amsmath}
\usepackage{url}
\usepackage[backref,colorlinks=true]{hyperref}
\allowdisplaybreaks

\begin{document}
\newenvironment {proof}{{\noindent\bf Proof.}}{\hfill $\Box$ \medskip}

\newtheorem{theorem}{Theorem}[section]
\newtheorem{lemma}[theorem]{Lemma}
\newtheorem{condition}[theorem]{Condition}
\newtheorem{proposition}[theorem]{Proposition}
\newtheorem{remark}[theorem]{Remark}
\newtheorem{definition}[theorem]{Definition}
\newtheorem{hypothesis}[theorem]{Hypothesis}
\newtheorem{corollary}[theorem]{Corollary}
\newtheorem{example}[theorem]{Example}
\newtheorem{descript}[theorem]{Description}
\newtheorem{assumption}[theorem]{Assumption}

\newcommand{\ba}{\begin{align}}
\newcommand{\ea}{\end{align}}

\def\P{\mathbb{P}}
\def\R{\mathbb{R}}
\def\E{\mathbb{E}}
\def\N{\mathbb{N}}

\renewcommand {\theequation}{\arabic{section}.\arabic{equation}}
\def \non{{\nonumber}}
\def \hat{\widehat}
\def \tilde{\widetilde}
\def \bar{\overline}

\def\ind{{\mathchoice {\rm 1\mskip-4mu l} {\rm 1\mskip-4mu l}
{\rm 1\mskip-4.5mu l} {\rm 1\mskip-5mu l}}}

\title{\Large\ {\bf The Fleming-Viot limit of an interacting spatial population with fast density regulation}}

\author{Ankit Gupta
\thanks{ A part of this work was done when I was a graduate student at University of Wisconsin, Madison. I wish to sincerely thank my thesis adviser, 
Prof. Thomas G. Kurtz, for his continuing support and encouragement. This work was completed while I was holding a postdoctoral appointment under 
Prof. Sylvie M\'{e}l\'{e}ard and Prof. Vincent Bansaye at \'Ecole Polytechnique, Paris.  
The hospitality and support provided by them during my stay is gratefully acknowledged.
} 
\thanks{ This work was supported by the professoral chair Jean Marjoulet, the project MANEGE (Mod\`eles Al\'eatoires en \'Ecologie, G\'en\'etique et \'Evolution) 
of ANR (French national research agency) and the Chair Modelisation Mathematique et Biodiversite VEOLIA-\'Ecole Polytechnique-MNHN-F.X.} \\ \ 
Centre de Math\'ematiquees Appliqu\'ees \\ \'Ecole Polytechnique \\  UMR CNRS 7641 Route de Saclay\\ 91128 Palaiseau C\'{e}dex, France. \\ gupta@cmap.polytechnique.fr
}
\date{\today}
\maketitle
\begin{abstract}
We consider population models in which the individuals reproduce, die and also migrate in space. The population size scales according to some parameter $N$, which can have different interpretations depending on the context. Each individual is assigned a mass of $1/N$ and the total mass in the system is called \emph{population density}. The dynamics has an intrinsic density regulation mechanism that drives the population density towards an equilibrium. We show that under a timescale separation between the \emph{slow} migration mechanism and the \emph{fast} density regulation mechanism, the population dynamics converges to a Fleming-Viot process as the scaling parameter $N$ approaches $\infty$. We first prove this result for a basic model in which the birth and death rates can only depend on the population density. In this case we obtain a \emph{neutral} Fleming-Viot process. We then extend this model by including position-dependence in the birth and death rates, as well as, offspring dispersal and immigration mechanisms. We show how these extensions add \emph{mutation} and \emph{selection} to the limiting Fleming-Viot process. All the results are proved in a multi-type setting, where there are $q$ types of individuals interacting with each other. We illustrate the usefulness of our convergence result by discussing applications in population genetics and cell biology.
\end{abstract}

\noindent Keywords: spatial population; density dependence; Fleming-Viot process; cell polarity; population genetics; Wright-Fisher model; diffusion approximation\\ 
\noindent Mathematical Subject Classification (2010):  60J68; 60J85 ; 60G57; 60F99
\medskip

\section{Introduction}\label{introduction}

Density-dependent models are well-known in population biology. In these models, the birth and death rates of individuals may depend on the 
\emph{density} of the population, where the term density refers to the population size under a suitably chosen normalization.
Many models in ecology, epidemiology and immunology can be suitably described by such models (see Thieme \cite{Thieme}). 
Considering the molecules of chemical species as the individuals in a population, we can also view a chemical reaction network as a density-dependent population model.

Density-dependent models are appealing because one can easily account for interactions among individuals by appropriately specifying 
the birth and death rates as functions of the population density. 
For competitive interactions, as in the Lotka-Volterra model (see \cite{LotkaBook,Volterra,Nisbet}), the death rate increases with the 
population density while for cooperative interactions, as in the Allee model (see \cite{Alee}), the birth rate increases with the population 
density. Density dependent models are good candidates for modeling natural populations that cannot grow indefinitely due to 
the limited availability of certain vital resources or due to severe competition at large population sizes. 
By having the death rate dominate the birth rate at \emph{large} densities, one can ensure that the population density does not go beyond a certain threshold.

For a population having $q$ types of individuals, the population density is a $q$ dimensional vector whose $i$-th component gives the density of the 
population of the $i$-th type. For such a multi-type population, a density-dependent model can be written in the deterministic setting as a system of $q$ 
ordinary differential equations. If all the trajectories of this system stay within a compact set at all times, 
then we say that the population dynamics has a \emph{density regulation mechanism}. 
Such a mechanism is called \emph{equilibrating} if all the trajectories reach a fixed point for this system as time goes to infinity. In such a situation, this fixed point is 
called the \emph{equilibrium population density}.      

In this paper, we will consider population models in which the individuals live in a \emph{geographical region} $E$, that is a compact metric space. Even though we have a spatial structure, for us the population density will always denote the population size divided by a normalization parameter $N$. In other words, our notion of population density is \emph{global} in the sense that it carries no information about the distribution of individuals in $E$. This is unlike other models of spatial populations where the population density is a spatially varying function specifying the local concentration of individuals at each location. The normalization parameter $N$ will be a large positive integer which 
can have various interpretations depending on the context. In ecological models, $N$ can be taken to be the \emph{carrying capacity} of a habitat, which is the maximum 
number of individuals that the habitat can support with its resources. In epidemic models, $N$ is usually the total 
population size, while in chemical reaction networks, $N$ measures the \emph{volume} of the system. In each of these cases, the population size at any time is of order $N$.

For the moment assume that all the individuals have the same type and that they reproduce, die and also migrate in $E$. At the time of birth, the offspring gets the same location as its parent.
The population consists of approximately $N$ individuals of a \emph{mass} of $1/N$ each. The population 
density at any time is just the total mass of the individuals that are alive. 
Suppose that the birth and death rates of the individuals depend on the population density in such a way that they induce a density regulation mechanism 
which is equilibrating. We also assume that the migration mechanism operates at a timescale that is $N$ times 
\emph{slower} than the density regulation mechanism. In such a setting, we can view the dynamics of the empirical measure of the population as a measure-valued 
Markov process parameterized by $N$. Our goal is to understand how this family of Markov processes behaves as $N \to \infty$. The population dynamics has two 
timescales separated by $N$. If we observe the process at the \emph{fast} timescale, then the effect of migration vanishes in the limit and it is uninteresting to 
consider the population with a spatial structure in this case. Therefore we will observe the dynamics at the \emph{slow} timescale and examine its behaviour in the infinite population 
limit. Since the density regulation mechanism is \emph{fast}, it will have enough time to re-equilibrate the population density between any two events at the slow timescale.
Hence in the limit $N \to \infty$, we would expect the population density to remain equilibrated at all times. We will show that it is indeed the case. 
However our main task is to understand the dynamics of the spatial distribution of the population in the infinite population limit. We will prove that in the limit, 
the spatial distribution of the population evolves according to a Fleming-Viot process which takes values in the space of probability measures over $E$. 
This process was introduced in the context of population genetics by Fleming and Viot \cite{FV} in 1979 and it has been very well-studied since then. An excellent survey
of Fleming-Viot processes is given by Ethier and Kurtz \cite{EK93}. The model we just described will be called the \emph{basic model.}
In this model, the birth and death rates of an individual were density dependent, but independent of the location of the individual. It shall be seen later that the 
limiting Fleming-Viot process in this case is \emph{neutral}, in the parlance of population genetics.  
If we add \emph{small} position-dependent terms to the birth and death rates, then in the limit we obtain a Fleming-Viot process with \emph{genetic selection}. Perhaps 
unsurprisingly, altering the birth rate this way leads to \emph{fecundity selection}, while altering the death rate leads to \emph{viability selection}, in the limiting Fleming-Viot process. We also consider extensions of the basic model by allowing for \emph{offspring dispersal} (offspring is born away from the parent) or 
\emph{immigration}. Such extensions add extra \emph{mutations} to the limiting Fleming-Viot process.
     
The results mentioned in the previous paragraph are proved in a multi-type setting. The population has $q$ types of individuals and  
each type of individual can give birth to an individual of each type. All the individuals are migrating in $E$ according to a type-dependent mechanism. 
Now we can view the joint dynamics of the empirical measures of the $q$ sub-populations as a Markov process parameterized by $N$. 
We make similar assumptions on the dynamics as before.
Again in the limit $N \to \infty$, the population density (which is now a $q$-dimensional vector) stays at an equilibrium at all times. 
Assuming the irreducibility of an underlying interaction matrix, we show that in the limit all the $q$ sub-populations become \emph{spatially inseparable}. 
This means that on any patch of $E$, either there is no mass present or there is mass of each type present in a proportion determined by the 
equilibrium density. Moreover the spatial distribution of each of the $q$ sub-populations evolves according to a single Fleming-Viot process. 
This Fleming-Viot process can be seen as describing the limiting dynamics of a \emph{mixed} population, formed by taking a suitable density-dependent 
convex combination of the $q$ sub-populations.

In ecological models, the individuals need resources to survive and reproduce.
Normally in spatial population models, resources are assumed to have a fixed distribution in space. 
As individuals move, they find the unexploited resources and compete for them \emph{locally} with other individuals present in their neighbourhood. 
Such a model is different from the models we consider in two ways. 
Firstly, due to the local nature of the interactions, the density is locally regulated rather than globally regulated as in our models. 
Secondly, since the discovery of resources is tied to the movement of individuals, it is reasonable to assume that both migration and birth-death mechanisms operate at the same 
timescale. For such spatial models, Oelschl\"ager \cite{Oelschlager} has shown in a multi-type setting that the dynamics converges in the infinite population limit to a 
system of reaction-diffusion partial differential equations. Such equations are in widespread use in biology (see Fife \cite{rdpde1}). 
We now discuss the conditions under which our models can be useful.  
Consider a situation where the resource is not fixed but rapidly mixing in the whole space. This resource is shared by all the individuals in the population. An individual may deplete the resource locally but its effect is felt globally due to the rapid mixing. This gives rise to global density dependence in a spatial population. If the individuals move very slowly in comparison to their resource consumption mechanism (which is linked to their birth and death mechanisms), then we have a situation in which our models can be used. 

This paper is motivated by our earlier work \cite{GUPTA} in which we study the phenomenon of cell polarity using a model considered here. Cell polarity refers to the clustering of molecules on the cell membrane. This clustering is essential to trigger various other cellular processes, such as bud formation \cite{AAWWref9} or immune response \cite{Weiner}. Therefore understanding how cells establish and maintain polarity is of vital importance. In \cite{AAWW}, Altschuler et. al. devised a mathematical model for this phenomenon, by abstracting the mechanisms that are commonly found in cells exhibiting polarity. Their model has a fixed number of molecules that can either reside on the membrane or in the cytosol. These molecules move slowly on the membrane but diffuse rapidly in the cytosol. The dynamics has a \emph{positive feedback} mechanism which allows a membrane molecule to pull a cytosol molecule to its location on the membrane. This mechanism is like a birth process in which a membrane molecule gives birth by exploiting the common resource (cytosol molecules) shared by all the membrane molecules. Since the migration of membrane molecules is \emph{slow} and the mixing of the resource is \emph{fast}, this model can be viewed as a model described in this paper (see Section \ref{sec:cellpolarity} for details). Therefore the results in this paper are applicable and we obtain a Fleming-Viot process in the infinite population limit. In \cite{GUPTA} we prove this convergence\footnote{This convergence was proved in \cite{GUPTA} using the technique of particle representation described in \cite{DK99}. This technique cannot be easily extended to the multi-type setting of this paper. Therefore the convergence proof in this paper is vastly different.} and use the limiting process to answer some interesting questions about the onset and structure of cell polarity. The model studied in \cite{GUPTA} is rather simplistic as all the molecules are assumed to be identical. Most cells that exhibit polarity have molecules of many different types participating in the feedback mechanism and migrating on the membrane in different ways (see \cite{DN,AAWWref9,Takaku}). It is natural to ask if the Fleming-Viot convergence is valid in this general framework. The results in this paper show that it is indeed the case as long as certain basic elements of the dynamics are preserved. This ensures that the analysis in \cite{GUPTA} can be extended to more complicated (and realistic !) models for cell polarity. We discuss this example further in Section \ref{sec:cellpolarity}.

Note that the geographical space $E$ can be considered as the space of \emph{genetic traits}. This casts our models into the setting of population genetics. The spatial migration can be seen as \emph{mutation} that may happen at any time during the life of an individual, while the offspring dispersal mechanism is like \emph{mutation} that can only happen at the time of birth of an offspring. We assume that the reproduction is \emph{clonal} in the absence of mutation. The position-dependent birth and death mechanism is analogous to the \emph{selection} mechanism in population genetics. Hence it is not surprising that spatial migration, offspring dispersal and position-dependent birth and death mechanisms correspond to mutation and selection in the limiting Fleming-Viot process. What is more interesting is that the \emph{sampling mechanism} arises naturally from our models in the infinite population limit. This sampling mechanism is a key feature of the standard models in population genetics, such as the Wright-Fisher model, the Moran model and their variants (see \cite{Wright,Moran,EwensBook}). This mechanism makes the models tractable by keeping the population size constant. It is done by matching the birth of an individual with the death of another individual chosen uniformly from the population. It is obvious that such a mechanism is quite unrealistic, at least for finite populations which are naturally fluctuating. However our Fleming-Viot convergence result shows that one can recover this sampling mechanism in the infinite population limit if the dynamics has an equilibrating density regulation mechanism that acts at a faster timescale than other events. It is well-known that a  Fleming-Viot process arises in the infinite population limit of an appropriately scaled version of the Wright-Fisher or the Moran model (see \cite{FV} and \cite{EK93}). Therefore if all the individuals have the same type (that is, $q=1$) and $h_{\textnormal{eq}}$ is the equilibrium population density, then for a large (but finite) value of the scaling parameter $N$, our models will have \emph{roughly} the same dynamical behaviour as a suitably chosen Wright-Fisher or Moran model with the constant population size $N h_{\textnormal{eq}}$. This insight provides a justification for the assumption of a constant population size in population genetics models. Most of the mathematical literature on population genetics is concerned with two types of questions. In the absence of mutation, one wants to know the probability and the time of \emph{fixation} of a particular genetic trait. The term \emph{fixation} describes the event in which the whole population has the same genetic trait. In the presence of mutation, one attempts to investigate the properties of the stationary distribution, if such a distribution is present. These questions are difficult to answer for finite populations and one typically answers them by studying the limiting Fleming-Viot process. Our discussion shows that for large $N$, the fixation times and probabilities or the stationary distribution will be approximately the same for our model and the corresponding Wright-Fisher or Moran model. We illustrate this point through an example in Section \ref{sec:log1}.

This paper is organized as follows. In Section 2 we describe the mathematical models that we consider and state our main results. 
In Section 3 we discuss the aforementioned  applications of our results in greater detail. Finally in Section 4 we prove the main results.

\section*{Notation}
We now introduce some notation that we will use throughout this paper. Let $\R$, $\R_+$, $\R_{*}$, $\N$ and $\N_{0}$ denote the sets of all reals, nonnegative reals, positive reals, positive integers and nonnegative integers respectively. For any $a,b \in \R$, their minimum is given by $a \wedge b$. 

Let $\left\| \cdot\right\|$ and $\langle \cdot , \cdot\rangle$ denote the standard Euclidean norm and inner product in $\R^n$ for any $n \in \N$. Moreover for 
any $v = (v_1,\dots,v_n)\in \R^n$, the norms $\left\| v\right\|_1 $ and $\left\| v\right\|_{\infty}$ are defined as
$\left\| v\right\|_1 = \sum_{i=1}^n |v_i|$ and $\left\| v\right\|_{\infty} = \max_{1 \leq i \leq n} |v_i|$. 
The vectors of all zeros and all ones in $\R^n$ are denoted by $\bar{0}_n$ and $\bar{1}_n$ respectively.
Let $\mathbb{M}(n,n)$ be the space of all $n \times n$ matrices with real entries. For any 
$M \in \mathbb{M}(n,n)$, the entry at the $i$-th row and the $j$-th column is indicated by $M_{ij}$. Its infinity norm is defined as
$ \left\| M \right\|_{\infty} = \max_{1 \leq i \leq n} \sum_{j=1}^{n} |M_{ij}|$
and its transpose and inverse are indicated by $M^{T}$ and $M^{-1}$ respectively. 
The symbol $I_n$ refers to the identity matrix in 
$\mathbb{M}(n,n)$. For any $v = (v_1,\dots,v_n) \in \R^n$, $\textrm{Diag}(v)$ refers to the matrix in $\mathbb{M}(n,n)$ whose non-diagonal entries are all $0$ and 
whose diagonal entries are $v_1,\dots,v_n$. A matrix in $\mathbb{M}(n,n)$ is called \emph{stable} if all its eigenvalues have strictly negative real parts.
While multiplying a matrix with a vector we always regard the vector as a column vector. 

Let $U \subset \R^n$ and $V \subset \R^m$. Then for any $k \in \N_0$, the class $C^k(U,V)$ refers to the set of all those functions $f$ that are defined on some open 
set $O \subset \R^n$ containing $U$ such that $f(x) \in V$ for all $x \in U$ and $f$ is $k$-times continuously differentiable at any $x \in O$.

Let $(S,d)$ be a metric space. Then by $B(S) \left( C(S) \right) $ we refer to the set of all bounded (continuous) real-valued Borel measurable functions. 
If $S$ is compact, then $C(S) \subset B(S)$ and both $B(S)$ and $C(S)$ are Banach spaces under the sup norm  $\left\| f \right\|_\infty= \sup_{x \in S } \left|  f(x)  \right|$. Recall that a class of functions in $B(S)$ is called an \emph{algebra} if it is closed under finite sums and products.
Let $\mathcal{B}(S)$ be the Borel sigma field on $S$. By $\mathcal{M}_F(S)$ and $\mathcal{P}(S)$ we denote the space of all finite positive Borel measures 
and the space of all Borel probability measures respectively. These measure spaces are equipped with the weak topology. For any $f \in B(S)$ and $\mu \in \mathcal{M}_F(S)$ 
let 
\[\langle f, \mu \rangle = \int_{E} f(x)\mu(dx).\] 

The space of cadlag functions (that is, right continuous functions with left limits) from $[0,\infty)$ to $S$ is denoted by $D_{S} [0,\infty)$ and it is endowed with the Skorohod topology (for details see Chapter 3, Ethier and Kurtz \cite{EK}). 
The space of continuous functions from $[0,\infty)$ to $S$ is denoted by $C_{S} [0,\infty)$ and it is endowed with the topology of uniform convergence over compact sets.
An operator $A$ on $B(S)$ is a linear mapping that maps any function in its domain $\mathcal{D}(A) \subset B(S)$ to a function in $B(S)$.
The notion of the \emph{martingale problem} associated to an operator $A$ is introduced and developed in Chapter 4, Ethier and Kurtz \cite{EK}. 
In this paper, by a solution of the martingale problem for $A$ we mean a measurable stochastic process $X$ with paths 
in $D_{S} [0,\infty)$ such that for any $f \in \mathcal{D}(A)$,
\[f(X(t)) - \int_{0}^{t} A f(X(s))ds\]
is a martingale with respect to the filtration generated by $X$. 
For a given initial distribution $\pi \in \mathcal{P}(S)$, a solution $X$ of the martingale problem for $A$ is a solution of the martingale problem for $(A,\pi)$ if $\pi = \P X(0)^{-1}$. 
If such a solution $X$ exists uniquely for all $\pi \in \mathcal{P}(S)$, then we say that the martingale problem for $A$ is well-posed. Additionally, we say that 
$A$ is the generator of the process $X$.

Throughout the paper $\Rightarrow$ denotes convergence in distribution.

\section{Model descriptions and the main result}

Our first task is to describe the models that we consider in the paper. 
As mentioned in Section \ref{introduction}, we model a population which resides in some compact metric space $E$ and in which the individuals have one of $q$ possible types. We denote these types by elements in the set $Q = \{1,2,\dots,q\}$. 
We identify each individual located at $x \in E$ with 
the Dirac measure $\delta_x$, concentrated at $x$. Moreover each individual is assigned a \emph{mass} of $1/N$ where $N \in \N$ is our scaling parameter. 
For any $i \in Q$, the population of type $i$ individuals can be represented by an atomic measure of the form
\begin{align*}
\mu_i = \frac{1}{N} \sum_{ j = 1}^{n_i} \delta_{x^i_j}, 
\end{align*}
where $n_i$ is the total number of type $i$ individuals and $x^i_1,\dots,x^i_{n_i} \in E$ are their locations. Define the space of atomic measures scaled by $N$ as 
\begin{align}
\label{defnmna}
\mathcal{M}_{N,a} ( E ) = \left\{ \frac{1}{N} \sum_{ j = 1}^{n} \delta_{x_j} : n \in \mathbb{N}_0 \textrm{ and } x_1,\dots,x_n \in E \right\}.
\end{align}
Note that $\mathcal{M}_{N,a} (E) \subset \mathcal{M}_F(E)$, where $ \mathcal{M}_F(E)$ is the space of finite positive measures. 
Let $\mathcal{M}^q_{N,a}(E)$ and $\mathcal{M}^q_{F} (E)$ be the spaces formed by taking products of 
$q$ copies of $\mathcal{M}_{N,a} (E)$ and $\mathcal{M}_{F} (E)$ respectively. Since for each $i \in Q$, the type $i$ population can be represented by a 
measure $\mu_i \in \mathcal{M}_{N,a}(E)$ , the entire population can be represented by a $q$-tuple of measures 
$\mu = (\mu_1,\dots,\mu_q) \in \mathcal{M}^q_{N,a} (E)$. 

Let $1_E$ denote the constant function in $C(E)$ which maps each point in $E$ to $1$.
Define the \emph{density map} $H : \mathcal{M}_F^q(E) \to \R^q_{+}$ as the continuous function given by
\begin{align}
\label{defn:densitymap}
H(\mu_1,\dots,\mu_q) = \left( \langle 1_E, \mu_1\rangle, \dots,\langle 1_E, \mu_q\rangle  \right) \textrm{ for any } \mu = (\mu_1,\dots,\mu_q) \in \mathcal{M}_F^q(E).  
\end{align}
We will refer to $h = H(\mu)$ as the density vector corresponding to $\mu \in \mathcal{M}_F^q(E)$. Note that if the population is represented 
by a $\mu \in \mathcal{M}^q_{N,a} (E)$ and if $h = (h_1,\dots,h_q)$ is the corresponding density vector, then $h_i$ is just the total number of type $i$ individuals 
divided by $N$. The density vector $h$ contains no information about the distribution of individuals on $E$.

\subsection{The type-dependent migration mechanism }
\label{migrationmechanism}
In our models, each individual of type $i \in Q$ will migrate according to an independent $E$-valued Markov process with generator $B_i$. 
We will assume that each operator $B_i$ generates a \emph{Feller} semigroup on $C(E)$ (see Chapter 4 in Ethier and Kurtz \cite{EK}). 
Furthermore we assume that there is an algebra of functions 
$\mathcal{D}_0 \subset C(E)$ which is dense in $C(E)$, contains $1_E$ and satisfies
\begin{align}
\label{defn:classD0}
\mathcal{D}_0 \subset \mathcal{D}(B_i) \textrm{ for all } i \in Q. 
\end{align}
The martingale problem corresponding to each $B_i$ is well-posed and any solution is a strong Markov process with sample paths in $D_E[0,\infty)$ 
(see Theorem 4.2.7 and Corollary 4.2.8 in \cite{EK}).

We now formally describe how this type-dependent migration of individuals translates into the evolution of our population in the space $\mathcal{M}^q_{N,a} (E)$.
 For each $n \in \N$, define a space of atomic probability measures as
\begin{align*}
\mathcal{P}_{n,a} = \left\{ \frac{1}{n} \sum_{j =1 }^{n} \delta_{x_j} : x_1,\dots,x_n \in E \right\}
\end{align*}
and a class of continuous real-valued functions over $\mathcal{P}(E)$ by
\begin{align}
 \label{defc0}
\mathcal{C}_0 = \left\{ F(\nu) = \prod_{l = 1}^{m} \langle f_l, \nu \rangle \ : \ f_1,\dots,f_m \in \mathcal{D}_0  \textrm{ and } m \in \N \right\}.
\end{align}
Suppose that $\nu = (1/n) \sum_{j=1}^n \delta_{x_j} \in \mathcal{P}_{n,a}$ and $F(\nu) = \prod_{l = 1}^{m} \langle f_l, \nu \rangle \in \mathcal{C}_0$. For positive integers $k \leq m$, let $P^m_k$ be the set of onto functions from $\{1,\dots,m\}$ to $\{1,\dots,k\}$ and for any $p \in P^m_k$ and $l = 1,\dots,k$ let
\begin{align}
\label{fplform}
f^{(p)}_l(x) = \prod_{ j \in p^{-1}(l)} f_j(x). 
\end{align}
Then we can write
\begin{align}
\label{Fformabove}
F(\nu)  = n^{-m} \prod_{l = 1}^m \left( \sum_{j=1}^n f_l(x_j) \right) & = n^{-m} \sum_{i_1,\dots,i_m = 1}^{n} \prod_{l=1}^{m} f_l(x_{i_l}) \notag \\
& = n^{-m} \sum_{k=1}^m \sum_{p \in P^m_k} \sum_{j_1 ,\dots, j_k = 1 }^n  \prod_{l = 1}^{k} f^{(p)}_l(x_{j_l}),
\end{align}
where the last term has summation over distinct choices of $j_1 ,\dots, j_k \in \{1,\dots,n\}$. For each $i \in Q, n \in \N$ we now define an operator 
$\mathbf{B}^n_i : \mathcal{D}(\mathbf{B}^n_i) = \mathcal{C}_0 \to B\left( \mathcal{P}_{n,a}\right)$ by 
\begin{align}
 \label{defbni}
\mathbf{B}^n_i F(\nu) = n^{-m} \sum_{k=1}^m \sum_{p \in P^m_k} \sum_{j_1 ,\dots, j_k = 1 }^n  \sum_{l = 1}^{k}  B_i f^{(p)}_l(x_{j_l}) 
\prod_{r = 1, r \neq l}^{k} f^{(p)}_r(x_{j_r}), 
\end{align}
where $F \in  \mathcal{C}_0$ is given by \eqref{Fformabove}. Observe that any $F \in \mathcal{C}_0$ is bounded and
\begin{align}
\label{boundednessofbni}
\sup_{n \in \N} \sup_{\nu \in \mathcal{P}_{n,a}} |\mathbf{B}^n_i F(\nu) | < \infty. 
\end{align}
One can easily verify that the martingale problem for each $\mathbf{B}^n_i$ is well-posed. 
If $\nu_0 = (1/n) \sum_{j=1}^n \delta_{x_j} \in \mathcal{P}_{n,a}$ then the solution of the martingale problem for 
$(\mathbf{B}^n_i, \delta_{\nu_0})$ is just the empirical measure process of a system of $n$ individuals moving in $E$ according to independent Markov processes with 
generators $B_i$ and initial positions $x_1,\dots,x_n$. For more details see Section 2.2 in Dawson\cite{DawsonEcole}.

For any $f_1,\dots,f_m \in \mathcal{D}_0$ consider a function $\hat{F} : \mathcal{M}^q_F(E) \to \R$ of the form
\begin{align}
\label{hatfform}
\hat{F}(\mu) = \prod_{j=1}^{m} \left( \sum_{i = 1}^q  c_i(h) \langle f_j, \mu_i \rangle \right), 
\end{align}
where $h = H(\mu)$ and $c : \R^q_{+} \to  \R^q_{+}$ is a function that satisfies
\begin{align}
\label{coefficientcondition}
\sup_{h \in \R^q_{+}} \langle c(h),h\rangle < \infty. 
\end{align}
Such a function $\hat{F}$ is bounded because
\begin{align*}
\sup_{\mu \in \mathcal{M}^q_F(E)} |\hat{F} (\mu)| \leq \left( \max_{l =1,\dots,m} \|f_l\|_{\infty} \right)^m \left( \sup_{h \in \R^q_{+}} \langle c(h),h\rangle \right)^m.
\end{align*}
Define a class of functions by
\begin{align}
\label{classc0q}
 \mathcal{C}^q_0  = \left\{ \hat{F} \in B\left( \mathcal{M}^q_F(E)\right) \textrm{ : } \hat{F} \textrm{ is given by } \eqref{hatfform} \textrm{ and } c 
\textrm{ satisfies } \eqref{coefficientcondition} \right\}.
\end{align}

If $\mu = (\mu_1,\dots,\mu_q) \in \mathcal{M}^q_{N,a}(E)$, then for each $i \in Q$, $\mu_i$ has the form $\mu_i = (1/N) \sum_{j=1}^{n_i} \delta_{x^i_j}$. 
Let $\nu_i = (1/n_i) \sum_{j=1}^{n_i} \delta_{x^i_j}$ if $n_i >0$ and $\nu_i = \delta_{x_0}$ if $n_i = 0$, where $x_0 \in E$ is arbitrary. 
Let $P^m_k$ be as before. Pick a $\hat{F} \in \mathcal{C}^q_0 $ of the form \eqref{hatfform}. 
For any $p \in P^m_k$ and $l = 1,\dots,k$ define $F^{(p)}_{l} \in \mathcal{C}_0$ by
$F^{(p)}_{l}(\nu) =  \langle f^{(p)}_l, \nu \rangle$,
where $f^{(p)}_l$ is given by \eqref{fplform}. We can write the function  $\hat{F}$ as 
\begin{align*}
\hat{F}(\mu) &=   \sum_{ i_1,\dots, i_m = 1}^{q} \prod_{j=1}^{m} c_{i_j}(h) \langle f_j, \mu_{i_j} \rangle  = \sum_{ i_1,\dots, i_m = 1}^{q} \prod_{j=1}^{m} ( c_{i_j}(h) h_{i_j} )\langle f_j, \nu_{i_j} \rangle \\
& = \sum_{k=1}^{m} \sum_{p \in P^m_k} 
\sum_{ l_1,\dots,l_k = 1 }^q (c_{l_1}(h) h_{l_1})^{|p^{-1}(1)|}\dots (c_{l_k} (h) h_{l_k})^{|p^{-1}(k)|} \prod_{r = 1}^{k} F^{(p)}_r(\nu_{l_r}),
\end{align*}
where the last term has summation over distinct choices of $l_1 ,\dots, l_k \in \{1,\dots,q\}$. 
Let $\mathbf{B}^N : \mathcal{D}(\mathbf{B}^N) = \mathcal{C}^q_0  \to B\left( \mathcal{M}^q_{N,a}(E)\right)$ be the operator whose action on any 
$\hat{F} \in \mathcal{C}^q_0$ written in this form is given by
\begin{align}
\label{maindefbn}
\mathbf{B}^N \hat{F}(\mu) = & \sum_{k=1}^{m} \sum_{p \in P^m_k} 
\sum_{ l_1,\dots,l_k = 1 }^q (c_{l_1}(h) h_{l_1})^{|p^{-1}(1)|}\dots (c_{l_k}(h) h_{l_k})^{|p^{-1}(k)|} \notag \\
& \times \sum_{r = 1}^{k} \mathbf{B}^{n_{l_r}}_{l_r} F^{(p)}_r(\nu_{l_r}) \prod_{j = 1, j \neq r}^{k} F^{(p)}_j(\nu_{l_j}),
\end{align}
where for any $n \in \N, i \in Q$, the operator $\mathbf{B}^n_i$ is given by \eqref{defbni}. For convenience $\mathbf{B}^0_i$ is defined to be the identity map on 
$\mathcal{C}_0$. The function $\mathbf{B}^N \hat{F}$ is bounded due to \eqref{boundednessofbni} and \eqref{coefficientcondition}.

The martingale problem for $\mathbf{B}^N$ is well-posed because the martingale problem for $\mathbf{B}^n_i$ is well-posed for each $i \in Q , n \in \N$. 
To see this suppose that $\mu_0 = (\mu_{0,1},\dots,\mu_{0,q}) \in  \mathcal{M}^q_{N,a}(E)$ and for each $i \in Q$, $\mu_{0,i} \in \mathcal{M}_{N,a}(E)$ has the form
$\mu_{0,i} = (1/N) \sum_{j=1}^{n_i} \delta_{x^i_{j}}.$  
If $n_i>0$ let $\nu_{0,i} = (1/n_i) \sum_{j=1}^{n_i} \delta_{x^i_{j}}$ and if $n_i = 0$ let $\nu_{0,i} = \delta_{x_0}$ for some arbitrary $x_0 \in E$. 
For each $i \in Q$, let $\{ \nu_i(t) : t \geq 0\}$ be the unique solution of the martingale problem for $(\mathbf{B}^{n_i}_i, \delta_{\nu_{0,i}})$. Then the process $\{ \mu(t) : t \geq 0  \}$ given by
\begin{align*}
\mu(t) = \left( \frac{n_1}{N} \nu_1(t),\dots, \frac{n_q}{N} \nu_q(t)\right), 
\end{align*}
is the unique solution to the martingale problem for $(\mathbf{B}^N,\delta_{\mu_0})$.

\subsection{The density regulation mechanism}\label{density_regulation_mechanism}

We mentioned in Section \ref{introduction} that in our models, the birth and death rates of individuals depend on the population density in such a way that 
they induce an equilibrating density regulation mechanism. We now describe this mechanism in our multi-type setting.

For each $i,j \in Q $, let $\beta_{ij}$ be a bounded function in $C^2(\R^q_{+},\R_+)$ and for each $i \in Q $ let $\rho_i$ be any function in $C^2(\R^q_{+},\R_+)$. 
For any $h \in \R^q_+$, let the matrix $A(h) \in \mathbb{M}(q,q)$ and the vector $\theta(h) \in \R^q$ be given by
\begin{align}
\label{defninteractionmatrix}
A_{ij}(h) & = \left\{
\begin{array}{cc}
\beta_{ji}(h) & \textrm{ if }  i \neq j  \\
\beta_{ii}(h) - \rho_i(h) & \text{ otherwise } \\
\end{array}
\right\},
\end{align}
and
\begin{align}
\label{defnthetah}
\theta(h) = A(h)h.
\end{align}

Consider $\theta$ to be a vector field over $\R^q_+$. Observe that if $h \in \R^q_+$ is such that $h_i = 0$ then 
$\theta_i(h) \geq 0$. This shows that any solution to the initial value problem
\begin{align}
\label{firstodeivp}
\frac{d h}{d t } = \theta(h), \quad h(0) = h_0 \in \R^q_+
\end{align}
stays inside $\R^q_+$ for all positive times for which it is defined. Standard existence and uniqueness theorems imply that for any $h_0 \in \R^q_+$, 
there is a solution $h(t)$ of $\eqref{firstodeivp}$ defined on some maximal time interval $[0,a)$. Moreover if $a < \infty$ 
then $\left\| h(t)\right\|_1 \to \infty$ as $t \to a_{-}$. Since $\beta_{ij}$ is bounded for each $i,j \in Q$, there is a positive constant $C_\theta$ such that
\begin{align}
\label{controlovertheta}
\sum_{i=1}^q \theta_i(h) \leq \sum_{i,j=1}^q \beta_{ji}(h)h_j \leq C_{\theta} \|h\|_1 \textrm{ for all } h \in \R^q_+
\end{align}
and hence 
\begin{align}
\label{ripeforgronwall}
\frac{d \left\| h(t)\right\|_1}{dt} = \sum_{i=1}^q \frac{d h_i(t)}{dt} = \sum_{i = 1}^{q} \theta_i(h(t)) \leq C_{\theta} \|h(t)\|_1.
\end{align}
Therefore using Gronwall's inequality and \eqref{ripeforgronwall} we obtain that $\left\|h(t)\right\|_1 \leq \left\|h(0)\right\|_1 e^{C_\theta t}$ for all $t \in [0,a)$ 
and so $\left\| h(t)\right\|_1$ cannot go to $\infty$ as $t \to a_{-}$. Thus $a = \infty$ and this shows that for any $h_0 \in \R^q_+$, the initial value problem 
\eqref{firstodeivp} has a unique solution which is defined for all $t \geq 0$. 

Let $\psi_\theta :\R^q_+ \times \R_+ \to  \R^q_+$ be the \emph{flow} associated to the vector field $\theta$. This means that $\psi_\theta$ satisfies 
\begin{align}
\label{flowequation}
\psi_\theta(x,t) = x + \int_{0}^{t} \theta(\psi_\theta(x,s))ds \ \textrm{    for all } (x,t) \in \R^q_+ \times \R_+.
\end{align}
This flow is well-defined because of the arguments given in the preceding paragraph. In fact since $\theta$ is in $C^2(\R^q_+,\R^q)$, the map $\psi_\theta$ is in 
$C^2(\R^q_+ \times \R_+,\R^q_+)$. This map also satisfies the semigroup property
\begin{align}
\label{semigroup_property}
\psi_\theta(x,t+s) = \psi_{\theta}( \psi_\theta(x,t) ,s) \ \textrm{ for all } x\in \R^q_+ \textrm{ and } s,t \in \R_+. 
\end{align}
We will say that a set $U \subset \R^q_+$ is $\psi_\theta$-invariant if for all $t \geq 0$, $\psi_\theta(U,t) \subset U $ where 
$\psi_\theta(U,t)= \{\psi_\theta(x,t) : x \in U\}$.
Before we proceed, we need to make some more assumptions.
\begin{assumption}
\label{mainassumptionsonr}
\begin{itemize}
\item[(A)] There exists a vector $h_{\textnormal{eq}} \in \R^q_{+}$, $h_{\textnormal{eq}} \neq \bar{0}_q$ such that $\theta(h_{\textnormal{eq}})= A(h_{\textnormal{eq}}) h_{\textnormal{eq}} = \bar{0}_q$.
\item[(B)] The Jacobian matrix $[J\theta(h_{\textnormal{eq}})]$ is stable.
\item[(C)] The matrix $A(h_{\textnormal{eq}})$ is irreducible, that is, there does not exist a permutation matrix $P \in \mathbb{M}(q,q)$ 
such that the matrix $P A(h_{\textnormal{eq}}) P^{-1}$ is block upper-triangular.
\item[(D)] For each $i,j \in Q$, the functions $\rho_i$ and $\beta_{ij}$ are analytic at $h_{\textnormal{eq}}$, that is, 
they agree with their Taylor series expansion in a neighbourhood of $h_{\textnormal{eq}}$. 
\end{itemize}
\end{assumption}
Part (A) says that there is a nonzero vector $h_{\textnormal{eq}} \in \R^q_{+}$ which is a fixed point for the flow $\psi_\theta$. 
Part (B) implies that this fixed point $h_{\textnormal{eq}}$ is \emph{asymptotically stable} for this flow. 
The significance of part (C) will become clear later in this section. Part (D) is a technical condition that we require to prove our main result.

We define the region of attraction of the fixed point $h_{\textnormal{eq}}$ for the flow $\psi_\theta$ by
\begin{align}
\label{defueq}
U_{\textnormal{eq}} = \left\{ h \in \R^q_+ : \lim_{t \to \infty} \psi_\theta(h,t) = h_{\textnormal{eq}}  \right\}. 
\end{align}
Part (B) of Assumption \ref{mainassumptionsonr} and Lemma 3.2 in Khalil \cite{Khalil}, ensure that $U_{\textnormal{eq}}$ is a $\psi_\theta$-invariant open set in 
$\R^q_+$. Note that $U_{\textnormal{eq}}$ may not be an open set in $\R^q$.

We are now ready to describe the density regulation mechanism. 
Suppose that when the population density vector is $h \in \R^q_+$, then for each $i,j \in Q$, 
an individual of type $i$ gives birth to an individual of type $j$ at rate $\beta_{ij}(h)$ and an individual of type $i$ dies at rate $\rho_i(h)$. 
At the time of birth, the offspring is placed at the same location in $E$ as its parent. Note that the birth and death rates do not depend 
on the location of the individuals. If the scaling parameter is $N \in \N$ and the mass of each individual is $1/N$, then we can view this 
density-dependent population dynamics as a Markov process over the state space $\mathcal{M}^q_{N,a}(E)$ with generator 
$\mathbf{R}^N : \mathcal{D}(\mathbf{R}^N) = B(\mathcal{M}^q_{N,a}(E)) \to B(\mathcal{M}^q_{N,a}(E))$ defined as follows.
For any $F \in B(\mathcal{M}^q_{N,a}(E))$ and any $\mu \in \mathcal{M}^q_{N,a}(E)$ with $h = H(\mu)$
\begin{align}
\label{defroperatorrn}
\mathbf{R}^N F(\mu) &= \sum_{i, j \in Q } N \int_{E} \beta_{ij}(h) \left( F \left( \mu +\frac{1}{N} \delta^j_{x} \right) - F(\mu) \right)\mu_i(dx)  \\
& +    \sum_{i \in Q }N \int_{E } \rho_{i}(h) \left( F \left( \mu -\frac{1}{N} \delta^i_{x} \right) - F(\mu) \right)\mu_i(dx), \notag
\end{align}
where for any $\mu= (\mu_1,\dots,\mu_q) \in \mathcal{M}^q_{N,a}(E)$, $j \in Q$ and $x \in E$ 
\begin{align*}
\mu \pm \frac{1}{N} \delta^j_x  = \left( \mu_1,\dots,\mu_{j-1}, \mu_j \pm \frac{1}{N} \delta_x, \mu_{j+1},\dots,\mu_q \right).
\end{align*}

Concrete results on the well-posedness of the martingale problem for $\mathbf{R}^N$ will come later. 
First let us understand how this operator $\mathbf{R}^N$ drives the population density to an equilibrium value. Let
$\{ \mu^N(t) : t \geq 0\}$ be a $\mathcal{M}^q_{N,a}(E)$-valued Markov process with generator $\mathbf{R}^N$ and let $\{ h^N(t) : t \geq 0\}$ be the corresponding density process defined by 
$h^N(t) = H(\mu^N(t))$. Assume that $h^N(0) \to h_0$ a.s.\ as $N \to \infty$ and $h_0$ is some vector in $U_{\textnormal{eq}}$. From Theorem 11.2.1 in \cite{EK}, one can conclude that as $N \to \infty$, $\{h^N(t) : t \geq 0\}$ converges in the Skorohod topology in $D_{\R^q}[0,\infty)$ to the process $\{ \psi_\theta(h_0,t) : t \geq 0\}$. 
Since $h_0 \in U_{\textnormal{eq}}$, $\psi_\theta(h_0,t) \rightarrow h_{\textnormal{eq}}$ as $t \to \infty$ which indicates that for a large $N$, 
the density process $h^N(\cdot)$ gets closer and closer to $h_{\textnormal{eq}}$ with time. This shows, at least informally, that the operator $\mathbf{R}^N$ drives the population 
density towards $h_{\textnormal{eq}}$. Henceforth we shall refer to the vector $h_{\textnormal{eq}}$ as the \emph{equilibrium population density}. 
Of course, this discussion totally ignores how $\mathbf{R}^N$ affects the spatial configuration of the population. Our main goal in this paper is to discover how 
the spatial distribution of the population evolves when the density is equilibrated at a faster timescale in comparison to the other mechanisms.

For any $h \in \R^q_+$, the matrix $A(h)$ given by \eqref{defninteractionmatrix} signifies how the various types of individuals interact when the population density is $h$. Part (C) of Assumption \ref{mainassumptionsonr} means that at the equilibrium population density, all the types of individuals are communicating with each other 
by influencing each other's birth and death rates.

\subsection{Mathematical Models}\label{mathmodels}

We now describe the models that we consider in this paper. We start with a basic model which only has spatial migration along with density regulation. 
We then extend this model by adding other features such as position dependence in the birth and death rates, offspring dispersal and immigration.
All the models will be parameterized by the scaling parameter $N$, with $1/N$ being the mass of each individual in the population. 
We will describe each model by specifying the generator of the associated Markov process. The well-posedness of the martingale problems corresponding to these generators is given by Proposition \ref{prop:wellposedness}. Our main results are presented in Section \ref{mainresultsection}. 

\subsubsection{Basic Model}\label{sec:basicmodel}

In this model, the individuals migrate according to the type-dependent migration mechanism specified in Section \ref{migrationmechanism} and
their birth and death rates regulate 
the population density as described in Section \ref{density_regulation_mechanism}. 
If the scaling parameter is $N$, then at any time we represent the population as a measure in $\mathcal{M}^q_{N,a}(E)$. The population evolves due to the following events.
\begin{itemize}
 \item Each individual of type $i \in Q$ migrates in $E$ according to an independent Markov process with generator $B_i$.
\item When the population density vector is $h \in \R^q_+$, each individual of type $i \in Q$ gives birth to an individual of type $j \in Q$ at rate $N \beta_{ij}(h)$. At the time of birth, the offspring is placed at the same location as its parent.
\item  When the population density vector is $h \in \R^q_+$, each individual of type $i \in Q$ dies at rate $N \rho_i(h)$.  
\end{itemize}
This population dynamics can be viewed as a $\mathcal{M}^q_{N,a}(E)$-valued Markov process whose generator 
$\mathbf{A}^N_0 : \mathcal{D}(\mathbf{A}^N_0) = \mathcal{C}_0^q \to B(\mathcal{M}^q_{N,a}(E))$ is given by
\begin{align}
\label{defn:an0}
\mathbf{A}^N_0 \hat{F} (\mu) & =  \mathbf{B}^N \hat{F} ( \mu ) + N^2 \sum_{i, j \in Q } 
\int_{E} \beta_{ij}(h) \left( \hat{F}  \left( \mu +\frac{1}{N} \delta^j_{x} \right) - \hat{F} (\mu) \right)\mu_i(dx)  \\
& +  N^2  \sum_{i \in Q } \int_{E } \rho_{i}(h) \left( \hat{F}  \left( \mu -\frac{1}{N} \delta^i_{x} \right) - \hat{F} (\mu) \right)\mu_i(dx), \notag
\end{align}
for any $\hat{F}  \in \mathcal{C}_0^q$. Here $h = H(\mu)$ is the density vector corresponding to $\mu$.

\subsubsection{Model with position dependence in the birth and death rates}\label{sec:positiondependentmodel}

In the above model, the birth and death rates of individuals do not depend on their location. 
One may want to consider models in ecology where some spatial locations are more 
advantageous for reproduction or some locations are more hazardous for survival. If we think of $E$ as the space of genetic traits, 
then one may consider models in which the trait of an individual influences its chances of reproduction or survival. 
To capture such situations we now introduce a model in which the birth and death rates of an individual can vary with its position. 
However we will assume that this position dependent variation is \emph{small}, in the sense that even though an individual's birth and death rate is of order $N$ 
(as in Section \ref{sec:basicmodel}), the spatial variation in these rates is of order $1$. We now make the model precise.

For each $i,j \in Q$, let $b^s_{ij}, d^s_i$ be bounded continuous functions from $E \times \R^q_{+}$ to $\R_+$. These functions determine the \emph{spatial} variation in the birth and death rates.
The migration of individuals is like in the basic model. However now the birth and death mechanism changes as follows.
\begin{itemize}
 \item When the population density vector is $h \in \R^q_+$, an individual of type $i \in Q$ located at $x \in E$, gives birth to an individual of type $j \in Q$ at rate 
$b^s_{ij}(x,h)+ N \beta_{ij}(h)$. At the time of birth, the offspring is placed at the same location as its parent. 
\item When the population density vector is $h \in \R^q_+$, an individual of type $i \in Q$ located at $x \in E$, 
dies at rate $d^s_{i}(x,h) + N \rho_i(h)$.
\end{itemize}
The evolution of our population under this dynamics can be viewed as a $\mathcal{M}^q_{N,a}(E)$-valued Markov process with generator 
$\mathbf{A}^N_1 : \mathcal{D}(\mathbf{A}^N_1) = \mathcal{C}_0^q \to B(\mathcal{M}^q_{N,a}(E))$ defined for any $\hat{F}  \in \mathcal{C}_0^q $ by
\begin{align}
\label{defn:an1}
\mathbf{A}^N_1 \hat{F} (\mu) &= \mathbf{B}^N \hat{F} (\mu) + 
 \sum_{i, j  \in Q} N \int_{E } \left( b^{s}_{i j}(x,h) +  N \beta_{ij}(h) \right)\left( \hat{F} \left( \mu + \frac{1}{N}\delta^j_{x}  \right) - \hat{F} (\mu)\right) 
  \mu_i(dx)  \\
&  + \sum_{i \in Q} N \int_{E }\left( d^{s}_{i}(x,h) + N \rho_i(h) \right)\left( \hat{F} \left( \mu - \frac{1}{N}\delta^i_{x}  \right) - \hat{F} (\mu)\right)   \mu_i(dx), \notag
\end{align} 
where $h = H(\mu)$ is the density vector corresponding to $\mu$.

\subsubsection{Model with offspring dispersal at birth}\label{sec:dispersalmodel}

In the basic model we described in Section \ref{sec:basicmodel}, the offspring is placed at the same location as its parent at the time of birth. 
However we may want to construct models where this restriction needs to be relaxed. 
For example, while modeling plant populations, one may wish to account for the spreading of seeds due to wind and other factors. Also in models for 
population genetics, where $E$ is the space of genetic traits, offsprings may be born with a different trait than their parents due to mutations. 
To consider such situations we now present a model in which an offspring may be born away from its parent. 
We allow this offspring dispersal to either be \emph{rare} (happens with probability proportional to $1/N$) or \emph{small} 
(the offspring is placed at a distance proportional to $1/N$ from the parent). We handle both these cases in a unified way.

For each $i,j \in Q$ and $N \in \N$, let $\vartheta^N_{ij}$ be a function from $E$ to $\mathcal{P}(E)$ and let $p^N_{ij}$ be a function from $E$ to $[0,1]$. 
The individuals migrate and die in the same way as described in the basic model (Section \ref{sec:basicmodel}). The birth rates are also the same as in the basic model. 
However when an individual of any type $i \in Q$ located at $x \in E$, gives birth to an individual of type $j$, the location of the offspring is $x$ with 
probability $(1-p^N_{ij}(x))$ and distributed according to $\vartheta^N_{ij}(x,\cdot)$ with probability $p^N_{ij}(x)$. 

To pass to the limit $N \to \infty$, we need an assumption on $p^N_{ij}$ and $\vartheta^N_{ij}$ which is stated below.
\begin{assumption}
\label{assmp:limitrelationshipondispersal}
For each $i,j \in Q$ we assume that there is an operator 
$C_{ij} : \mathcal{D}(C_{ij}) \to C(E)$, whose domain is taken to be the same as $\mathcal{D}_0$ (see \eqref{defn:classD0}) for convenience, such that for every 
$f\in  \mathcal{D}_0$.
\begin{align*}
\lim_{N \to \infty} \sup_{x \in E} \left|  N  p^N_{ij}(x) \int_{E} \left( f(y) - f(x) \right) \vartheta^N_{ij}(x,dy) - C_{ij}f(x)      \right| = 0.  
\end{align*}
\end{assumption}

The evolution of our population under the dynamics described above can be viewed as a $\mathcal{M}^q_{N,a}(E)$-valued Markov process with generator 
$\mathbf{A}^N_2 : \mathcal{D}(\mathbf{A}^N_2) = \mathcal{C}_0^q \to B(\mathcal{M}^q_{N,a}(E))$ defined by its action on any $\hat{F}  \in \mathcal{C}_0^q $ by
\begin{align}
\label{defn:an2}
\mathbf{A}^N_2 \hat{F} (\mu) &= \mathbf{B}^N \hat{F} ( \mu ) +  N^2 \sum_{i \in Q } \int_{E } \rho_{i}(h) \left( \hat{F}  \left( \mu -\frac{1}{N} \delta^i_{x} \right) - \hat{F} (\mu) \right)\mu_i(dx)\\
& +  N^2 \sum_{i,j  \in Q}  \int_{E } \beta_{ij}(h)  \left( 1-  p^N_{ij}(x) \right) \left( \hat{F}\left( \mu +\frac{1}{N} \delta^j_{x} \right) - \hat{F} (\mu) \right)\mu_i(dx) \notag \\
& +  N^2 \sum_{i,j \in Q }  \int_{E } \beta_{ij}(h) p^N_{ij}(x) \left[ \int_{E} \left( \hat{F} \left( \mu +\frac{1}{N} \delta^j_{y} \right) - \hat{F} (\mu) \right) \vartheta^N_{ij}(x,dy) \right]\mu_i(dx) ,\notag
\end{align} 
where $h = H(\mu)$ is the density vector corresponding to $\mu$. 

\subsubsection{Model with immigration}\label{sec:immmigrationmodel}

Consider a population whose dynamics is as described in the basic model (Section \ref{sec:basicmodel}). In addition, suppose that the individuals of each type are 
immigrating to $E$ at a certain density dependent rate and settling down according to some distribution on $E$. In this section we model this situation.
Such a model can help us understand the effects of immigration on the population demography.

For each $i \in Q$, let $\kappa_i : \R^q_+ \to \R_+$ be a continuous function satisfying     
\begin{align}
\label{lipchitzboundonkappa}
\kappa_i(h) \leq C ( 1+\left\| h\right\|_1) \textrm{ for all } h \in \R^q_{+}, 
\end{align}  
for some $C>0$. The individuals migrate, reproduce and die as in the basic model. 
Moreover, when the population density vector is $h \in \R^q_+$, the individuals of each type $i \in Q$ arrive in the population at rate $N \kappa_i(h)$ and their 
initial location is given by the distribution $\Theta_i \in \mathcal{P}(E)$.

The evolution of our population under this dynamics can be viewed as a $\mathcal{M}^q_{N,a}(E)$-valued Markov process with generator 
$\mathbf{A}^N_3 : \mathcal{D}(\mathbf{A}^N_3) = \mathcal{C}_0^q \to B(\mathcal{M}^q_{N,a}(E))$ defined by its action on any $\hat{F}  \in \mathcal{C}_0^q $ as
\begin{align}
\label{defn:an3}
\mathbf{A}^N_3 \hat{F} (\mu) &= \mathbf{B}^N \hat{F} ( \mu ) + N \sum_{i \in Q }  \kappa_i(h) \int_{E} \left( \hat{F}  \left( \mu +\frac{1}{N} \delta^i_{x} \right) - \hat{F} (\mu)  \right) \Theta_i(dx)     \\
& +  N^2 \sum_{i , j  \in Q}  \int_{E } \beta_{ij}(h)   \left( \hat{F}  \left( \mu +\frac{1}{N} \delta^j_{x} \right) - \hat{F} (\mu) \right)\mu_i(dx) \notag \\
& +  N^2 \sum_{i \in Q} \int_{E} \rho_{i}(h) \left( \hat{F}  \left( \mu -\frac{1}{N} \delta^i_{x} \right) - \hat{F} (\mu) \right)\mu_i(dx), \notag
\end{align} 
where $h = H(\mu)$ is the density vector corresponding to $\mu$.

\begin{remark}
\label{rem:splittingofoperators}
{\rm 
For each $l \in \{ 0 ,1,2,3 \}$ define the operators $\mathbf{G}^N_l : \mathcal{D}(\mathbf{G}^N_l) = B(\mathcal{M}^q_F(E)) \to B(\mathcal{M}^q_F(E)) $ as follows. 
For any $F \in B\left(\mathcal{M}^q_F(E)\right)$ let
 \begin{align}
 \mathbf{G}^N_0 F(\mu) =& 0, \label{defgn0}  \\
\mathbf{G}^N_1 F(\mu) =&  N \left[ \sum_{i, j \in Q } \int_{E}  b^{s}_{i j}(x,h)\left( F\left( \mu + \frac{1}{N}\delta^j_{x}  \right) - F(\mu)\right)  \mu_i(dx) 
 \right.  \label{defgn1} \\& \left. + \sum_{i \in Q } \int_{E } d^{s}_{i}(x,h) \left( F\left( \mu - \frac{1}{N}\delta^i_{x}  \right) - F(\mu)\right)   \mu_i(dx)\right], \notag
 \end{align}
 \begin{align}
 \mathbf{G}^N_2 F(\mu)  =&  N^2 \sum_{i, j \in Q } \beta_{ij}(h)  \int_{E }  \left[ \int_{E} \left( F \left( \mu +\frac{1}{N} \delta^j_{y} \right) - F \left( \mu +\frac{1}{N} \delta^j_{x} \right) \right) \vartheta^N_{ij}(x,dy) \right]  \notag \\
 & \times p^N_{ij}(x) \mu_i(dx) \label{defgn2}\\
 \textrm{ and } \mathbf{G}^N_3 F(\mu)  =&   N \sum_{i \in Q}  \kappa_i(h) \int_{E} \left( F \left( \mu +\frac{1}{N} \delta^i_{x} \right) - F(\mu)  \right) \Theta_i(dx), \label{defgn3}
\end{align}
where $h = H(\mu)$. Then for each $l \in \{ 0 ,1,2,3 \}$ and $\hat{F}  \in \mathcal{C}^q_0$ we can write
\begin{align}
\label{splittingofanl}
\mathbf{A}^N_l \hat{F}  =  \mathbf{B}^N  \hat{F}  + N \mathbf{R}^N \hat{F}  + \mathbf{G}^N_l \hat{F} . 
\end{align} 
This form makes the timescale separation clear between the \emph{fast} density regulation mechanism ($N \mathbf{R}^N$) and the \emph{slow} migration ($\mathbf{B}^N$),
position-dependent birth and death ($\mathbf{G}^N_1$), offspring dispersal ($\mathbf{G}^N_2$) and immigration ($\mathbf{G}^N_3$) mechanisms. }
\hfill $\square$
\end{remark}

\subsection{The main results}
\label{mainresultsection}

The main results of this paper are concerned with the limiting behaviour of the dynamics under the models described in Section \ref{mathmodels}. 
Before we present these results we must first verify that all the models in Section \ref{mathmodels} can be represented by a suitable Markov process. This is established by the following proposition which will be proved in Section \ref{sec:wellposedness}.
\begin{proposition}
\label{prop:wellposedness}
For each $l \in \{0,1,2,3\}$ and $N \in \N$, the $D_{\mathcal{M}^q_{N,a} (E)}[0,\infty)$ martingale problem for $\mathbf{A}^N_l$ is well-posed. 
\end{proposition}

We now begin analyzing how a sequence of Markov processes with generators $\mathbf{A}^N_l$ behave as $N \to \infty$. The next proposition exhibits some important properties about the limiting dynamics. The proof of this proposition is given in Section \ref{propertiesoflimitingprocesses}. Recall the definition of $U_{\textnormal{eq}}$ from \eqref{defueq}.
\begin{proposition}
\label{prop:constancyofh_mainresult}
Fix a $l \in \{0,1,2,3\}$. For each $N \in \N$, let $\{ \mu^N(t) : t \geq 0\}$ be a solution of the martingale problem for $\mathbf{A}^N_l$ and 
let $\{ h^N(t) = H(\mu^N(t)) : t \geq 0 \}$ be the corresponding density process. 
Assume that there is a compact set $K_0 \subset U_{\textnormal{eq}}$ such that $h^N(0) \in K_0$ a.s.\ for all $N \in \N$. Let $t_N$  be a sequence of positive numbers satisfying $t_N \to 0$ and $N t_N \to \infty$ as $N \to \infty$. Then we have:
\begin{itemize}
 \item[(A)] For all $T>0$ 
\begin{align*}
\sup_{t \in [0,T] } \left\| h^N(t+t_N) - h_{\textnormal{eq}}\right\|_1 \Rightarrow 0 \textrm{ as } N \to \infty.
\end{align*}
\item[(B)] For all $T>0$, $f \in C(E)$ and $i,j \in Q$
\begin{align*}
\sup_{t \in [0,T] } \left| h^N_j(t+t_N) \langle f,\mu^N_i(t+t_N) \rangle -  h^N_i(t+t_N)\langle f,\mu^N_j(t+t_N) \rangle \right| \Rightarrow 0 \textrm{ as } N \to \infty.
\end{align*}  
\end{itemize}
\end{proposition}
Let the processes $\{ \mu^N(t) : t \geq 0\}$ and $\{ h^N(t) :  t \geq 0 \}$ be as in the above proposition. Part (A) of this proposition implies that the process $h^N(\cdot + t_N) \Rightarrow h_{\textnormal{eq}}$ as $N \to \infty$ in $D_{\R^q}[0,\infty)$. In other words, for large $N$, the density process is constantly near the equilibrium population density $h_{\textnormal{eq}}$ (after a small time shift $t_N$). This emphasizes the point that we made in Section \ref{introduction}. The density regulation mechanism operating at a faster timescale than our timescale of observation, keeps the population density equilibrated at all times. 
Note that the process $\{ \mu^N(t) = (\mu^N_1(t),\dots,\mu^N_q(t)) : t \geq 0 \}$ is $\mathcal{M}^q_F(E)$-valued and it keeps track of how the populations corresponding to all the $q$-types are evolving in the space $E$. Part (B) of the above proposition shows that in the limit, all the $q$ sub-populations are spatially \emph{fused} (in proportions determined by the density vector). 
Hence their spatial evolution can be studied together by using a single $\mathcal{P}(E)$-valued process. 
This kind of model reduction result is quite common in stochastic reaction networks with multiple timescales (see \cite{Ball} and \cite{HWKang}), where one can often equilibrate the concentrations of the \emph{fast} chemical species and derive a reduced model for the dynamics of the \emph{slow} species.
In our case, the dynamics at the fast timescale equilibrates the population density as well as the relative abundances of all the $q$ sub-populations at each location on $E$. These relative abundances equilibrate because the birth-death interaction matrix $A(h)$ (see \eqref{defninteractionmatrix}) is irreducible at the equilibrium density $h_{\textnormal{eq}}$ (see part (C) of Assumption \ref{mainassumptionsonr}).
 
The proof of Proposition \ref{prop:constancyofh_mainresult} will exploit the form \eqref{splittingofanl} of the operator $\mathbf{A}^N_l$. A brief outline of the proof is as follows.
We will define a $\R^q_+ \times \R^{q-1}$ valued process $\{ X^N(t) : t \geq 0 \}$ by
\begin{align}
\label{defn_xn}
X^N(t) = (h^N_1(t),\dots,h^N_q(t), Y^N_1(t),\dots,Y^N_{q-1}(t)),
\end{align}  
where each $Y^N_i(t)$ is a density-dependent linear combination of terms like $(h^N_j(t) \langle f,\mu^N_i(t) \rangle$ $-  h^N_i(t)\langle f,\mu^N_j(t) \rangle)$ for some choice of $f \in C(E)$. Next we will show that $X^N$ is a semimartingale which satisfies an equation of the form
 \begin{align}
\label{prop:defxn_mainresult}
X^N(t) = X^N(0) + N \int_{0}^{t} F(X^N(s))ds + Z^N(t), 
\end{align} 
where $\{Z^N : N \in \N\}$ is a sequence of $\R^{2q-1}$-semimartingales which is tight in the space $D_{\R^{2q-1}} [0,\infty)$. This clearly indicates that for large values of $N$, the drift term of the form $N F(X^N(\cdot))$ completely overwhelms the effect of the semimartingale $Z^N$. Equations like \eqref{prop:defxn_mainresult} were studied by Katzenberger in \cite{kat} in a much more general setting. He showed that under certain conditions, the sequence of semimartingales $\{ X^N : N \in \N \}$ converges in distribution to a semimartingale $X$ as $N \to \infty$. Moreover $X$ only takes values in a set $\Gamma$ which is an invariant manifold for the deterministic flow induced by $F$. In our case, this set $\Gamma$ only consists of one point $x_{\textnormal{eq}} = (h_{\textnormal{eq}}, \bar{0}_{q-1})$ and this enables us to prove Proposition \ref{prop:constancyofh_mainresult}. The details are given in Section \ref{propertiesoflimitingprocesses}. 

We mentioned before that in the limit $N \to \infty$, the spatial evolution of all the $q$ sub-populations is governed by a single $\mathcal{P}(E)$-valued process. Our next result, Theorem \ref{mainresult}, shows that this $\mathcal{P}(E)$-valued process is in fact a Fleming-Viot process that can be characterized by its generator. Before we state Theorem \ref{mainresult} we first need to introduce several objects. The existence and properties of some of these objects will be studied in the appendix.

Recall the equilibrium population density vector $h_{\textnormal{eq}} = (h_{\textnormal{eq},1},\dots,h_{\textnormal{eq},q})$ from Section \ref{density_regulation_mechanism}. It can be verified that this vector has strictly positive components (see part (A) of Lemma \ref{lemma:appendix1}). Moreover part (C) of Lemma \ref{lemma:appendix1} shows that there is a unique vector $v_{\textnormal{eq}} = (v_{\textnormal{eq},1},\dots,v_{\textnormal{eq},q}) \in \R^q_{*}$ such that 
\begin{align}
\label{defb:veq}
v_{\textnormal{eq}} A(h_{\textnormal{eq}}) = \bar{0}_q \textrm{ and }\langle v_{\textnormal{eq}}, h_{\textnormal{eq}} \rangle = \sum_{i=1}^q v_{\textnormal{eq},i} h_{\textnormal{eq},i} = 1. 
\end{align}
Observe that $\mathcal{D}_0 \subset C(E)$ satisfies \eqref{defn:classD0}. Define an operator $B_{\textnormal{avg}} : \mathcal{D}(B_{\textnormal{avg}}) = \mathcal{D}_0 \to C(E)$ by
\begin{align}
B_{\textnormal{avg}} f = \sum_{i \in Q}  v_{\textnormal{eq}, i } h_{\textnormal{eq}, i} B_i f \ \textrm{ for } \ f \in \mathcal{D}_0.
\end{align}
From \eqref{defb:veq} we can see that the operator $B_{\textnormal{avg}}$ is a convex combination of the operators $\{B_i : i \in Q\}$. Let 
$\gamma_{\textrm{smpl}}$ be the positive constant given by 
\begin{align}
\label{defnresamplingrate}
\gamma_{\textrm{smpl}} =  \sum_{i \in Q } v^2_{\textnormal{eq},i} h_{\textnormal{eq},i} \rho_i(h_{\textnormal{eq}}).
\end{align}
For each $i,j \in Q$, let the functions $b^{s}_{ij}, d^{s}_i$ be as in Section \ref{sec:positiondependentmodel}. Define $b^{s}_{\textnormal{avg}}, d^s_{\textnormal{avg}} \in C(E)$ as
\begin{align*}
b^{s}_{\textnormal{avg}}(x) = \sum_{i,j \in Q} b^{s}_{ij}(x,h_{\textnormal{eq}}) v_{\textnormal{eq},j} h_{\textnormal{eq},i} \ \textrm{ and }  \
d^{s}_{\textnormal{avg}}(x) = \sum_{i \in Q} d^{s}_{i}(x,h_{\textnormal{eq}}) v_{\textnormal{eq},i} h_{\textnormal{eq},i} \textrm{ for } x \in E. 
\end{align*}
For each $i,j \in Q$, let the operator $C_{ij}$ be as in Assumption \ref{assmp:limitrelationshipondispersal}. Define the operator 
$C_{\textnormal{avg}} : \mathcal{D}(C_{\textnormal{avg}}) = \mathcal{D}_0 \to C(E)$ by
\begin{align}
\label{defn:cavg}
C_{\textnormal{avg}} f = \sum_{i,j \in Q} \beta_{ij}(h_{\textnormal{eq}}) v_{\textnormal{eq},j} h_{\textnormal{eq}, i} C_{ij}f \textrm{ for } f \in \mathcal{D}_0.
\end{align}
For each $i \in Q$, let $\kappa_i, \Theta_i$ be as in Section \ref{sec:immmigrationmodel}. Define 
the operator $I_{\textnormal{avg}} : \mathcal{D}(I_{\textnormal{avg}}) = B(\mathcal{P}(E)) \to B(\mathcal{P}(E))$ by 
\begin{align}
\label{defn:iavg} 
I_{\textnormal{avg}} f(x) = \sum_{i \in Q} \kappa_i(h_{\textnormal{eq}}) v_{\textnormal{eq},i} \int_{E} (f(y)-f(x)) \Theta_i(dy) \textrm{ for } f \in B(\mathcal{P}(E)).
\end{align}

We now define the operators $\mathbf{A}_0,\mathbf{A}_1$,$\mathbf{A}_2$ and $\mathbf{A}_3$ with domain $\mathcal{C}_0$ (see \eqref{defc0}) as below. 
For any $F(\nu) = \prod_{l = 1}^{m} \langle f_l , \nu \rangle \in \mathcal{C}_0$ let  
\begin{align}
\label{main:gena0}
\mathbf{A}_0 F(\nu) &= \sum_{l = 1}^{m} \langle B_{\textnormal{avg}} f_l,\nu \rangle \prod_{j \neq l } \langle f_j , \nu \rangle 
\\& + \gamma_{\textrm{smpl}}\sum_{1 \leq l \neq k \leq m} \left( \left\langle f_{l}f_{k},\nu \right\rangle - \left\langle f_{l},\nu \right\rangle \left\langle f_{k},\nu \right\rangle
 \right)  \prod_{j \neq l,k} \langle f_j,\nu\rangle, \notag \\ 
\mathbf{A}_1 F(\nu) &= \mathbf{A}_0 F(\nu) + \sum_{l = 1}^{m} \left(  \left\langle b^{s}_{\textnormal{avg}} f_l , \nu \right\rangle - \left\langle b^{s}_{\textnormal{avg}}  , \nu \right\rangle \left\langle  f_l ,\nu \right\rangle  \right)
 \prod_{j \neq l} \langle f_j , \nu \rangle  \label{main:gena1} \\ 
& + \sum_{l = 1}^{m} \left( \left\langle d^{s}_{\textnormal{avg}}  , \nu \right\rangle \left\langle  f_l ,\nu \right\rangle   -  \left\langle d^{s}_{\textnormal{avg}} f_l , \nu \right\rangle   \right)
 \prod_{j \neq l} \langle f_j , \nu \rangle,\notag\\
\mathbf{A}_2 F(\nu) &= \mathbf{A}_0 F(\nu) + \sum_{l = 1}^{m} \langle  C_{\textnormal{avg}} f_l , \nu \rangle  \prod_{j \neq l} \langle f_j , \nu \rangle  \label{main:gena2} \\
\textrm{ and } \ \mathbf{A}_3 F(\nu) &= \mathbf{A}_0 F(\nu) + \sum_{l = 1}^{m} \langle  I_{\textnormal{avg}} f_l , \nu \rangle  \prod_{j \neq l} \langle f_j , \nu \rangle \label{main:gena3}.
\end{align}
We will assume that the operators $B_{\textnormal{avg}}$, $(B_{\textnormal{avg}}+C_{\textnormal{avg}})$ and $(B_{\textnormal{avg}}+I_{\textnormal{avg}})$ generate Feller semigroups on $C(E)$.
The well-posedness of the martingale problems corresponding to $\mathbf{A}_0,\mathbf{A}_1$,$\mathbf{A}_2$ and $\mathbf{A}_3$ follows from Theorem 3.2 in \cite{EK93}.
In fact, any solution will have sample paths in $C_{\mathcal{P}(E)}[0,\infty)$. The operator $\mathbf{A}_0$ is the generator of a \emph{neutral} Fleming-Viot process on $E$ with 
\emph{mutation} operator $B_{\textnormal{avg}}$ and \emph{sampling} rate $2 \gamma_{\textrm{smpl}}$. The operators $\mathbf{A}_2$ and $\mathbf{A}_3$ generate a similar 
Fleming-Viot process with the \emph{mutation} operator changed to $(B_{\textnormal{avg}}+C_{\textnormal{avg}})$ and $(B_{\textnormal{avg}}+I_{\textnormal{avg}})$ respectively. 
The operator $\mathbf{A}_1$ also generates a similar Fleming-Viot process, but with \emph{selection}. The last two terms in its definition correspond to 
\emph{fecundity selection} (with intensity function $b^s_{\textnormal{avg}}$) and \emph{viability selection} (with intensity function $d^{s}_{\textnormal{avg}}$). 
See Donnelly and Kurtz \cite{DKgen} for more details. We now formally state the main result of our paper. The proof is given in Section \ref{section:fvconvergence}.
\begin{theorem}
\label{mainresult}
Fix a $l \in \{0,1,2,3\}$ and let $\{ \mu^N(t) : t \geq 0 \}$ be a solution to the martingale problem for $\mathbf{A}^N_l$. Suppose that $\mu^N(0) \Rightarrow \mu(0)$ as $N \to \infty$ and 
$H(\mu(0)) \in U_{\textnormal{eq}}$ a.s.\, where $U_{\textnormal{eq}}$ is given by \eqref{defueq}. Let $t_N$ be a sequence as in Proposition \ref{prop:constancyofh_mainresult}. Define another process $\{\hat{\mu}^N(t) : t \geq 0\}$ by
\begin{align}
\label{defnhatmun}
 \hat{\mu}^N(t) = \mu^N(t+t_N) \textrm{ for } t \geq 0.
\end{align}
 Then there exists a distribution $\pi \in \mathcal{P}\left( \mathcal{P}(E) \right)$ such that $\hat{\mu}^N \Rightarrow h_{\textnormal{eq}} \nu$ in $D_{\mathcal{M}^q_F(E)}[0,\infty)$ 
as $N \to \infty$ and $\{ \nu(t) : t \geq 0 \}$ is a Fleming-Viot process with type space $E$, generator $\mathbf{A}_l$ and initial distribution $\pi$.
\end{theorem}
\begin{remark}
The initial distribution $\pi$ of the process $\{ \nu(t) : t \geq 0 \}$ is related to the distribution of $\mu(0)$. This relation is stated in Remark \ref{remark:initialdistribution}. 
\end{remark}
\begin{remark}
\label{rem:morecomplicatedmodels}
In Section \ref{mathmodels} we first defined a basic model and then constructed auxiliary models by adding other mechanisms, one at a time. These other mechanisms are 
position dependent birth and death, offspring dispersal and immigration. One can consider models in which more than one of these mechanisms are simultaneously 
added to the basic model. The proof will demonstrate that the generator of the limiting Fleming-Viot process is then obtained by adding the correct term 
corresponding to each of these additional mechanisms to the operator $\mathbf{A}_0$. This correct term can be seen from the definitions of $\mathbf{A}_1$, $\mathbf{A}_2$ and $\mathbf{A}_3$.
For example, one can have the basic model along with 
position dependent birth and death (Section \ref{sec:positiondependentmodel}) and  offspring dispersal (Section \ref{sec:dispersalmodel}). 
Then the limiting Fleming-Viot process has the generator given by 
\begin{align*}   
\mathbf{A} F(\nu) &= \mathbf{A}_0 F(\nu) + \sum_{l = 1}^{m} \left(  \left\langle b^{s}_{\textnormal{avg}} f_l , \nu \right\rangle - \left\langle b^{s}_{\textnormal{avg}}  , \nu \right\rangle \left\langle  f_l ,\nu \right\rangle  \right)
 \prod_{j \neq l} \langle f_j , \nu \rangle  \\ 
& + \sum_{l = 1}^{m} \left( \left\langle d^{s}_{\textnormal{avg}}  , \nu \right\rangle \left\langle  f_l ,\nu \right\rangle   -  \left\langle d^{s}_{\textnormal{avg}} f_l , \nu \right\rangle   \right)
 \prod_{j \neq l} \langle f_j , \nu \rangle + \sum_{l = 1}^{m} \langle  C_{\textnormal{avg}} f_l , \nu \rangle  \prod_{j \neq l} \langle f_j , \nu \rangle,
\end{align*}
for any $F(\nu) = \prod_{l = 1}^{m} \langle f_l , \nu \rangle \in \mathcal{C}_0$.
\end{remark}

We now give a heuristic explanation of why the dynamics under the models described in Section \ref{mathmodels} converges to a Fleming-Viot process. Note that part (A) of Proposition \ref{prop:constancyofh_mainresult} says that the population density is \emph{pinned} to a constant value $h_{\textnormal{eq}}$ in the limit. Therefore any addition of new mass in the population must be concurrently offset by an equal reduction of existing mass and vice versa. Furthermore, when the mass is reduced or added to keep the balance, this reduction or addition happens at locations that are chosen more or less uniformly from the current empirical measure of the population. This is because the birth and deaths rates of individuals are dominated by a term which is density dependent but location independent. This argument offers some intuition as to why the fast birth-death terms (that form part of the operator $N \mathbf{R}^N$) give rise to the \emph{sampling} term in the limit (the second term in $\mathbf{A}_0$). It also shows why the position dependent birth and death terms in $\mathbf{A}^N_1$ become \emph{selection} terms in $\mathbf{A}_1$ and the offspring dispersal (immigration) term in $\mathbf{A}^N_2$ ($\mathbf{A}^N_3$) becomes a \emph{mutation} term in $\mathbf{A}_2$ ($\mathbf{A}_3$). Since the position of an individual in $E$ can also be seen as its genetic trait, one can interpret the migration on $E$ as genetic mutation. Hence it is not surprising that the migration operators appear as part of the mutation operator in the limiting process. Part (B) of Proposition \ref{prop:constancyofh_mainresult} says that in the limit, all the $q$ sub-populations become spatially inseparable. This causes all the mechanisms in the limiting process to appear in an \emph{averaged} form.

Let $\{ \mu^N(t) = (\mu^N_1(t),\dots, \mu^N_q(t)) : t \geq 0 \}$ be a Markov process with generator $\mathbf{A}^N_l$, for some $l \in \{0,1,2,3\}$. It is difficult to prove the convergence of this process directly because the density regulation mechanism acts on it at the fast timescale. This can be seen by splitting the operator $\mathbf{A}^N_l$ according to \eqref{splittingofanl} and noting that $N \mathbf{R}^N$ becomes unbounded as $N \to \infty$. To pass to the limit we consider another measure-valued process $\{ \nu^N(t) : t \geq 0 \}$ that is constructed by suitably combining the various components of  $\{ \mu^N(t) : t \geq 0 \}$. In particular
\begin{align}
\label{tempformofnuN}
 \nu^N(t) = \sum_{i =1}^q \Lambda_i(h^N(t)) \mu^N_i(t) \ \textnormal{ for } t \geq0,
\end{align}
where $h^N(t) = H( \mu^N(t) )$ and $\Lambda = (\Lambda_1,\dots,\Lambda_q)$ is a function from $\R^q_+$ to $\R^q_+$ which satisfies certain conditions. These conditions are chosen to ensure that  $\{ \nu^N(t) : t \geq 0 \}$ is a $\mathcal{P}(E)$-valued process whose dynamics is such that the density regulation mechanism acts at the slow timescale. Such a function $\Lambda$ can be shown to exist by proving that a certain system of coupled partial differential equations has a solution with some desired properties. This is done in Section \ref{section:pdesolution}. We will then show that as $N \to \infty$ we have $\nu^N \Rightarrow \nu$ where $\{\nu(t) : t \geq 0\}$ is the Fleming-Viot process specified by Theorem \ref{mainresult}. This convergence along with Proposition \ref{prop:constancyofh_mainresult} allow us to prove Theorem \ref{mainresult}.  The details of the proof are given in Section \ref{section:fvconvergence}. 

The discussion in the preceding paragraph also shows that intuitively we can think of the limiting Fleming-Viot process as describing the spatial evolution of a \emph{mixed} population formed by taking a suitable density-dependent linear combination of all the $q$ sub-populations. This is reminiscent of the notion of \emph{virtual} species (formed by linearly combining several chemical species), that are needed in the specification of the reduced models in chemical reaction networks with multiple timescales (see \cite{Petzold}). 

 \section{Applications} \label{applications}

In this section we discuss the applications mentioned in Section \ref{introduction} in greater detail. Note that a Fleming-Viot process usually has continuous paths (see \cite{EK93}). Hence Theorem \ref{mainresult} can be seen as a \emph{diffusion approximation} result which shows that a stochastic process with jumps can be approximated by a process with continuous paths. Such results provide a justification for drawing inferences about the original process (with jumps) by analyzing a more tractable process with continuous paths.

To demonstrate the usefulness of Theorem \ref{mainresult}  we present two examples. In the first example we consider a population genetics model having logistic interactions along with \emph{rare mutation} and \emph{weak selection}. The words \emph{rare} and \emph{weak} indicate that the mutation and selection events occur at a slower timescale than other events.
The difference between this model and a standard population genetics model (Wright-Fisher or Moran) is that the population size is not fixed but fluctuating due to the logistic interactions. Theorem \ref{mainresult} guarantees that by taking the infinite population limit in a suitable way, we obtain a Fleming-Viot process. In many cases this limiting process is well-studied and using its properties one can estimate fixation probabilities, fixation times and the stationary distribution for the finite population model. Our second example sheds light on the phenomenon of cell polarity which refers to the spatial crowding of molecules on the cell membrane. We draw upon our work in \cite{GUPTA} to show that Fleming-Viot  convergence can help us understand how cells establish and maintain polarity. In \cite{GUPTA} we only consider a very simple model, but the results in this paper ensure that the same analysis holds for a general class of models.

\subsection{Logistic model for population genetics} \label{sec:log1}

The logistic growth model is very popular in ecology. 
It was proposed by Verhulst \cite{Verhulst} in 1838 to describe the growth of a population in the presence of competition for resources. 
In this model each individual reproduces at rate $\beta$ and dies at a rate $\rho P /N$, where $P$ is the current population size and $N$ is the 
\emph{carrying capacity} of the habitat. In the deterministic setting, the population size $(P)$ evolves 
as a function of time $(t)$ according to the ordinary differential equation
\[\frac{d P}{d t } = \beta P - \rho \frac{P^2}{N}.\]  
Let $h(t) = P(t)/N$ be the \emph{population density} at time $t$. Then the above differential equation becomes
\begin{align}
\label{logisticode} 
\frac{d h}{dt} = \beta h - \rho h^2.
\end{align}
It is immediate that if $h(0)>0$ then 
$h(t) \to h_{\textnormal{eq}} := \beta/\rho$ as $t \to \infty$.

We now construct a population genetics model that has logistic interactions along with \emph{rare mutation} and \emph{weak selection}. 
Suppose the compact metric space $E$ is the set of all the genetic traits that an individual can have. 
Each individual is given a mass of $1/N$, with $N$ being the carrying capacity as before. 
The population at time $t$ can be represented by the measure
\begin{align}
\label{log:measureN}
\bar{\mu}^N(t) = \frac{1}{N} \sum_{ i = 1 }^{n^N(t)} \delta_{x_i}, 
\end{align}
where $n^N(t)$ is the number of individuals at time $t$ and $x_1,x_2,\dots \in E$ are their genetic traits. Let $b^{s}$ be a continuous function from $E$ to $\R_+$. When the population density (total mass) is $h$, an individual with trait $x \in E$ gives birth at rate $(\beta + b^{s}(x)/N)$ and dies at rate $\rho h$. Its offspring has the same trait $x$ with probability $(1-p(x)/N)$. However with 
probability $p(x)/N$, the offspring is a \emph{mutant} and its trait is chosen according to the distribution $\vartheta(x,\cdot) \in \mathcal{P}(E)$. 
The process $\{ \bar{\mu}^N(t) :  t \geq 0 \}$ can be viewed as a Markov process with state space $\mathcal{M}_{N,a}(E)$ (see \eqref{defnmna}). 
The timescale at which we have described the dynamics is such that the mutation and selection events will vanish in the limit $N \to \infty$.
Therefore to study their effects, we must observe the process at the timescale which is $N$ times slower. Let 
\[ \mu^N(t) = \bar{\mu}^N(N t) \textrm{ for } t \geq 0.\]   
The dynamics of $\{ \mu^N(t) : t \geq 0\}$ has fast density regulation along with position-dependent birth (see Section \ref{sec:positiondependentmodel}) and offspring dispersal mechanism (see Section \ref{sec:dispersalmodel}).
Assuming that $\mu^N(0) \Rightarrow \mu(0)$ as $N \to \infty$ and $\langle 1_E , \mu(0) \rangle > 0 $ a.s.\ Theorem \ref{mainresult} gives us the following. If $t_N$ is a sequence satisfying $t_N \to 0$ and $N t_N \to \infty$, then the process $\mu^N(\cdot + t_N) \Rightarrow h_{\textnormal{eq}} \nu$ as $N \to \infty$, where $h_{\textnormal{eq}} = \beta/\rho$ and $\{ \nu(t) : t \geq 0 \}$ is a Fleming-Viot process with generator given by
\begin{align}
\label{log1:gena}
\mathbf{A} F(\nu) &= \beta \sum_{l = 1}^{m} \left[ \int_{E}  p(x) \left( \int_{E}( f_l(y)-f_l(x)) \vartheta(x,dy) \right)  \nu(dx) \right] \prod_{j \neq l } \langle f_j , \nu \rangle 
\\ & + \sum_{l = 1}^{m} \left( \langle b^s f_l, \nu \rangle - \langle b^s,\nu \rangle \langle f_l, \nu \rangle \right) \prod_{j \neq l } \langle f_j , \nu \rangle \notag \\
& +  \rho \sum_{1 \leq l \neq k \leq m} \left( \left\langle f_{l}f_{k},\nu \right\rangle - \left\langle f_{l},\nu \right\rangle \left\langle f_{k},\nu \right\rangle\right)  \prod_{j \neq l,k} \langle f_j,\nu\rangle \notag
\end{align}
for any $F(\nu) = \prod_{l = 1}^{m} \langle f_l , \nu \rangle$ where $f_1,\dots,f_m \in B(E)$. 
This is of course the generator of a Fleming-Viot process with mutation and fecundity selection. We now present a couple of cases where this process is well-studied.

Suppose that $E = \{1,\dots,K\}$ for some $K \in \N$. Then for any $t \geq 0$ we can express $\nu(t) \in \mathcal{P}(E)$ as the $K$-tuple $(\nu_1(t),\dots,\nu_K(t) )$, where $\nu_i(t)$ is the proportion of individuals having genetic trait $i \in E$. This representation allows us to view $\{\nu(t) : t \geq 0\}$ as a process over the $K$-simplex
\begin{align*}
\Delta_K = \left\{ (x_1,\dots,x_K) : \ x_i \geq 0 \textnormal{ and } \sum_{i =1}^K x_i = 1 \right\}.
\end{align*} 
For all $i,j \in E$ set $\theta_{ij} = \beta p(i) \vartheta(i, \{j\})$ and $\alpha_i = b^{s}(i)$. Then $\{\nu(t) : t \geq 0\}$ is a \emph{diffusion} process over $\Delta_K$ with generator given by
\begin{align*}
\mathbf{A} f(\nu) = \rho \sum_{i,j \in E} \nu_i(\delta_{ij} - \nu_j) \frac{ \partial^2 f (\nu) }{\partial \nu_i \partial \nu_j} + 
\sum_{j \in E} \left(  \sum_{i \in E}  \left( \theta_{ij} \nu_i  + \nu_j \alpha_i  (\delta_{ij}  - \nu_i ) \right)  \right)  \frac{ \partial f (\nu) }{\partial \nu_j} 
\end{align*}
where $f \in C^2(\R^K_+, \R)$, $\nu = (\nu_1,\dots,\nu_K) $ and $\delta_{ij}$ is the Kronecker delta function. This is the generator of the \emph{Wright-Fisher} diffusion process \cite{kimura}. Many explicit results about the fixation probabilities, fixation times and the stationary distribution can be found in \cite{EwensBook}. 

Let us return to the situation where $E$ is a general compact metric space and the dynamics evolves according to \eqref{log1:gena}. Assume that for all $x \in E$ we have $b^{s}(x) = 0$, $p(x) =1$ and $\vartheta(x,\cdot) = \vartheta_0(\cdot)$, for some non-atomic probability measure $\vartheta_0 \in \mathcal{P}(E)$. The resulting Fleming-Viot process $\{\nu(t) : t \geq 0\}$ arises as a reformulation of the infinitely-many-neutral-alleles model due to Kimura and Crow \cite{KC64} (see \cite{EK} and Section 9.2 in \cite{EK93} for more details). In this case, $\nu(t)$ can be written as a countable sum $\sum_{ i = 1}^{\infty} a_i \delta_{x_i}$ for any $t > 0$ (see Theorem 7.2 in \cite{EK93}). This means that at time $t$, $a_i$ fraction of the population is located at $x_i$. Arranging these $a_i$-s in descending order we can extract a process over the ordered infinite simplex
\begin{align*}
\hat{\Delta}_\infty = \left\{ (x_1,x_2,\dots) : \ x_1 \geq x_2 \dots \geq 0 \textnormal{ and } \sum_{i =1}^\infty x_i = 1 \right\}.
\end{align*} 
This extracted process is a diffusion process over $\hat{\Delta}_\infty$ whose various properties are presented in \cite{EK81}. Furthermore in \cite{EK93} it is shown that the Fleming-Viot process $\{\nu(t) : t \geq 0\}$ is ergodic and its unique stationary 
distribution $\Pi \in \mathcal{P}(\mathcal{P}(E))$ is given by
\begin{align}
\label{stdisexplicit}
 \Pi(S) = \P \left( \sum_{i=1}^{\infty} \phi_i \delta_{\xi_i}  \in S \right) \textrm{ for all } S \in \mathcal{B}(\mathcal{P}(E)),
\end{align}
where the infinite vector $(\phi_1,\phi_2,\dots)$ has the Poisson-Dirichlet distribution with parameter $\beta/2\rho$ and $\xi_1,\xi_2,\dots$ are i.i.d.\ with distribution $\vartheta_0$, independent of $(\phi_1,\phi_2,\dots)$. The Poisson-Dirichlet distribution was introduced and studied by Kingman \cite{King75} in 1975.

The results mentioned in the last two paragraphs indicate the behaviour of the \emph{evolutionary} dynamics under our original model for large values of $N$.

\subsection{Cell polarity}\label{sec:cellpolarity}

Cell polarity is an important phenomenon and understanding the mechanisms responsible for it is a matter of fundamental concern for biologists. 
It is widely accepted that polarity is established in $3$ stages (see \cite{DN,AAWW,AAWWref7}), which can be described as follows:
\begin{enumerate}
\item An unpolarized cell receives a spatial cue that may be \emph{intrinsic} (coming from inside the cell) or \emph{extrinsic} (coming from the extracellular environment). 
\item This cue is interpreted by the membrane-bound receptor molecules. 
\item The feedback network inside the cell is activated, which amplifies the weak initial signal into a robust signal that can direct the molecules towards the clustering site. 
\end{enumerate}
The feedback network has two components :  
\emph{positive feedback} which enables the membrane molecules to pull the cytosol molecules to their location on the membrane, and 
\emph{negative feedback} that pushes the membrane molecules into the cytosol. Positive feedback is responsible for the localized recruitment of molecules on the membrane while negative feedback helps in regulating the population size on the membrane. The molecules diffuse slowly on the membrane but rapidly in the cytosol.

Even though the feedback mechanism may bring the molecules together on the membrane, any clusters that form may not persist due to spatial diffusion. This caused some biologists to propose that other additional mechanisms are needed to generate spatial asymmetry (see \cite{AAWWref3, AAWWref14}), but these mechanisms are not always found in cells that exhibit polarity. Hence it is important to investigate if the feedback mechanism can alone counter spatial diffusion to establish cell polarity. For this purpose, Altschuler et. al. \cite{AAWW} formulated a simple model based on the mechanisms mentioned above. We now describe their model. Consider the cell to be a sphere of radius $R$ in $\R^3$. The whole cell has $N$ molecules which may be present on the membrane or in the cytosol. The following four mechanisms change the configuration of molecules in the cell.
\begin{enumerate}
\item \emph{Association mechanism:} Each molecule in the cytosol can move to a uniformly chosen location on the membrane at rate $k_{\textnormal{on}}$. 
\item \emph{Positive feedback:} Each molecule on the membrane pulls another molecule from the cytosol to its location at rate 
$k_{\textnormal{fb}} \times \textrm{(fraction of molecules in the cytosol)}$. 
\item \emph{Negative feedback:} Each molecule on the membrane is pushed into the cytosol at rate $k_{\textnormal{off}}$. 
\item \emph{Spatial migration:} Each membrane molecule is constantly diffusing on the membrane according to an independent Brownian motion with diffusion rate $D$. 
\end{enumerate}
The association mechanism provides the initial spatial cue to trigger cluster formation. In \cite{AAWW}, this spatial cue is intrinsic because the authors are concerned with \emph{spontaneous} cell polarity, which means that polarity is established without any extracellular influence.
Hence the association mechanism acts uniformly on the membrane. When one wants to consider polarity that is established in response to a chemical gradient (see \cite{Weiner}) then a molecule \emph{associating} itself to the membrane must choose its location according to some distribution that encodes the gradient information. We mentioned in Section \ref{introduction}, that the positive feedback mechanism is like a birth process, where the pulled cytosol molecule is the \emph{offspring} of the recruiting membrane molecule. This introduces genealogical relationships between the membrane molecules. A set of membrane molecules are said to belong to a \emph{clan} if they have a common ancestor. Note that when the diffusion rate ($D$) is small, we would expect the clan members to be huddled together. 

The analysis of the above model in \cite{AAWW} gives some interesting results. When the dynamics is described deterministically, using a reaction-diffusion partial differential equation, then the model fails to capture cell polarity. However in the stochastic setting, the model does predict the formation of clusters in certain parameter regimes, when the number of molecules ($N$) is small. This result is proved by showing that the number of clans on the membrane drops to $1$ at certain times. For small $D$, one would observe a cluster at these times. However the frequency of these events is proportional to $N^{-1}$, which indicates that polarity cannot occur in the large population limit $N \to \infty$, unless other mechanisms are present. 

In \cite{GUPTA} we rigorously study this model under a different scaling of parameters. We multiply $k_{\textnormal{fb}}$ and $k_{\textnormal{off}}$ by $N$, leaving $k_{\textnormal{on}}$ and $D$ unchanged.  We keep track of the locations of the membrane molecules as well as their \emph{clan identities}. A clan identity is a number in $[0,1]$ which is passed unaltered from the parent molecule to the offspring. A molecule that associates itself on the membrane is assigned a uniformly chosen clan identity in $[0,1]$. At any time, the molecules on the membrane that have the same clan identity should have a common ancestor and hence they must belong to the same clan. Note that here the number of types ($q$) is equal to $1$ and the population density is the same as the fraction of cell molecules that are on the membrane. Let $E = \hat{E} \times [0,1]$, where $\hat{E}$ is the membrane (sphere of radius $R$ in $\R^3$).  When the number of molecules is $N$, the population dynamics is described by a $\mathcal{M}_{N,a}(E)$-valued process $\{\mu^N(t) : t \geq 0\}$ as before. For any $h \in \R_+$ let $\beta(h) = k_{\textnormal{fb}} (1-h) $, $\rho(h) = k_{\textnormal{off}}$ and $\kappa(h) = k_{\textnormal{on}}(1-h)$. Let $\Theta \in \mathcal{P}(E)$ be the uniform distribution on $E$ and let the spatial migration operator $B$ (see Section \ref{migrationmechanism}) be $(D/2)\Delta$, where $\Delta$ denotes the Laplace-Beltrami operator on the sphere $\hat{E}$. For any $f : \hat{E} \times [0,1] \to \R$, $\Delta$ acts on $f$ only as the function of the first coordinate. The operator $(D/2) \Delta$ is just the generator of the Brownian motion on $\hat{E}$ with diffusion rate $D$. With this notation one can verify that this model is a special case of the model in Section \ref{sec:immmigrationmodel}. The association mechanism is analogous to immigration while the feedback mechanism gives rise to the density regulation mechanism. Theorem \ref{mainresult} (see also Theorem 2.3 in \cite{GUPTA}) shows that as $N \to \infty$ we have $\mu^N \Rightarrow h_{\textnormal{eq}} \nu$ where 
\[h_{\textnormal{eq}}  = 1 -\frac{k_{\textnormal{off}}}{k_{\textnormal{fb}} }\]
and $\{\nu(t): t \geq 0\}$ is a $\mathcal{P}(E)$-valued Fleming-Viot process with generator 
\begin{align*}
\mathbf{A} F(\nu) &= \frac{D}{2} \sum_{l = 1}^{m} \langle \Delta  f_l,\nu \rangle \prod_{j \neq l } \langle f_j , \nu \rangle  
+ \frac{k_{\textnormal{off}}}{h_{\textnormal{eq}}}\sum_{1 \leq l \neq k \leq m} \left( \left\langle f_{l}f_{k},\nu \right\rangle - \left\langle f_{l},\nu \right\rangle \left\langle f_{k},\nu \right\rangle
 \right)  \prod_{j \neq l,k} \langle f_j,\nu\rangle \\
&+ k_{\textnormal{on}} \frac{ (1 -h_{\textnormal{eq}}) }{h_{\textnormal{eq}}  } \sum_{l = 1}^{m}  \left(  \int_{E}\int_{E} (f_l(y) - f_l(x))\Theta(dy)  \nu(dx) \right)   \prod_{j \neq l} \langle f_j , \nu \rangle
 \end{align*}
for any $F(\nu) = \prod_{l = 1}^{m} \langle f_l , \nu \rangle \in \mathcal{C}_0$. It can be shown that this Fleming-Viot process is ergodic and has a unique stationary distribution in $\mathcal{P}(\mathcal{P}(E))$ (see Section 5 in \cite{EK93} and Proposition 2.5 in \cite{GUPTA}). To study the evolution of the clan sizes we define a $\mathcal{P}([0,1])$-valued process $\{\nu_c(t) : t \geq 0\}$ by 
\begin{align*}
\nu_c(t,S) = \nu(t, \hat{E} \times S), \  S \in \mathcal{B}([0,1]).
\end{align*}
This is a Fleming-Viot process that describes the infinitely-many-neutral-alleles model (recall the discussion in Section \ref{sec:log1}). Therefore for any $t > 0$, we can write $\nu_c(t) = \sum_{i = 1}^{\infty} a_i \delta_{x_i}$, which means that $a_i$ fraction of the population has clan identity $x_i$. At stationarity, the clan sizes (arranged in descending order) are distributed according to the Poisson-Dirichlet distribution with parameter $\alpha = k_{\textnormal{on}}/ k_{\textnormal{fb}} $. Properties of the Poisson-Dirichlet distribution (see \cite{FengPoisson}) tell us that for any small $\epsilon > 0$, there is a positive probability of the largest clan having size greater than $(1- \epsilon)$. Furthermore one can show that at stationarity the molecules in each clan are concentrated on a circular patch on the membrane. The square of the radius of this patch can be approximately computed as (see Theorem 2.7 in \cite{GUPTA})
\begin{align*}
 \frac{ 2D }{\left(  \frac{( k_{\textnormal{on}} + k_{\textnormal{fb}}) k_{\textnormal{off}}}{ ( k_{\textnormal{fb}} - k_{\textnormal{off}} )  }  + \frac{D}{R^2}\right) }. 
\end{align*}
The last two assertions imply that if $D$ is small in comparison to $R^2$, then at stationarity there is a positive probability that most of the membrane molecules are in one clan and that clan is spread over a small area on the membrane. Due to ergodicity this event will occur infinitely often in any trajectory of the process $\{\nu(t): t \geq 0\}$. Whenever this event happens we can expect the cell to be \emph{polarized}. Therefore the limiting process exhibits recurring cell polarity. In \cite{GUPTA} we discuss how the frequency of observing polarity depends on various model parameters.

The above analysis shows that if the feedback mechanism is strong enough, it can counter spatial diffusion to generate cell polarity. However this conclusion is based on a highly simplified model. As mentioned in Section \ref{introduction}, most cells that exhibit polarity have complicated feedback circuits, with molecules of several types pulling each other on and off the membrane at various type-dependent rates. These different types of molecules may also have their own migration and association mechanisms. It would be interesting to know if the above analysis can be extended to general multi-type models for cell polarity. The results in this paper show that this can indeed be done as long as the feedback mechanism satisfies the assumptions in Section \ref{density_regulation_mechanism}, and acts at a faster timescale than the association and migration mechanisms. In this case, Theorem \ref{mainresult} guarantees convergence to a Fleming-Viot process and this limiting process can then be analyzed in the same way as in \cite{GUPTA}. This enables us to draw similar conclusions about the onset of cell polarity in this multi-type setting.

\section{Proofs}

\subsection{Well-posedness of the martingale problems for $\mathbf{A}^N_l$}
\label{sec:wellposedness}
 Recall the definitions of the operators $\mathbf{A}^N_0,\mathbf{A}^N_1,\mathbf{A}^N_2$ and $\mathbf{A}^N_3$ from \eqref{defn:an0}, \eqref{defn:an1}, \eqref{defn:an2} and \eqref{defn:an3}. In this section we prove Proposition \ref{prop:wellposedness} which says that the martingale problem for these operators is well-posed in the space $D_{\mathcal{M}^q_{N,a} (E)}[0,\infty)$. Pick a $l \in \{0,1,2,3\}$.
If we do not allow the dynamics under $\mathbf{A}^N_l$ to leave a compact set of $\mathcal{M}^q_{N,a} (E)$, then $\mathbf{A}^N_l$ can be viewed as a bounded perturbation of the migration operator $\mathbf{B}^N$ (given by \eqref{maindefbn}). The well-posedness of the corresponding martingale problem is immediate from the well-posedness of the martingale problem for $\mathbf{B}^N$ (see Chapter 4 in Ethier and Kurtz \cite{EK}). In our case, the dynamics under $\mathbf{A}^N_l$ may exit any compact set of $\mathcal{M}^q_{N,a} (E)$. However we can still argue the well-posedness of the corresponding martingale problem by showing that this exit time tends to infinity as the compact set gets bigger and bigger in size. We now make these ideas precise.   

\begin{lemma}
\label{lemma:stoppingtimesinfinity}
Fix a $l \in \{ 0,1,2,3\}$ , $N \in \N$ and $\pi \in \mathcal{P}\left( \mathcal{M}^q_{N,a}(E) \right)$. For each $k \in \N$ let $\{ \mu_k(t) : t \geq 0 \}$ be a 
$ \mathcal{M}^q_{N,a}(E)$-valued process with initial distribution $\pi$. Define a stopping time
\begin{align}
\label{defnstippingtimetauk}
 \tau_k = \inf \left\{ t \geq 0 : \| H(\mu_k(t-))\|_{1} \geq k \textrm{ or }   \| H(\mu_k(t))\|_1 \geq k \right\},
\end{align}
where $H$ is the density map \eqref{defn:densitymap}.
Suppose that for each $k \in \N$ and $\hat{F}  \in \mathcal{C}^q_0$
\begin{align*}
\hat{F} (\mu_k(t \wedge \tau_k)) - \hat{F} (\mu_k(0)) - \int_{0}^{t \wedge \tau_k} \mathbf{A}^N_l \hat{F} (\mu_k(s))ds
\end{align*}
is a martingale. Then for any $t \geq 0$
\begin{align*}
\lim_{k  \to \infty} \P \left(  \tau_k \leq t\right) = 0.
\end{align*}
\end{lemma}
\begin{proof}
Let $\{ h_k(t) = H(\mu_k(t)) : t \geq 0\}$ be the density process corresponding to $\mu_k$ and let $c_k : \R^q_{+} \to \R^q_{+}$ be the function defined by
\begin{align}
\label{defckhvector}
c_k(h) = (c_{k,1}(h),\dots,c_{k,q}(h))  = \left\{
\begin{array}{cc}
\bar{1}_q  & \textrm{ if } \| h \|_1 < 2k \\
\bar{0}_q & \textrm{ otherwise }. 
\end{array}\right.
\end{align}

Pick any $\epsilon \in (0,1)$. Since for each $k \in \N$ the distribution of $\mu_k(0)$ is $\pi$ there must exist a $k_{\epsilon} > 0$ such that
\begin{align*}
\P \left( \| h_k(0) \|_1 > k_{\epsilon} \right) < \epsilon \textrm{ for all } k \in \N. 
\end{align*} 
If $k$ satisfies $k_{\epsilon} \leq k \epsilon$ then
\begin{align}
\label{initialboundonh0expectation}
\E \left( \left\langle  c_k(h_k(0)), h_k(0)  \right\rangle \right) \leq k_{\epsilon}\P \left( \| h_k(0)\|_1 \leq k_{\epsilon} \right) + 
2k \P \left( \| h_k(0)\|_1 > k_{\epsilon} \right) \leq 3 \epsilon k. 
\end{align}

For each $i \in Q$ and $k \in \N$, define a function $\hat{F}^k_i \in \mathcal{C}^q_0$ by $\hat{F}^k_i(\mu) = c_{k,i}(h)h_i$ where $h = H(\mu)$. 
From \eqref{splittingofanl} we know that for any $\mu \in \mathcal{M}^q_{N,a}(E)$
\begin{align*}
\mathbf{A}^N_l \hat{F}^k_i(\mu) =  \mathbf{B}^N  \hat{F}^k_i(\mu) + N \mathbf{R}^N \hat{F}^k_i(\mu) + \mathbf{G}^N_l \hat{F}^k_i(\mu). 
\end{align*}
One can easily verify that $\mathbf{B}^N \hat{F}^k_i(\mu)=0$,  $\mathbf{R}^N \hat{F}^k_i(\mu) = c_{k,i}(h)\theta_i(h)$, $\mathbf{G}^N_0 \hat{F}^k_i(\mu) = 0$, $\mathbf{G}^N_2 \hat{F}^k_i(\mu) = 0$, 
$\mathbf{G}^N_3 \hat{F}^k_i(\mu) = c_{k,i}(h)\kappa_i(h)$ and finally
\[ \mathbf{G}^N_1 \hat{F}^k_i(\mu) = c_{k,i}(h) \left( \sum_{j \in Q } \int_{E}  b^{s}_{j i}(x,h)\mu_j(dx) -\int_{E } d^{s}_{i}(x,h) \mu_i(dx) \right).\] 
Note that for each $i,j \in Q$, the functions $ b^{s}_{ji}$, $d^s_i$ are bounded, while the functions 
$\theta_i$ and $\kappa_i$ satisfy \eqref{controlovertheta} and \eqref{lipchitzboundonkappa}. 
This implies that there exists a positive constant $C$ (depending on $N$ and $l$) such that
\begin{align}
\label{estimateofanl}
\mathbf{A}^N_l \hat{F}^k_i(\mu) \leq C \left( 1+  \langle c_k(h), h \rangle \right) \textrm{ for all } \mu \in \mathcal{M}^q_{N,a}(E).
\end{align}
By the assumption stated in the statement of this lemma we can say that
\begin{align*}
\hat{F}^k_i(\mu_k(t \wedge \tau_k)) -  \hat{F}^k_i(\mu_k(0)) - \int_{0}^{t \wedge \tau_k} \mathbf{A}^N_l \hat{F}^k_i(\mu_k(s))ds
\end{align*}
is a martingale starting at $0$. Taking expectations we get 
\begin{align}
\label{expectationoffik}
 \E \left( \hat{F}^k_i(\mu_k(t \wedge \tau_k)) \right) = \E \left( \hat{F}^k_i(\mu_k(0))  \right) + \E \left( \int_{0}^{t \wedge \tau_k}  \mathbf{A}^N_l \hat{F}^k_i(\mu_k(s))  ds \right).
\end{align}
Let $\hat{F}^k : \mathcal{M}^q_F(E) \to \R$ be given by $\hat{F}^k(\mu) = \sum_{i \in Q} \hat{F}^k_i(\mu) = \langle c_k(h), h \rangle$. 
Then summing over $i \in Q$ in \eqref{expectationoffik} and using \eqref{estimateofanl} we arrive at 
\begin{align*}
\E \left( \hat{F}^k(\mu_k(t \wedge \tau_k)) \right) & \leq \E \left( \hat{F}^k(\mu_k(0))  \right) + 
C q\int_{0}^{t} \left[ 1 + \E\left( \hat{F}^k( \mu_k(s \wedge \tau_k)) \right)\right]ds .
\end{align*}
From \eqref{initialboundonh0expectation} and Gronwall's inequality, for $k \geq k_{\epsilon}/ \epsilon$ we obtain
\begin{align*}
 \E \left( \hat{F}^k(\mu_k(t \wedge \tau_k)) \right) \leq  \left( 3 k \epsilon + C q t \right)e^{C q t}.
\end{align*}
Then by Markov's inequality 
\begin{align}
\label{app:markovinequality}
\lim_{k \to \infty} \P \left( \hat{F}^k(\mu_k(t \wedge \tau_k))  \geq k\right)  & \leq  \lim_{k \to \infty} \frac{ \E \left(\hat{F}^k(\mu_k(t \wedge \tau_k)) \right)}{k} \notag \\& \leq \lim_{k \to \infty} \frac{  \left( 3 k \epsilon + C qt \right)e^{C q t}}{k} \notag \\& \leq 3 \epsilon e^{C q t} .
\end{align}
Observe that 
\begin{align*}
 \P \left( \tau_k \leq t \right) & =   \P \left(  \left\| h_k(t \wedge \tau_k) \right\|_1 \geq k\right)  
 \\ &= \P \left( \hat{F}^k(\mu_k(t \wedge \tau_k))  \geq k\right) + \P \left(  \left\| h_k(t \wedge \tau_k) \right\|_1 \geq 2 k\right).
\end{align*}
For large $k$, the second probability on the right is $0$ because of the following reason.
The process $h_k$ has jumps of size $1/N$ and hence the definition of $\tau_k$ (see \eqref{defnstippingtimetauk}) implies that $\left\| h_k(t \wedge \tau_k) \right\|_1  \leq k + (1/N) < 2k$. Therefore using \eqref{app:markovinequality} we get
\begin{align*}
 \lim_{k \to \infty} \P \left( \tau_k \leq t \right) =  \lim_{k \to \infty} \P \left( \hat{F}^k(\mu_k(t \wedge \tau_k))  \geq k\right)  \leq 3 \epsilon e^{C q t}.
 \end{align*}
Letting $\epsilon \to 0$ proves the lemma.
\end{proof}

\begin{proof}[Proof of Proposition \ref{prop:wellposedness}]
Fix a $N \in \N$ and a $l \in \{0,1,2,3\}$. The space $\mathcal{M}^q_{N,a}(E)$ is complete and separable and for each $k \in \N$ the set 
\begin{align}
\label{defsetuk}
U_k = \left\{  \mu \in \mathcal{M}^q_{N,a} (E) :  \|H(\mu)\|_1 < k\right\}
\end{align}
is open with a compact closure in $\mathcal{M}^q_F(E)$. Define an operator $\mathbf{L}_k : \mathcal{D}(\mathbf{L}_k) = \mathcal{C}^q_0 \to B\left( \mathcal{M}^q_{N,a}(E)\right)$ by
\[\mathbf{L}_k \hat{F}(\mu) =  \mathbf{B}^N \hat{F}(\mu) + \ind_{U_k}(\mu) \left( N \mathbf{R}^N \hat{F}(\mu)+ \mathbf{G}^N_l \hat{F}(\mu)  \right),\]
for any $\hat{F} \in \mathcal{C}^q_0$. The operator $\mathbf{L}_k$ can be seen as a bounded perturbation of the operator $\mathbf{B}^N$. We argued in Section \ref{migrationmechanism} that the martingale problem for $\mathbf{B}^N$ is well-posed.
From Theorem 4.10.3 in Ethier and Kurtz \cite{EK}, the martingale problem for 
$\mathbf{L}_k$ is well-posed for each $k \in \N$. Pick a $\pi \in \mathcal{P}\left( \mathcal{M}^q_{N,a}(E) \right)$ and let $\{ \mu_k(t) : t \geq 0\}$ be the unique solution 
to the martingale problem for $(\mathbf{L}_k, \pi)$. Define a stopping time by
\begin{align*}
\tau_k = \inf\{ t \geq 0 : \mu_k(t) \notin U_k \textrm{ or } \mu_k(t-) \notin U_k \}. 
\end{align*}
Then for any $\hat{F} \in \mathcal{C}^q_0$
\begin{align*}
\hat{F}(\mu_k(t )) -  \hat{F}(\mu_k(0)) - \int_{0}^{t} \mathbf{L}_k \hat{F}(\mu_k(s))ds
\end{align*}
is a martingale. From \eqref{splittingofanl} one can see that if $\mu \in U_k$ then $\mathbf{A}^N_l \hat{F}(\mu) = \mathbf{L}_k \hat{F}(\mu)$.
Using the optional sampling theorem we get that
\begin{align*}
& \hat{F}(\mu_k(t \wedge \tau_k)) -  \hat{F}(\mu_k(0)) - \int_{0}^{t\wedge \tau_k} \mathbf{L}_k \hat{F}(\mu_k(s))ds 
\\& = \hat{F}(\mu_k(t \wedge \tau_k)) -  \hat{F}(\mu_k(0)) - \int_{0}^{t\wedge \tau_k} \mathbf{A}^N_l \hat{F}(\mu_k(s))ds 
\end{align*} 
is a martingale.
Lemma \ref{lemma:stoppingtimesinfinity} ensures that for any $t \geq 0$,
\begin{align*}
\lim_{k \to \infty} \P(\tau_k \leq t) = 0. 
\end{align*}
From Theorem 4.6.3 in Ethier and Kurtz \cite{EK} we can conclude that there exists a unique solution to the martingale problem for $(\mathbf{A}^N_l,\pi)$.
\end{proof}

\subsection{Properties of the limiting process}
\label{propertiesoflimitingprocesses}

The goal of this section is to prove Proposition \ref{prop:constancyofh_mainresult} which gives important insights into the limiting behaviour of the dynamics under the models described in Section \ref{mathmodels}. As mentioned in Section \ref{mainresultsection} our proof of Proposition \ref{prop:constancyofh_mainresult} will rely on the work of Katzenberger \cite{kat} which studies how semimartingales behave when they are driven by a fast drift that forces them to stay on a stable manifold. Before we can use the framework in \cite{kat} we need to prove some preliminary results. We start by recalling a tightness condition for semimartingales.
\begin{condition}
\label{NICE}
Let $\{Z^N : N \in \N \}$ be a sequence of $\{\mathcal{F}^N_t\}$-semimartingales with paths in $D_{\R^d}[0,\infty)$. Assume that for all $T>0$
\begin{align}
\label{cond:discontinuitiesdisappaer}
 \sup_{0 \leq t \leq T} \left\| Z^N(t) - Z^N(t-)\right\| \Rightarrow 0 \textrm{ as } N \to \infty.
\end{align}
Moreover assume that each semimartingale $Z^N$ can be written as
\[Z^N(t) = M^N(t) + \int_{0}^{t}A^N(s)ds  \]
where $M^N$ is a square integrable  $\{\mathcal{F}^N_t\}$-martingale and $A^N$ is a $\{\mathcal{F}^N_t\}$ adapted process satisfying
\[\sup_{N \in \N} \E \left( [M^N]_t + \int_{0}^{t} \left\| A^N(s) \right\| ds \right)  < \infty\]
for each $t \geq 0$, where $[M^N]_t$ is the quadratic variation of the martingale $M^N$.  
\end{condition}
\begin{remark}
\label{remark:nice}
If a sequence of semimartingales $\{Z^N : N \in \N \}$ satisfies Condition \ref{NICE} then this sequence is tight in $D_{\R^d}[0,\infty)$ in the Skorohod topology 
(see Corollary 2.3.3 in Joffe and Metivier \cite{Joffe}) and any limit point $Z$ is a semimartingale with continuous sample paths. 
\end{remark}

\begin{lemma}
\label{lemma:semproperty}
Pick a $l \in \{0,1,2,3\}$, $i \in Q$ and $f \in \mathcal{D}_0$. For each $N \in \N$, let $\{ \mu^N(t) : t \geq 0 \}$ be a solution of the martingale problem for $\mathbf{A}^N_l$. 
Define a real-valued process $\{ Z^{N}(t) : t \geq 0\}$ by 
\begin{align*}
Z^{N}(t) =&  \langle f,\mu^N_i(t) \rangle  - \langle f,\mu^N_i(0) \rangle \\
&- N \int_{0}^{t} \left[  \sum_{j \in Q}  \beta_{ji}(h^N(s)) \langle f,\mu^N_j(s) \rangle - \rho_i(h^N(s))  \langle f,\mu^N_i(s) \rangle  \right]ds,
\end{align*}
where $h^N(t) = H(\mu^N(t))$. Then $Z^N$ is a semimartingale with respect to the filtration generated by $\{ \mu^N(t) : t \geq 0 \}$. 
For any compact $K \subset \R^q_{+} $ define
\begin{align}
\label{defoflambdank}
\lambda^N(K) = \inf \left\{ t \geq 0 : h^N(t-) \notin \stackrel{\mathrm{o}}{K} \textrm{ or } h^N(t) \notin \stackrel{\mathrm{o}}{K}   \right\},
\end{align}
where $\stackrel{\mathrm{o}}{K}$ denotes the interior of the set $K$. If $Z^N_K$ is the semimartingale given by $Z^N_K(t) = Z^{N}(t \wedge \lambda^N(K))$ for $t\geq 0$, 
then the sequence of semimartingales $\{Z^N_K : N \in \N \}$ satisfies Condition \ref{NICE}.
\end{lemma}

\begin{proof}
Let $i \in Q$ and $f \in \mathcal{D}_0$ be fixed. For each $N \in \N$ and $l \in \{0,1,2,3\}$ define a function $a^N_l : \mathcal{M}^q_F(E) \to \R$ by
\begin{align*}
 a^N_0(\mu) &=  \langle B_i f, \mu_i \rangle, \\
 a^N_1(\mu) &=  \langle B_i f, \mu_i \rangle +  \sum_{j \in Q } \int_{E}  b^{s}_{ji}(x,h)f(x) \mu_j(dx)  -  \int_{E } d^{s}_{i}(x,h) f(x)  \mu_i(dx), \\
 a^N_2(\mu) & = \langle B_i f, \mu_i \rangle + N \sum_{j \in Q }  \int_{E } \beta_{ji}(h) p^N_{ji}(x) \left[ \int_{E} \left( f(y)-f(x) \right) \vartheta^N_{ji}(x,dy) \right]\mu_j(dx) \\
\textrm{ and } 
a^N_3(\mu) & = \langle B_i f, \mu_i \rangle + \kappa_i(h) \int_{E} f(x)\Theta_i(dx),  
\end{align*}
where $h = H(\mu)$.
Let $U^N_k \subset \mathcal{M}^q_{N,a}(E) $ be given by
\begin{align}
\label{defuk}
 U^N_k = \left\{ \mu \in \mathcal{M}^q_{N,a}(E) : \|H(\mu)\|_1 < 2 k \right\}.
\end{align}
Then for each $k \in \N$ and $l \in \{0,1,2,3\}$ we have
\begin{align}
\label{boundonanl}
\sup_{N \in \N} \sup_{\mu \in U^N_k } a^N_l(\mu) < \infty. 
\end{align}
To see this note that for any $N \in \N$, $U^N_k \subset U_k := \{ \mu \in \mathcal{M}^q_{F}(E) : \|H(\mu)\|_1 \leq 2 k \}$ and $U_k$ is a compact subset of $\mathcal{M}^q_{F}(E)$. For $l \in \{0,1,3\}$, $a^N_l$ is a continuous function which does not depend on $N$ and hence we get \eqref{boundonanl} simply by observing that
\begin{align}
\label{boundonanl_1}
 \sup_{N \in \N} \sup_{\mu \in U^N_k } a^N_l(\mu) \leq \sup_{\mu \in U_k } a^N_l(\mu) < \infty. 
\end{align}
Similarly if we define a continuous function $\hat{a}_2 : \mathcal{M}^q_F(E) \to \R$ by
\[ \hat{a}_2(\mu) = \langle B_i f, \mu_i \rangle + \sum_{j \in Q }  \beta_{ji}(h)  \langle C_{ji} f, \mu_j \rangle \] 
then we also have
\begin{align}
\label{boundonanl_2}
 \sup_{N \in \N} \sup_{\mu \in U^N_k } \hat{a}_2(\mu) \leq \sup_{\mu \in U_k } \hat{a}_2(\mu)  < \infty. 
\end{align}
Here $C_{ji}$'s are the operators satisfying Assumption \ref{assmp:limitrelationshipondispersal}. This assumption also implies that
\begin{align*}
& \sup_{N \in \N} \sup_{\mu \in U^N_k }  \left|  a^N_2(\mu) -  \hat{a}_2(\mu) \right| \\
& \leq  \sup_{N \in \N} \sup_{\mu \in U^N_k }      \sum_{j \in Q }  \beta_{ji}(h) \int_{E } \left(  N p^N_{ji}(x)  \int_{E} \left( f(y)-f(x) \right) \vartheta^N_{ji}(x,dy) - C_{ji}f(x) \right) \mu_j(dx)  \\
& \leq \sup_{N \in \N} \sup_{\mu \in U^N_k } \sum_{j \in Q }  \beta_{ji}(h)h_j \sup_{x \in E } \left(  N p^N_{ji}(x)  \int_{E} \left( f(y)-f(x) \right) \vartheta^N_{ji}(x,dy) - C_{ji}f(x) \right) \\
&< \infty.
\end{align*}
This bound along with \eqref{boundonanl_2} and the triangle inequality shows \eqref{boundonanl} for $l = 2$.  

Let $c_k : \R^q_{+} \to \R^q_{+}$ be given by \eqref{defckhvector}. Define $\hat{F}_k : \mathcal{M}^q_F(E) \to \R$ by
\[\hat{F}_k(\mu) = c_{k,i}(h) \langle f, \mu_i\rangle,\]
where $h = H(\mu)$. One can verify that for any $\mu \in U^N_k$ 
\begin{align}
& \left( \mathbf{B}^N + \mathbf{G}^N_l \right) \hat{F}_k(\mu)  = a^N_l(\mu) \label{formanf} \\
\textrm{ and } \ & \mathbf{R}^N \hat{F}_k(\mu) = \sum_{j \in Q}  \beta_{ji}(h) \langle f,\mu_j \rangle - \rho_i(h)  \langle f,\mu_i \rangle \label{formrnf}.
\end{align} 
Suppose $\{ \mu^N(t) : t \geq 0\} $ solves the martingale problem for $\mathbf{A}^N_l$ and $\{ h^N(t) = H( \mu^N(t) ) : t \geq 0 \}$ is the corresponding density process. Define another process $\{ m^N(t) : t \geq 0\}$ by
\begin{align}
\label{lemma:defmn}
 m^N(t) = Z^N(t) - \int_{0}^{t} a^N_l(\mu^N(s))ds.
\end{align}
For any $k \in \N$ let
\begin{align}
\label{lemma:deftnk}
\tau^N_k = \inf \{ t \geq 0 : \|h^N(t-)\|_1 \geq k \textrm{ or }\|h^N(t)\|_1 \geq k  \}. 
\end{align}
From Lemma \ref{lemma:stoppingtimesinfinity} we can conclude that for any fixed $N$, the stopping times $\tau^N_k$ 
converge to $\infty$ a.s.\ as $k \to \infty$. 
Observe that $\hat{F}_k$ belongs to the class $\mathcal{C}^q_0 = \mathcal{D}(\mathbf{A}^N_l)$. Hence
\begin{align*}
\hat{F}_k(\mu^N(t))- \hat{F}_k(\mu^N(0)) -\int_{0}^{t} \mathbf{A}^N_l \hat{F}_k(\mu^N(s))ds
\end{align*}
is a martingale. From \eqref{splittingofanl} and the optional sampling theorem we get that 
\begin{align}
\label{lemma:defmnkt}
m^N_k(t) & = \hat{F}_k(\mu^N(t \wedge \tau^N_k)) -  \hat{F}_k(\mu^N(0)) - \int_{0}^{t\wedge \tau^N_k} \mathbf{A}^N_l \hat{F}_k(\mu^N(s))ds \\
         & = \hat{F}_k(\mu^N(t \wedge \tau^N_k)) -  \hat{F}_k(\mu^N(0)) - N \int_{0}^{t\wedge \tau^N_k} \mathbf{R}^N \hat{F}_k(\mu^N(s))ds \notag \\
& - \int_{0}^{t \wedge \tau^N_k} \left( \mathbf{B}^N +  \mathbf{G}^N_l\right)\hat{F}_k(\mu^N(s))ds \notag
\end{align}
is also a martingale. If the set $(0, t \wedge \tau^N_k]$ is non-empty then for any $s \in (0, t \wedge \tau^N_k]$, we have $c_k(h^N(s)) = \bar{1}_q$ and therefore $\hat{F}_k(\mu^N(s)) = \langle f , \mu^N_i(s)\rangle$.
If the set $(0, t \wedge \tau^N_k]$ is empty then $t \wedge \tau^N_k = 0$ and in this case $m^N_k(t) = 0$. From \eqref{formanf} and \eqref{formrnf} we see that for all $t \geq 0$, $m^N_k(t) = m^N(t \wedge \tau^N_k)$, where $m^N$ is defined by \eqref{lemma:defmn}. But $m^N_k$ is a martingale and for a fixed $N$, $\tau^N_k \to \infty$ a.s.\ as $k \to \infty$. Therefore we can conclude that $m^N$ is local martingale and hence $Z^N$ is a semimartingale.

Let $\hat{F}^2_k : \mathcal{M}^q_F(E) \to \R$ be given by $\hat{F}^2_k(\mu) = (\hat{F}_k(\mu))^2$. Note that for any $\mu \in \mathcal{M}_F(E)$
\begin{align*}
\left\langle f , \mu \pm \frac{1}{N} \delta_x \right\rangle^2 -  \langle f , \mu \rangle^2 &= \pm 2 \langle f , \mu \rangle \frac{f(x)}{N} + \frac{f^2(x)}{N^2}.
\end{align*}
Using this one can verify that if $\mu \in U^N_k$ and $h = H(\mu)$ then we have
\begin{align*}
&N \left( \mathbf{R}^N \hat{F}_k^2(\mu) - 2 \hat{F}_k(\mu) \mathbf{R}^N \hat{F}_k(\mu) \right) 
 =  \sum_{j \in Q} \beta_{ji}(h) \langle f^2 ,\mu_j \rangle + \rho_i(h) \langle f^2,\mu_i \rangle, \\  
& \mathbf{G}^{N}_1 \hat{F}^2_k(\mu) - 2 \hat{F}_k(\mu) \mathbf{G}^{N}_1 \hat{F}_k(\mu) \\ &= \frac{1}{N} \sum_{j \in Q} \int_{E}  b^{s}_{ji}(x,h)f^2(x) \mu_j(dx)  +  \int_{E } d^{s}_{i}(x,h) f^2(x)  \mu_i(dx), \\
& \mathbf{G}^N_2 \hat{F}^2_k(\mu)  -  2 \hat{F}_k(\mu) \mathbf{G}^{N}_2 \hat{F}_k(\mu) \\ & =   \sum_{j \in Q }  \int_{E } \beta_{ji}(h) p^N_{ji}(x) 
\left( \int_{E} (f^2(y)-f^2(x) ) \vartheta^N_{ji}(x,dy) \right)\mu_j(dx)\\
\textrm{ and } & \mathbf{G}^N_3 \hat{F}^2_k(\mu) -  2 \hat{F}_k(\mu) \mathbf{G}^{N}_3 \hat{F}_k(\mu) =   \frac{1}{N} \kappa_i(h) \int_{E} f^2(x)\Theta_i(dx).
\end{align*}
Also for any $\mu \in U^N_k$
\begin{align*}
 \mathbf{B}^N \hat{F}^2_k(\mu) -  2 \hat{F}_k(\mu)\mathbf{B}^N \hat{F}_k(\mu) = \frac{1}{N} \langle B_i f^2 - 2 f B_i f, \mu_i \rangle. 
\end{align*}
Using \eqref{splittingofanl} and the above expressions we can show in a manner similar to \eqref{boundonanl} that for any $k \in \N$ we have
\begin{align}
\label{lemma:boundforqv}
\sup_{N \in \N} \sup_{ \mu \in U^N_k }  |\mathbf{A}^N_l \hat{F}^2_k(\mu) - 2 \hat{F}_k(\mu)\mathbf{A}^N_l \hat{F}_k(\mu) | < \infty.
\end{align}
The function $\hat{F}^2_k$ is also in $\mathcal{C}^q_0 = \mathcal{D}(\mathbf{A}^N_l)$. Therefore if $m^N_k$ is the martingale given by \eqref{lemma:defmnkt} then 
\begin{align*}
\left( m^N_k(t) \right)^2 - \int_{0}^{t \wedge \tau_k} \left( \mathbf{A}^N_l \hat{F}^2_k(\mu^N(s)) - 2 \hat{F}_k(\mu^N(s))\mathbf{A}^N_l \hat{F}_k(\mu^N(s))\right)ds 
\end{align*}
is also a martingale. Therefore the expected quadratic variation of $m^N_k$ can be computed as
\begin{align}
\label{lemma:finalboundonqv2}
\E \left( [m^N_k]_t \right) =\E \left( (m^N_k(t))^2 \right) = 
\E \left(  \int_{0}^{t \wedge \tau^N_k}  \left( \mathbf{A}^N_l \hat{F}^2_k(\mu^N(s)) - 2 \hat{F}_k(\mu^N(s))\mathbf{A}^N_l \hat{F}_k(\mu^N(s))\right)ds \right). 
\end{align}
For any fixed $k \in \N$, the sequence of semimartingales $\{Z^N(\cdot \wedge \tau^N_k) : N \in \N \}$ satisfy \eqref{cond:discontinuitiesdisappaer} because the 
discontinuities of 
$\mu^N$ are of size proportional to $1/N$. From \eqref{lemma:defmn} we can see that the semimartingale $Z^N(\cdot \wedge \tau^N_k)$ can be decomposed as
\begin{align*}
Z^N(t \wedge \tau^N_k) = m^N(t \wedge \tau^N_k) + \int_{0}^{t \wedge \tau^N_k } a^N_l(\mu^N(s))ds =   m^N_k(t) + \int_{0}^{t} \ind_{\{s \leq \tau^N_k\}} a^N_l(\mu^N(s))ds.
\end{align*}
For any $0 < s \leq \tau^N_k$, $\mu^N(s) \in  U^N_k$. Using \eqref{lemma:finalboundonqv2}, \eqref{lemma:boundforqv} and \eqref{boundonanl} we can see that
\begin{align}
\label{lemma:finalboundonqv}
\sup_{N \in \N}  \E \left( [m^N_k]_t  + \int_{0}^{t} \ind_{\{s \leq \tau^N_k\}} a^N_l(\mu^N(s))ds\right) < \infty.
\end{align}  
Therefore for any $k \in \N$, the sequence of semimartingales $\{ Z^N(\cdot \wedge \tau^N_k)\}$ satisfies Condition \ref{NICE}. 
For any compact set $K \subset \R^q_{+}$, there exists a $k$ such that $K \subset \{ h \in \R^q_{+}: \|h\|_1 < k \}$. If the stopping time $\lambda^N_K$ is defined by 
\eqref{defoflambdank} then $\lambda^N_K \leq \tau^N_k$ a.s.\ where $\tau^N_k$ is given by \eqref{lemma:deftnk}. Hence it is immediate that if $Z^N_K$ is the semimartingale 
defined by $Z^N_K(\cdot) = Z^N( \cdot\wedge \lambda^N(K))$ then the sequence of semimartingales $\{ Z^N_K : N  \in \N \}$ will also satisfy Condition \ref{NICE}.  
\end{proof}

\begin{proof}[Proof of Proposition \ref{prop:constancyofh_mainresult}]
Let $\{\mathcal{F}^N_t\}$ be the filtration generated by the process $\{ \mu^N(t) : t \geq 0\}$. For any compact set $K \subset U_{\textnormal{eq}}$ let $\lambda^N(K)$ be given by \eqref{defoflambdank}. 
In this proof a sequence of $\{\mathcal{F}^N_t\}$-semimartingales $\{Z^N : N \in \N\}$ with paths in $D_{\R^d}[0,\infty)$ will be called \emph{well-behaved} 
if for each compact $K \subset U_{\textnormal{eq}}$, the sequence of semimartingales $\{Z^N (\cdot \wedge \lambda^N(K))\}$ satisfies Condition \ref{NICE}.

Using Lemma \ref{lemma:semproperty} with $f = 1_E$, (where $1_E$ is as in \eqref{defn:densitymap}) for each $i \in Q$ 
we obtain a well-behaved $\R$-valued semimartingale $Z^{N,1}_i$ such that
\begin{align*}
\langle 1_E, \mu^N_i(t) \rangle & = \langle 1_E, \mu^N_i(0) \rangle + 
N  \int_{0}^{t} \left[  \sum_{j \in Q}  \beta_{ji}(h^N(s)) \langle 1_E, \mu^N_j(s) \rangle  - \rho_i(h^N(s)) \langle 1_E, \mu^N_i(s) \rangle   \right]ds 
\\& + Z^{N,1}_i(t).
\end{align*}
Recall the definition of the matrix $A(h)$ from \eqref{defninteractionmatrix}.
The above expression is the same as 
\begin{align*}
h^N_i(t) & = h^N_i(0) \rangle + N  \int_{0}^{t} \left[  \sum_{j \in Q}  \beta_{ji}(h^N(s)) h^N_j(s) - \rho_i(h^N(s)) h^N_i(s)   \right]ds + Z^{N,1}_i(t)\\
         & =   h^N_i(0)  + N \int_{0}^{t} \left[ \sum_{j \in Q} A_{ij}(h^N(s))h^N_j(s) \right] ds  + Z^{N,1}_i(t).
\end{align*}
If we let $Z^{N,1}$ to be the $\R^q$-valued semimartingale given by
\[Z^{N,1}(t) = (Z^{N,1}_1(t),\dots,Z^{N,1}_q(t)) \textrm{ for }  t \geq 0,\]
then $Z^{N,1}$ is also a well-behaved semimartingale. The semimartingale $h^N$ satisfies
\begin{align}
h^N(t)  & = h^N(0)  + N \int_{0}^{t} A(h^N(s)) h^N(s)ds  + Z^{N,1}(t) \notag \\
        & = h^N(0)  + N \int_{0}^{t} \theta(h^N(s))ds  + Z^{N,1}(t) ,\label{prop:vectorformofhn}
\end{align}
where the last equality holds due to definition \eqref{defnthetah}. Now fix a $f \in \mathcal{D}_0$. From Lemma \ref{lemma:semproperty}, 
for each $i \in Q$ there is a well-behaved semimartingale $Z^{N,f}_i$ such that
\begin{align*}
\langle f , \mu^N_i(t) \rangle & = \langle f , \mu^N_i(0) \rangle +  N \left( \int_{0}^{t} \sum_{j \in Q } \beta_{ji}(h^N(s)) \langle  f , \mu^N_j (s)\rangle ds -  \int_{0}^{t} \rho_i(h^N(s)) \langle  f , \mu^N_i (s)\rangle ds \right) \\
& + Z^{N,f}_i(t) \\
& = \langle f , \mu^N_i(0) \rangle +  N \int_{0}^{t} \left[ \sum_{j \in Q} A_{ij}(h^N(s)) \langle  f , \mu^N_j (s)\rangle \right]  ds + Z^{N,f}_i(t).   
\end{align*}
Using the integration by parts formula for semimartingales, for each $i,j \in Q$ we can write
\begin{align}
\label{prop:expansionofhixj}
h^N_i(t) \langle f , \mu^N_j(t) \rangle &= h^N_i(0) \langle f , \mu^N_j(0) \rangle + 
N \left( \int_{0}^{t} h^N_i(s) \left[ \sum_{k \in Q } A_{jk}(h^N(s)) \langle f , \mu^N_k(s) \rangle \right] ds  \right. \notag \\  &
\left. + \int_{0}^{t} \langle f , \mu^N_j(s) \rangle \left[ \sum_{k \in Q } A_{ik}(h^N(s)) h^N_k(s) \right]  ds \right) + Z^N_{ij}(t),
\end{align}
where $Z^N_{ij}$ is another well-behaved semimartingale given by
\begin{align*}
Z^N_{ij}(t) = \int_{0}^{t} h^N_i(s)Z^{N,f}_j(s)ds +  \int_{0}^{t}  \langle f , \mu^N_j(s) \rangle  Z^{N,1}_i(s)ds + [Z^{N,1}_i,Z^{N,f}_j]_t. 
\end{align*}
The last term in the above equation is the cross-variation term between $Z^{N,1}_i$ and $Z^{N,f}_j$. Now for each $i \in Q$ define the semimartingale $Y^N_i$ by
\begin{align}
\label{prop:defyni}
Y^N_i(t) =  \langle f , \mu^N_i(t) \rangle \left( \sum_{j \in Q}   h^N_j(t)\right)  - h^N_i(t)\left( \sum_{j \in Q}  \langle f , \mu^N_j(t) \rangle\right). 
\end{align}
The semimartingale $Y^N_i$ is just a linear combination of the semimartingales of the form $h^N_j(t) \langle f , \mu^N_i(t) \rangle$. 
Using \eqref{prop:expansionofhixj} we can write
\begin{align}
\label{prop:expansionofyni}
Y^N_i(t) =  Y^N_i(0) + N \int_{0}^{t} L^N_i(s) ds + Z^{N,2}_i(t), 
\end{align}
where $Z^{N,2}_i$ is a well-behaved semimartingale and for any $t \geq 0$
\begin{align*}
L^N_i(t) & = \sum_{j,k \in Q} \left[ A_{i k}(h^N(t)) h^N_j(t) \langle f , \mu^N_k(t) \rangle  + A_{j k}(h^N(t)) h^N_k(t) \langle f , \mu^N_i(t) \rangle  
\right. \\ & \left. -A_{j k}(h^N(t)) h^N_i(t) \langle f , \mu^N_k(t) \rangle - A_{i k}(h^N(t)) h^N_k(t) \langle f , \mu^N_j(t) \rangle \right] \\
& = \sum_{k \in Q} A_{ik}(h^N(t)) Y^N_k(t) +  \langle f , \mu^N_i(t) \rangle \left[ \sum_{j,k \in Q}  A_{j k}(h^N(t)) h^N_k(t) \right] \\& - h^N_i(t) \left[ \sum_{j,k \in Q}  A_{j k}(h^N(t))   \langle f , \mu^N_k(t) \rangle \right].
\end{align*}
Define a matrix $G(h)$ for each $h \in \R^q_{+}$ as follows 
\begin{align*}
G(h) = \left\{ 
\begin{array}{cc}
 A(\bar{0}_q) & \textrm{ if } h = \bar{0}_q \\
   A(h) +  \left( \frac{\langle \bar{1}_q , \theta(h) \rangle }{\langle \bar{1}_q , h \rangle } I_q - \frac{h}{\langle \bar{1}_q , h \rangle } \bar{1}^T_q A(h) \right) &  \textrm{otherwise}. 
\end{array} \right.
 \end{align*}
Note that if $h^N(t) \neq \bar{0}_q$ then
\begin{align*}
& \langle f , \mu^N_i(t) \rangle = \frac{ Y^N_i(t) +  h^N_i(t)\left( \sum_{j \in Q} \langle f , \mu^N_j(t) \rangle\right)}{\langle \bar{1}_q, h^N(t)\rangle} \textrm{ and }\\
& \sum_{j,k \in Q} A_{jk}(h^N(t)) \langle f , \mu^N_k(t) \rangle  \\&= \frac{ \sum_{j,k \in Q} A_{jk}(h^N(t)) Y^N_k(t)+
\left( \sum_{j \in Q}   \langle f , \mu^N_j(t) \rangle \right) \left( \sum_{j,k \in Q} A_{jk}(h^N(t)) h^N_k(t)  \right) }{  \langle \bar{1}_q ,  h^N(t) \rangle}.
\end{align*}
This allows us to write
\begin{align*}
L^N_i(t) & = \sum_{j \in Q} G_{ij}(h^N(t)) Y^N_j(t) 
\end{align*}
and hence from \eqref{prop:expansionofyni}
\begin{align}
\label{prop:expansionofyni2}
Y^N_i(t) =  Y^N_i(0) + N \int_{0}^{t} \left[ \sum_{j \in Q}  G_{ij}(h^N(s)) Y^N_j(s) \right] ds + Z^{N,2}_i(t). 
\end{align}
Let $Y^N$ and $Z^{N,2}$ be the $\R^{q-1}$-valued semimartingales given by
\begin{align*}
Y^N(t) = (Y^N_1(t),\dots, Y^N_{q-1}(t)) \textrm{ and } Z^{N,2}(t) = (Z^{N,2}_1(t),\dots, Z^{N,2}_{q-1}(t)) \textrm{ for } t \geq 0. 
\end{align*}
For each $h \in \R^q_{+}$ let $\bar{G}(h) \in \mathbb{M}(q-1,q-1)$ be the matrix defined by
\[\bar{G}_{ij}(h) = G_{ij}(h) - G_{iq}(h)\textrm{ for all }i,j \in \{1,\dots,q-1\}.\]
Observe that $\sum_{i \in Q} Y^N_i(t) = 0$ and hence $Y^N_q(t) = -\sum_{i=1}^{q-1} Y^N_i(t)$. From \eqref{prop:expansionofyni2} we get
\begin{align}
\label{prop:expansionofyni3}
Y^N(t) = Y^N(0) + N \int_{0}^{t} \bar{G}(h^N(s))Y^N(s)ds + Z^{N,2}(t).  
\end{align}
From part (E) of Lemma \ref{lemma:appendix1} the matrix $\bar{G}(h_{\textnormal{eq}})$ is stable, that is, all its eigenvalues have strictly negative real parts. 

We now define a $\R^q_{+} \times \R^{q-1}$-valued semimartingale $X^N$ by 
\begin{align}
\label{prop:defxn0}
X^N(t) = (h^N(t), Y^N(t)) = (h^N_1(t),\dots,h^N_q(t),Y^N_1(t),\dots,Y^N_{q-1}(t)) \textrm{ for } t \geq 0.
\end{align}
From \eqref{prop:vectorformofhn} and \eqref{prop:expansionofyni3} we can see that $X^N$ satisfies
\begin{align}
\label{prop:defxn}
X^N(t) = X^N(0) + N \int_{0}^{t} F(X^N(s))ds + Z^N(t) 
\end{align}
where $Z^N(t) = (Z^{N,1}(t), Z^{N,2}(t))$ is a well-behaved semimartingale and 
$F : \R^q_{+} \times \R^{q-1} \to \R^{2q-1}$ is the function given by
\begin{align}
\label{prop:defndrift}
F(x) = (\theta(h), \bar{G}(h)y) \textrm{ for } x = (h,y) \in \R^q_{+} \times \R^{q-1}.  
\end{align}
Let $x_{\textnormal{eq}} = (h_{\textnormal{eq}}, \bar{0}_{q-1})$. Then $F(x_{\textnormal{eq}}) = \bar{0}_{2q-1}$ and the Jacobian matrix 
$[J F(x_{\textnormal{eq}})]\in \mathbb{M}(2q-1,2q-1)$ has the block lower-triangular form  
\begin{align*}
[J F(x_{\textnormal{eq}})] = \left[
\begin{array}{cc}
[J \theta (h_{\textnormal{eq}})]& O_{q,q-1} \\
C &  \bar{G}(h_{\textnormal{eq}}) \\
\end{array}
\right], 
\end{align*}
where $C$ is some $(q-1) \times q$ matrix in $\mathbb{M}(q-1,q)$ and $O_{q,q-1}$ is the $q \times (q-1)$ matrix of zeroes. 
We mentioned above that the matrix $\bar{G}(h_{\textnormal{eq}})$ is stable and part (B) of Assumption \ref{mainassumptionsonr} says that the matrix $[J \theta (h_{\textnormal{eq}})]$ is also stable. 
Due to the block triangular form, the matrix $[J F(x_{\textnormal{eq}})]$ is stable as well.

Pick any $h_0 \in U_{\textnormal{eq}}$ and $y_0 \in \R^{q-1}$. Let $h(t) = \psi_\theta(h_0,t)$ for all $t \geq 0$, where $\psi_\theta$ is the \emph{flow} defined in Section 
\ref{density_regulation_mechanism}. The set $U_{\textnormal{eq}} \subset \R^q_{+}$ in $\psi_\theta$-invariant and hence $h(t) \in U_{\textnormal{eq}}$ for all $t \geq 0$. 
Let $y(t)$ be the unique solution of the initial value problem 
\begin{align}
\label{ivp2}
\frac{d y}{d t} = \bar{G}(h(t))y , \quad y(0)=y_0. 
\end{align}
Since the above differential equation is linear in the $y$ variable, the solution $y(t)$ is defined for all $t \geq 0$. 
Moreover $h(t) \to h_{\textnormal{eq}}$ and $\bar{G}(h(t)) \to \bar{G}(h_{\textnormal{eq}})$ as $t \to \infty$. 
The matrix $\bar{G}(h_{\textnormal{eq}})$ is stable and therefore $y(t) \to \bar{0}_{q-1}$ as $t \to \infty$.
Let $U(x_{\textnormal{eq}}) = U_{\textnormal{eq}} \times \R^{q-1}$. For each $x_0 = (h_0,y_0) \in U(x_{\textnormal{eq}}) $ and $t \geq 0$, let $\psi_F(x_0, t) = (h(t),y(t))$ with $h(t) = \psi_\theta(h_0,t)$ and 
$y(t)$ being the solution of \eqref{ivp2}.

The mapping $\psi_F : U(x_{\textnormal{eq}}) \times \R_+ \to U(x_{\textnormal{eq}})$ is the flow of the vector field $F$ on $U(x_{\textnormal{eq}})$. For all $x \in U(x_{\textnormal{eq}})$ and $t \geq 0$, $\psi_F$ satisfies 
\begin{align}
\label{flowofpsif}
\psi_F(x,t) = x + \int_{0}^{t} F( \psi_F(x,s) )ds. 
\end{align}
From the discussion in the preceding paragraph we can conclude that
\begin{align}
\label{prop:limitingbehaviorofpsif}
\lim_{t \to \infty} \psi_F(x,t)=x_{\textnormal{eq}} \textrm{ for all } x \in  U(x_{\textnormal{eq}}).
\end{align}
We have assumed in this proposition that there is a compact set $K_0 \subset U_{\textnormal{eq}}$ such that $h^N(0) \in K_0$ a.s.\ for each $N \in \N$. 
Note that any function $f \in \mathcal{D}_0$ is bounded. The definition of the semimartingale $Y^N$ guarantees that there is a 
compact set $K_1 \subset \R^{q-1}$ such that $X^N(0) = (h^N(0),Y^N(0)) \in K_0 \times K_1 \subset U(x_{\textnormal{eq}})$ a.s.\ for each $N \in \N$.

Now consider the equation \eqref{prop:defxn}. For large values of $N$, the semimartingale $X^N$ is driven by a large drift term of the form $N F(X^N(\cdot))$. 
The vector $x_{\textnormal{eq}}$ is a stable fixed point for this drift term and $U(x_{\textnormal{eq}})$ is its region of attraction. 
If we start in this region of attraction, then this drift is very forceful. It completely overwhelms the effect of the well-behaved semimartingale $Z^N$ and 
drives $X^N$ to the stable fixed point $x_{\textnormal{eq}}$. 
Moreover as $N$ gets large, the trajectories of $X^N$ start looking more and more like the trajectories of the deterministic flow $\psi_F$ with time compressed by a 
factor of $N$. These ideas are made precise in a much more general setting by Katzenberger \cite{kat}. We use Theorem 6.3 in \cite{kat} to deduce that
for any $T>0$
\begin{align}
\label{prop:mainrelationforxn}
\sup_{t \in [0,T]}\left\| X^N(t) - \psi_F(X^N(0),Nt) \right\| \Rightarrow 0 \textrm{ as } N \to \infty,
\end{align}
where $\| \cdot\|$ is the standard Euclidean norm. From the definition of $X^N$ and $\psi_F$ it is also clear that for any $T>0$
\begin{align}
\label{prop:mainrelationforhn}
\sup_{t \in [0,T]}\left\| h^N(t) - \psi_\theta(h^N(0),Nt) \right\| \Rightarrow 0 \textrm{ as } N \to \infty. 
\end{align}

From now on, for any $x \in \R^n$ and $\epsilon>0$, let $B^n_\epsilon(x)$ denote the open ball in $\R^n$ centered at $x$ with radius $\epsilon$. 
We have already argued that the Jacobian matrix of $F$ at $x_{\textnormal{eq}}$ is stable. 
From a simple linearization argument (see for example the proof of Theorem 3.7 in Khalil \cite{Khalil}) we can see that there exists a 
$\delta_0>0$ such that the open ball 
$B^{2q-1}_{\delta_0}(x_{\textnormal{eq}})  \subset U(x_{\textnormal{eq}})$ and for every $\delta \in (0,\delta_0)$ there exists a $\psi_F$-invariant open set $W_{\delta}$ whose 
closure $\bar{W}_{\delta}$ is contained in $B^{2q-1}_{\delta}(x_{\textnormal{eq}})$. 
From \eqref{prop:limitingbehaviorofpsif} we obtain
\begin{align*}
K_0 \times K_1  \subset \bigcup_{t \geq 0} \left\{ x \in U(x_{\textnormal{eq}}) : \psi_F(x,t) \subset W_{\delta}\right\}.
\end{align*}
As $ W_{\delta}$ is $\psi_F$-invariant, the open sets on the right are getting bigger and bigger as $t$ increases. 
Compactness of $K_0 \times K_1$ implies that there exists a $t_{\delta}>0$ such that $\psi_F(x,t) \in W_{\delta}$ for all $x \in K_0 \times K_1$ and $t \geq t_\delta$. 
This immediately gives us
\begin{align*}
 \sup_{t \geq t_{\delta}} \sup_{x \in K_0 \times K_1} \left\| \psi_F(x,t) - x_{\textnormal{eq}}  \right\| \leq \delta
\end{align*}
and letting $\delta \to 0$ we obtain
\begin{align}
\label{prop:compactdriven}
 \limsup_{t \to \infty} \sup_{x \in K_0 \times K_1} \left\| \psi_F(x,t) - x_{\textnormal{eq}}  \right\| = 0.
\end{align}
Now let $t_N$ be any sequence satisfying the conditions of this proposition. Then for any $T > 0$ 
\begin{align*}
&\sup_{t \in [0,T]} \left\|X^N(t+t_N) - x_{\textnormal{eq}} \right\|
\\& \leq  \sup_{t \in [0,T]} \left\|X^N(t+t_N) - \psi_F(X^N(0), N(t+t_N)) \right\| +  \sup_{t \in [0,T]}  \left\|\psi_F(X^N(0), Nt+N t_N) -x_{\textnormal{eq}} \right\| \\
& \leq  \sup_{t \in [0,T]} \left\|X^N(t+t_N) - \psi_F(X^N(0), N(t+t_N)) \right\| +  \sup_{t \geq N t_N}  \sup_{x \in K_0 \times K_1} \left\|\psi_F(x,t) -x_{\textnormal{eq}} \right\|,
\end{align*}
where the second inequality is true because $X^N(0) \in K_0 \times K_1$ a.s.\ for each $N \in \N$. 
From \eqref{prop:mainrelationforxn} and \eqref{prop:compactdriven} we can see that as $N \to \infty$
\begin{align*}
\sup_{t \in [0,T]} \left\|X^N(t+t_N) - x_{\textnormal{eq}} \right\| \Rightarrow 0
\end{align*}
which of course implies that
\begin{align}
\label{thetwogoodlimits}
\sup_{t \in [0,T]} \left\|h^N(t+t_N) - h_{\textnormal{eq}} \right\| \Rightarrow 0 \textrm{ and }  \sup_{t \in [0,T]} \left\|Y^N(t+t_N)\right\| \Rightarrow 0.
\end{align}
This proves part (A) of the proposition since the norms $\| \cdot\|$ and $\| \cdot\|_1$ are equivalent in $\R^q$. From the definition of $Y^N_i$ we can check that for each $i,j \in Q$ and $t \geq 0$
\begin{align*}
h^N_j(t) Y^N_i(t) - h^N_i(t) Y^N_j(t) = \left( \sum_{l \in Q} h^N_l(t) \right) \left( h^N_j(t) \langle f , \mu^N_i(t)\rangle - h^N_i(t) \langle f , \mu^N_j(t)\rangle \right) .
\end{align*}
The limits \eqref{thetwogoodlimits} immediately give us part (B) of the proposition for any $f \in \mathcal{D}_0$. 
As $\mathcal{D}_0$ is dense in $C(E)$, part (B) holds for any $f \in C(E)$.
\end{proof}

The following lemma will be useful in proving Theorem \ref{mainresult}.
\begin{lemma}
\label{compact_containment}
Let the notation and assumptions be the same as in Proposition \ref{prop:constancyofh_mainresult}. Then there is a compact set $K \subset U_{\textnormal{eq}}$ such that for all $T > 0$
\begin{align*}
\lim_{N \to \infty} \P \left( h^N(t) \notin K \textrm{ for any } t \in [0,T] \right) = 0. 
\end{align*}
\end{lemma}
\begin{proof}
 Since the Jacobian matrix $[J \theta (h_{\textnormal{eq}})]$ is stable (part (B) of Assumption \ref{mainassumptionsonr}), by a 
linearization argument similar to the one referred in the proof above, we can find an $\epsilon > 0$ such that the open ball $B^q_{\epsilon}(h_{\textnormal{eq}}) \subset U_{\textnormal{eq}}$ and there exists a 
$\psi_\theta$-invariant open set $U_\epsilon$ such that its closure $\bar{U}_{\epsilon} \subset B^q_{\epsilon}(h_{\textnormal{eq}})$. One can argue as before that since $K_0$ is a compact 
set, there exists a $t_{\epsilon}$ such that for all $t \geq t_{\epsilon}$ and $x \in K_0$, $\psi_\theta(x,t) \in U_{\epsilon}$. 
If we define $\hat{K}_0$ as 
\begin{align}
\label{defofkatk0}
\hat{K}_0 = \left\{ h \in U_{\textnormal{eq}} : \psi_{\theta}(h,t_\epsilon) \subset  \bar{U}_{\epsilon} \right\}, 
\end{align}
then it is a $\psi_\theta$-invariant compact set containing $K_0$. Because we have assumed that $h^N(0) \in K_0$ a.s.\ for all $N \in \N$, we must have that for all $t \geq 0$, 
$\psi_\theta(h^N(0),t) \in \hat{K}_0 \textrm{ a.s.}$. But $U_{\textnormal{eq}}$ is open in $\R^q_{+}$ (see Section \ref{density_regulation_mechanism}) and so there is a $\gamma >0$ such that 
\[K = \left\{ x \in U_{\textnormal{eq}} : \inf_{y \in \hat{K}_0} \|y-x\| \leq \gamma \right\}\]
is a compact subset of $U_{\textnormal{eq}}$. Observe that if for some $t \geq 0$, $h^N(t) \notin K$ then we must have that $\|h^N(t) - \psi_{\theta}(h^N(0),Nt)\| > \gamma$. Therefore   
\begin{align*}
 \lim_{N \to \infty} \P \left(h^N(t) \notin K\textrm{ for any } t \in [0,T] \right) \leq 
\lim_{N \to \infty} \P \left( \sup_{t \in [0,T]} \left\| h^N(t) -\psi_\theta(h^N(0),Nt) \right\| > \gamma \right).
\end{align*}
The limit on the right is $0$ due to \eqref{prop:mainrelationforhn} and this proves the lemma.
\end{proof}

\subsection{Solution to a system of partial differential equations}\label{section:pdesolution}

Recall the discussion at the end of Section \ref{mainresultsection}. To prove Theorem \ref{mainresult} we require a function $\Lambda$ that allows us to construct a $\mathcal{P}(E)$ valued process $\{\nu^N(t) : t \geq 0\}$ (see \eqref{tempformofnuN}) whose dynamics is \emph{well-behaved} as $N$ approaches $\infty$. The goal of this section is to guarantee that such a function $\Lambda$ exists.

Specifically, we need to show that for some open set $\hat{U}_{\textnormal{eq}} \subset \R^q$ containing $U_{\textnormal{eq}}$ (given by \eqref{defueq}), we have a function $\Lambda \in C^2( \hat{U}_{\textnormal{eq}} , \R^q_{*})$ 
which satisfies the following:  
\begin{align}
& A^{T}(h) \Lambda(h) +   \left[ J \Lambda(h)\right] \theta(h) = \bar{0}_q \textrm{ for all } h \in \hat{U}_{\textnormal{eq}}, \label{kernelequation} \\
&\langle \Lambda(h), h \rangle = 1 \textrm{ for all } h \in \hat{U}_{\textnormal{eq}} \label{innerproductcondition} \\
\textrm{ and }\quad & \Lambda(h_{\textnormal{eq}}) = v_{\textnormal{eq}} \label{initialcondition}.
\end{align}
Here $v_{\textnormal{eq}}$ is defined in \eqref{defb:veq} and $\left[ J \Lambda(h)\right]$ in equation \eqref{kernelequation} refers to the Jacobian matrix of $\Lambda$ at $h$. The significance of the above relations will become clear in Section \ref{section:fvconvergence}.

The major difficulty in solving \eqref{kernelequation} arises in the neighbourhood of $h_{\textnormal{eq}}$. This is because $\theta(h_{\textnormal{eq}}) = \bar{0}_q$ 
(part (A) of Assumption \ref{mainassumptionsonr}), which causes degeneracy in the system. However the next proposition shows that by employing power series expansions 
we can get around this problem and find an analytic solution to \eqref{kernelequation} in a neighbourhood of $h_{\textnormal{eq}}$. 
We later construct an open set $\hat{U}_{\textnormal{eq}}$ containing $U_{\textnormal{eq}}$ and extend the solution over the whole $\hat{U}_{\textnormal{eq}}$. We also show that this solution has all the properties we desire.

\begin{proposition}\label{prop:analyticsolution}
There exists an open set $V$ containing $h_{\textnormal{eq}}$ such that the equation \eqref{kernelequation} has an analytic solution $\Lambda$ on $V$ satisfying \eqref{initialcondition}.
\end{proposition}

\begin{proof} 
We first transform the equation \eqref{kernelequation} into another equation that is easier to work with. Let $\lambda_1,\dots,\lambda_q$ be the 
eigenvalues of the matrix $[J\theta (h_{\textnormal{eq}})]$. We will prove this proposition under the assumption that all these eigenvalues are real. 
We later remark how the proof changes when they take complex values.

We know from part (B) of Assumption \ref{mainassumptionsonr} that $\lambda_i < 0$ for each $i \in Q$. Pick an $\epsilon_0 \in (0,1)$ such that 
\begin{align}
\label{prop:conditiononepsilon0}
 \lambda_i < -4 \epsilon_0 \textrm{ for all } i \in Q.
\end{align}
Let $M_1 \in \mathbb{M}(q,q)$ be the matrix representing the Jordan canonical form of $[J \theta (h_{\textnormal{eq}})]$. Its diagonal is occupied by 
$\lambda_1,\dots,\lambda_q$, while its super-diagonal entries are either $0$ or $1$. All the other entries are $0$. Let 
$P_1 \in \mathbb{M}(q,q)$ be the invertible matrix such that $P_1 [J\theta (h_{\textnormal{eq}})] P^{-1}_1 = M_1.$
Let $P_2 = \textrm{Diag}(1,\epsilon_0,\epsilon^2_0,\dots,\epsilon^{q-1}_0)$,  $P =  P^{-1}_2 P_1$ and $M = P_2^{-1} M_1 P_2$. Then
\begin{align}
\label{prop:jordancanonicalform}
 P [J \theta (h_{\textnormal{eq}})] P^{-1} = M,
\end{align}
and $M$ is just the matrix $M_1$ with each $1$ on the super-diagonal replaced by $\epsilon_0$.

In this proof, $\bar{0}$ will always denote the vector of zeroes in $\R^q$. 
Since $h_{\textnormal{eq}} \in \R^q_{*}$ (see part (A) of Lemma \ref{lemma:appendix1}) and the map $x \mapsto h_{\textnormal{eq}} + P^{-1} x$ is continuous, we can find a $r_0>0$ such that for any $x \in B^q_{r_0}(\bar{0})$ we have 
$h_{\textnormal{eq}} + P^{-1}x \in \R^q_{*}$, where $B^q_{r_0}(\bar{0})$ is the open ball in $\R^q$ with radius $r_0$ centered at $\bar{0}$. 
For all $x \in B^q_{r_0}(\bar{0})$, let 
$\hat{A}(x) \in \mathbb{M}(q,q)$ and $\hat{\theta}(x) \in \R^q$ be given by 
\begin{align*}
\hat{A}(x) = A^T(h_{\textnormal{eq}} + P^{-1} x)  \textrm{ and } \hat{\theta}(x) = P \theta(h_{\textnormal{eq}} + P^{-1} x). 
\end{align*}
Suppose $ \beta : U \to \R^q$ is a function which is analytic in an open set $U \subset B^q_{r_0}(\bar{0})$ and for all $x\in U$ 
\begin{align}
\label{kernelequation0}
\hat{A}(x) \beta(x) + [J \beta (x)] \hat{\theta}(x) = \bar{0}
\end{align}
along with 
\begin{align}
\label{initialconditionforbeta}
\beta(\bar{0}) = v_{\textnormal{eq}},
\end{align}
where $v_{\textnormal{eq}}$ is given by \eqref{defb:veq}.
If $V \subset \R^q_{+}$ is the image of $U$ under the map $x \mapsto h_{\textnormal{eq}} + P^{-1} x$, then $V$ is an open set containing $h_{\textnormal{eq}}$ and the function 
$\Lambda :  V \to \R^q$ defined by $\Lambda(h) = \beta(P(h-h_{\textnormal{eq}}))$ is an analytic solution to \eqref{kernelequation} satisfying \eqref{initialcondition}.
Hence to prove the proposition it suffices to show that equation \eqref{kernelequation0} has a solution $\beta$ in some neighbourhood of $\bar{0}$ which satisfies \eqref{initialconditionforbeta}.

We will be using the multi-index notation to write the power series in $q$ variables. 
For any multi-index $\alpha = (\alpha_1,\alpha_2,\dots,\alpha_q) \in \mathbb{N}_0^{q}$ let $|\alpha| = \alpha_1+\alpha_2+\dots+\alpha_q$ and 
$\alpha ! = \alpha_1! \alpha_2!\dots \alpha_q!$. 
For two multi-indices $\nu = (\nu_1,\dots,\nu_q) \in \N^q_0$ and $\alpha = (\alpha_1,\dots,\alpha_q) \in \N_0^q$, 
we say that $\nu \leq \alpha$ if $\nu_i \leq \alpha_i$ for all $i=1,\dots,q$ and we say that $\nu < \alpha$ if $\nu \leq \alpha$ and $\nu \neq \alpha$. 
If $\nu \leq \alpha$ then 
\[ \left (\begin{array}{c} \alpha \\ \nu \end{array} \right ) = \frac{\alpha !}{ \nu ! (\alpha - \nu)!} .\]
For any vector $x \in \R^{q}$ and multi-index $\alpha = (\alpha_1,\alpha_2,\dots,\alpha_q) \in \mathbb{N}_0^{q}$ define
\[ x^\alpha = x_1^{\alpha_1} x_2^{\alpha_2}\dots x_q^{\alpha_q}\]
and the differential operator $D_{\alpha}$ as
\[D_{\alpha} = \frac{\partial^{\alpha_1} }{ \partial x_1^{\alpha_1}}\dots  \frac{\partial^{\alpha_q} }{ \partial x_q^{\alpha_q}} .\]
The operator $D_\alpha$ acts component-wise on matrix and vector valued functions.

Consider the function $\beta$ given by the power series
\begin{align}
\label{prop:defbeta}
\beta(x) = v_{\textnormal{eq}} + \sum_{|\alpha| = 1 }^{\infty} \gamma_\alpha x^{\alpha},
\end{align}
where $\gamma_{\alpha} \in \R^q$ is given by
\begin{align}
\label{prop:gammaalphareln}
 \gamma_\alpha = \frac{D_\alpha \beta (0)}{\alpha !}.
\end{align}
This function $\beta$ satisfies \eqref{initialconditionforbeta}. 
To prove the proposition it suffices to show that the vectors $\gamma_{\alpha}$ can be suitably chosen such that $\beta$ satisfies \eqref{kernelequation0} and there exists a 
positive constant $C$ such that 
\begin{align}
\label{prop:conditiononcalpha}
\| \gamma_\alpha \|_{\infty} \leq C^{|\alpha|} \textrm{ for all } \alpha \in \mathbb{N}_0^{q}. 
\end{align}
The last condition ensures the absolute convergence of the power series \eqref{prop:defbeta} in a neighbourhood of $\bar{0}$. 

Since $\hat{\theta}(\bar{0}) = \theta(h_{\textnormal{eq}}) = \bar{0}$, if we plug $x = \bar{0}$ in (\ref{kernelequation0}) we obtain
\begin{align*}
\hat{A}(\bar{0}) \beta(\bar{0}) = A^T(h_{\textnormal{eq}})v_{\textnormal{eq}}= 0.   
\end{align*}
This is satisfied because of the choice of $v_{\textnormal{eq}}$ (see \eqref{defb:veq}).

Applying the operator $D_\alpha$ to equation \eqref{kernelequation0} and using the product rule for multi-derivatives we get 
\begin{align*}
\bar{0}   & =  D_{\alpha} \left( \hat{A}(x) \beta(x) \right)  +  D_{\alpha} \left( [ J \beta (x) ] \hat{\theta}(x) \right) \\ 
          & =  \sum_{\nu \leq \alpha} \left (\begin{array}{c} \alpha \\ \nu \end{array} \right )   \left( D_{(\alpha-\nu)} \hat{A}(x)  \right) 
\left(  D_{\nu}  \beta (x)\right) +  \sum_{\nu \leq \alpha}  \left (\begin{array}{c} \alpha \\ \nu \end{array} \right ) 
\left( D_{\nu} [ J\beta(x)] \right) \left(   D_{(\alpha - \nu)} \hat{\theta} (x)\right) \\
& =  \hat{A}(x)   D_{\alpha}\beta(x)  + \left( D_{\alpha}  [J \beta (x)] \right) \hat{\theta}(x) 
+  \sum_{\{\nu < \alpha, |\alpha-\nu|=1\}} \left (\begin{array}{c} \alpha \\ \nu \end{array} \right )  \left(  D_{\nu} [ J \beta (x) ] \right) 
 \left( D_{(\alpha - \nu)}  \hat{\theta}(x) \right)\\& 
+ \sum_{\nu < \alpha} \left (\begin{array}{c} \alpha \\ \nu \end{array} \right ) \left( D_{(\alpha-\nu)}  \hat{A}(x) \right)\left( D_{\nu}\beta(x) \right)  
+ \sum_{\{\nu < \alpha, |\alpha-\nu|>1\}} \left (\begin{array}{c} \alpha \\ \nu \end{array} \right ) \left( D_{\nu}   [J\beta (x)] \right) 
\left(  D_{(\alpha - \nu)} \hat{\theta}(x) \right).
\end{align*}
On rearranging we obtain
\begin{align}
\label{prop:mainproductruleequation}
&\hat{A}(x)  D_{\alpha}\beta(x)  + \left( D_{\alpha} [J \beta (x)] \right) \hat{\theta}(x) 
+  \sum_{\{\nu < \alpha, |\alpha-\nu|=1\}} \left (\begin{array}{c} \alpha \\ \nu \end{array} \right ) \left( D_{\nu}  [ J\beta(x)] \right) 
 \left( D_{(\alpha - \nu)}  \hat{\theta}(x) \right) \notag \\
& =  - \sum_{\nu < \alpha} \left (\begin{array}{c} \alpha \\ \nu \end{array} \right ) \left( D_{(\alpha-\nu)} \hat{A}(x) \right) 
\left( D_{\nu} \beta(x) \right)  -\sum_{\{\nu < \alpha, |\alpha-\nu|>1\}} \left (\begin{array}{c} \alpha \\ \nu \end{array} \right ) \left( D_{\nu}  [J \beta (x)] \right) 
 \left( D_{(\alpha - \nu)}  \hat{\theta}(x) \right).
\end{align}
For any $j \in Q $, let $e_j  \in \mathbb{N}_0^{q}$ be the multi-index $(0,\dots,0,1,0,\dots,0)$, with the $1$ at the $j$-th position. 
Observe that if $|\alpha - \nu| = 1 $ then $\nu = \alpha - e_j$ for some $j \in Q$. Therefore
\begin{align*}
\sum_{\{\nu < \alpha, |\alpha-\nu|=1\}} \left (\begin{array}{c} \alpha \\ \nu \end{array} \right ) \left(  D_{\nu}   [ J \beta (x)] \right) 
\left( D_{(\alpha - \nu)} \hat{\theta}(x) \right) &= \sum_{j \in Q} \alpha_j  \left( D_{(\alpha - e_j)}   [ J \beta (x)] \right) \left(  \partial_j \hat{\theta}(x) \right) \\
& = \sum_{j , k  \in Q } \alpha_j \left( \partial_j \hat{\theta}_k(x) \right) \left( D_{(\alpha - e_j+ e_k)} \beta (x)\right) \\ 
& = \sum_{j , k  \in Q } \alpha_j [J \hat{\theta} (x)]_{kj}  \left( D_{(\alpha - e_j+ e_k)} \beta (x) \right). 
\end{align*}
Note that $[J \hat{\theta} (\bar{0})] = M$ (see \eqref{prop:jordancanonicalform}). 
This matrix has the eigenvalues $\lambda_1,\dots,\lambda_q$ on the diagonal and either $0$ or $\epsilon_0$ on the super-diagonal. 
For each $j = 2,\dots,q$ let $\epsilon_j = \epsilon_0$ if $M_{(j-1)j} = \epsilon_0$ and $\epsilon_j = 0$ otherwise. Then for $x = \bar{0}$ we obtain
\begin{align*}
&\sum_{\{\nu < \alpha, |\alpha-\nu|=1\}} \left (\begin{array}{c} \alpha \\ \nu \end{array} \right ) \left(  D_{\nu}   [ J \beta (\bar{0})] \right) 
\left( D_{(\alpha - \nu)} \hat{\theta}(\bar{0}) \right) \\&= 
\sum_{j \in Q } \alpha_j \lambda_j D_{ \alpha } \beta (\bar{0}) +\sum_{j = 2}^q \alpha_j  \epsilon_j  D_{(\alpha - e_j+ e_{j-1})} \beta (\bar{0}).
\end{align*}
Note that $\hat{\theta}(\bar{0}) = P \theta(h_{\textnormal{eq}}) = \bar{0}$ and for each $\alpha \in \N^q_0$, $\gamma_\alpha$ is given by \eqref{prop:gammaalphareln}. 
We plug $x = \bar{0}$ in \eqref{prop:mainproductruleequation} and divide by $\alpha !$ to get
\begin{align}
\label{prop:mainproductruleequation2}
 \hat{A}(\bar{0}) \gamma_{\alpha} & + \sum_{j \in Q } \alpha_j \lambda_j \gamma_{ \alpha } +  
\sum_{j = 2 , \alpha_j > 0}^q (\alpha_{j-1}+1) \epsilon_j  \gamma_{(\alpha - e_j+ e_{j-1})} = Y_{\alpha} \\  
\textrm{ where } Y_{\alpha} &=   - \sum_{\nu < \alpha}  \frac{ \left( D_{(\alpha-\nu)}  \hat{A}(\bar{0})\right) }{ (\alpha- \nu) !} \gamma_{\nu}  
-\sum_{\{\nu < \alpha, |\alpha-\nu|>1\}} \frac{ \left( D_{\nu}   [J \beta (\bar{0})] \right)}{ \nu ! } 
\frac{  \left(  D_{(\alpha - \nu)}\hat{\theta}(\bar{0}) \right)}{(\alpha-\nu)!}  \notag.
\end{align}
The second term can be simplified as 
\begin{align*}
\sum_{\{\nu < \alpha, |\alpha-\nu|>1\}} \frac{ \left(   D_{\nu} [J \beta (\bar{0})] \right)}{ \nu ! } 
\frac{ \left( D_{(\alpha - \nu)} \hat{\theta}(\bar{0}) \right)}{(\alpha-\nu)!}
& =  \sum_{\{\nu < \alpha, |\alpha-\nu|>1\}} \sum_{j \in Q } \frac{ \left( D_{(\alpha - \nu)} \hat{\theta}_j (\bar{0}) \right)}{(\alpha-\nu)!}
\frac{ \left( D_{(\nu + e_j)} \beta (\bar{0}) \right)}{ \nu ! }\\
& =  \sum_{\{\nu < \alpha, |\alpha-\nu|>1\}} \sum_{j \in Q} (\nu_j+1)\frac{ \left( D_{(\alpha - \nu)} \hat{\theta}_j (\bar{0}) \right)}{(\alpha-\nu)!} \gamma_{(\nu + e_j)}. 
\end{align*}
Therefore we can write $Y_{\alpha}$ as 
\begin{align}
\label{prop:defnyalpha}
Y_{\alpha} &=   - \sum_{\nu < \alpha}  \frac{ \left( D_{(\alpha-\nu)}  \hat{A}(\bar{0})\right) }{ (\alpha- \nu) !} \gamma_{\nu}  
-\sum_{\{\nu < \alpha, |\alpha-\nu|>1\}} \sum_{j \in Q} (\nu_j+1)\frac{ \left( D_{(\alpha - \nu)} \hat{\theta}_j (\bar{0}) \right)}{(\alpha-\nu)!} \gamma_{(\nu + e_j)}. 
\end{align}

For each $k \in \N$, let $S_k$ be the set of multi-indices given by $S_k = \left\{ \alpha \in \N_0^q : |\alpha| = k \right\}$. The number of elements in $S_k$ is 
\[s_k = \left (\begin{array}{c} k + q -1 \\ q-1 \end{array} \right ). \] 
We order the multi-indices in $S_k$ as follows.  We say that $\nu \preceq \alpha$ if and only if $\sum_{i \in Q } i \nu_i \leq \sum_{i \in Q} i \alpha_i$.
Let $\alpha^k(1),\dots,\alpha^k(s_k)$ be all the elements of $S_k$ listed in the order given by $\preceq$.

Let the matrix $\Xi^{(k)} \in \mathbb{M}(q s_k, q s_k)$ be a block matrix composed of $s_k^2$ blocks of size $q \times q$. For each $i,j \in \{1,2,\dots,s_k\}$ 
the block starting at row $q(i-1)+1$ and column $q(j-1)+1$ of matrix $\Xi^{(k)}$ is occupied by the matrix $L_{ij} \in \mathbb{M}(q,q)$ defined as follows.
If $i = j$ then $L_{ii} = \hat{A}(\bar{0}) + \left( \sum_{l \in Q} \alpha^k_l(i) \lambda_l \right) I_q$. If $i$ and $j$ are such that 
 $\alpha^k(j) = \alpha^k(i) -e_l + e_{l-1}$ for some $l \in \{2,\dots,q\}$ then $L_{ij} = \epsilon_l (\alpha^k_{l-1}(i)+1) I_q$. For every other $i$ and $j$, $L_{ij}$ is just 
a matrix of zeroes. The matrix $\Xi^{(k)}$ is lower block-triangular and its determinant is given by
\begin{align*}
\textrm{Det}\left( \Xi^{(k)} \right)= \prod_{ i = 1}^{s_k } \textrm{Det}\left( \hat{A}(\bar{0}) + \left( \sum_{l \in Q} \alpha^k_l(i) \lambda_l \right) I_q  \right).
\end{align*}
The eigenvalues $\lambda_1,\dots,\lambda_q$ satisfy \eqref{prop:conditiononepsilon0}. 
Since all the eigenvalues of the matrix $\hat{A}(\bar{0}) = A^T(h_{\textnormal{eq}})$ have non-positive real parts (see part (B) of Lemma \ref{lemma:appendix1}), 
the above determinant is non-zero. Hence the matrix $\Xi^{(k)}$ is invertible.

Let $X^{(k)}$ and $Y^{(k)}$ be the vectors in $\R^{qs_k}$ given by
\begin{align*}
 X^{(k)} = \left( \gamma_{\alpha^k(1)},\gamma_{\alpha^k(2)},\dots, \gamma_{\alpha^k(s_k)} \right) \textrm{ and }  Y^{(k)} = \left( Y_{\alpha^k(1)},Y_{\alpha^k(2)},\dots, Y_{\alpha^k(s_k)}  \right).
\end{align*}
Using \eqref{prop:mainproductruleequation2} we obtain the following linear system  
\[\Xi^{(k)} X^{(k)} = Y^{(k)}\]
and since the matrix $\Xi^{(k)}$ is invertible
\begin{align} 
\label{prop:defxk}
X^{(k)} = [ \Xi^{(k)}]^{-1} Y^{(k)}. 
\end{align}
Note that $Y^{(k)}$ only depends on $\{ \gamma_{\alpha} : \alpha \in S_{l} \textrm{ for } l \in \{0,1,\dots,k-1\} \}$. Hence for each $k \in \N$ 
we can solve for the whole set $\{ \gamma_\alpha : \alpha \in S_k\}$ using \eqref{prop:defxk}. Doing this iteratively for each $k$ we can solve for 
$\gamma_\alpha$ for all $\alpha \in \N^q_0$. The function $\beta$ given by \eqref{prop:defbeta} with this choice of $\gamma_\alpha$'s will solve \eqref{kernelequation0} 
in a neighbourhood of $\bar{0}$ if we can show that \eqref{prop:conditiononcalpha} holds for some $C>0$. Showing this will be our next task.

Any entry on the diagonal of $\Xi^{(k)}$ has the form $\hat{A}_{ii}(\bar{0}) + \sum_{j \in  Q} \lambda_j \alpha_j$
for some $\alpha \in S_k$ and $i \in Q$. Observe that $\hat{A}(\bar{0}) = A^T(h_{\textnormal{eq}})$ and this matrix only has non-positive entries on its diagonal 
(see \eqref{defninteractionmatrix}).
From \eqref{prop:conditiononepsilon0}, for $\alpha \in S_k$ we obtain the estimate
\begin{align}
\label{diagdominant1}
 \left| \hat{A}_{ii}(\bar{0}) + \sum_{j \in Q} \lambda_j \alpha_j \right| \geq 4 \epsilon_0 k.
\end{align}
For each row of $\Xi^{(k)}$, the sum of the absolute values of the non-diagonal entries is bounded above by
\begin{align}
\label{diagdominant2}
 \max_{i \in Q}\left| \sum_{j \in Q, j \neq i} \hat{A}_{ij}(\bar{0}) + \sum_{l=2}^q \epsilon_l  (\alpha_{l-1}+1)\right| \leq 
\max_{i \in Q}\left( \sum_{j \in Q, j \neq i} |\hat{A}_{ij}(\bar{0})| \right) + \epsilon_0 (k+q).
\end{align}
Hence from \eqref{diagdominant1} and \eqref{diagdominant2} we can conclude that there exists a 
$K_0 \in \N$ such that for all $k \geq K_0$ the matrix $\Xi^{(k)}$ is strictly diagonally dominant and we have
\begin{align*}
\min_{1 \leq l \leq  q s_k} \left| \left| [\Xi^{(k)}]_{ll} \right| -  \sum_{r = 1 , r \neq l}^{q s_k} \left| [\Xi^{(k)}]_{lr} \right|   \right|  \geq k\epsilon_0.
\end{align*}
Theorem 1 in Varah \cite{Varah} shows that for all $k \geq K_0$
\begin{align}
\label{varahsbound}
 \left\| [\Xi^{(k)}]^{-1}\right\|_{\infty} \leq \frac{1}{k \epsilon_0}. 
\end{align}

Part (D) of Assumption \ref{mainassumptionsonr} says that for each $i,j \in Q$, the functions $\rho_i$ and $\beta_{ij}$ are analytic in a neighbourhood of $h_{\textnormal{eq}}$. 
This implies that there is a neighbourhood $U$ of $\bar{0}$ such that the $\mathbb{M}(q,q)$-valued function $\hat{A}$ and the $\R^q$-valued function $\hat{\theta}$ are 
analytic component-wise on $U$. Therefore there is a constant $C_0$ such that  
\begin{align}
\label{prop:conditiononhats}
\| D_{\alpha} \hat{A}(\bar{0}) \|_{\infty} \leq C_0^{|\alpha|} \alpha! \textrm{ and } \| D_{\alpha} \hat{\theta}(\bar{0}) \|_{\infty} \leq C_0^{|\alpha|} \alpha!  
\textrm{ for all } \alpha \in \mathbb{N}_0^{q}. 
\end{align}
We can assume that $C_0>1$. Choose a $\delta>0 $ satisfying
\begin{align}
 \label{prop:conditionondelta}
\delta < \left(  \frac{\epsilon_0}{ C_0 q(q+1) 2^{q+3}}   \right)
\end{align}
and define $C = C_0 / \delta$. We will prove \eqref{prop:conditiononcalpha} by induction. Let $k > K_0$ and suppose that $C$ is large enough to satisfy
\begin{align}
\label{prop:inductionhyp}
\left\| \gamma_\nu\right\|_{\infty} \leq C^{|\nu|} 
\end{align}
for all $l \in \{1,2,\dots,k-1\}$ and $\nu \in S_l$. To prove \eqref{prop:conditiononcalpha} we need to show that 
$\left\| \gamma_\alpha\right\|_{\infty} \leq C^k$ for all $\alpha \in S_k$. This is equivalent to showing that 
$\left\| X^{(k)}\right\|_{\infty} \leq C^k$. From \eqref{prop:defxk} and \eqref{varahsbound} we have
\begin{align*}
\left\|X^{(k)} \right\|_{\infty} &\leq \left\|[\Xi^{(k)}]^{-1} \right\|_{\infty} \left\| Y^{(k)} \right\|_{\infty} \leq \frac{1}{k \epsilon_0} \left\| Y^{(k)} \right\|_{\infty}.  
\end{align*}
Hence to prove \eqref{prop:conditiononcalpha} it suffices to show that
\begin{align}
\label{prop:boundonyks}
 \left\| Y^{(k)} \right\|_{\infty} = \max_{\alpha \in S_k} \left\| Y_\alpha\right\|_{\infty} \leq k \epsilon_0 C^k.
\end{align}
From \eqref{prop:defnyalpha}, \eqref{prop:conditiononhats} and \eqref{prop:inductionhyp}, for any $\alpha \in S_k$  we get 
\begin{align*}
\left\| Y_{\alpha}\right\|_{\infty} & \leq  \sum_{\nu < \alpha}  C_0^{|\alpha-\nu|} C^{|\nu|}+ 
\sum_{\{\nu < \alpha, |\alpha-\nu|>1\}} \sum_{j \in Q} (\nu_j+1) C_0^{|\alpha-\nu|}  C^{|\nu|+1}  \\
& = \sum_{\nu < \alpha}  C_0^{|\alpha-\nu|} C^{|\nu|}+ \sum_{\{\nu < \alpha, |\alpha-\nu|>1\}} (|\nu|+q) C_0^{|\alpha-\nu|}  C^{|\nu|+1}. 
\end{align*}
But $C = C_0/ \delta$ and $|\alpha| = k$. Hence
\begin{align}
\left\| Y_{\alpha}\right\|_{\infty} & \leq 
  C^{k} \left( \sum_{\nu < \alpha} \delta^{|\alpha - \nu|}   + C \sum_{\{\nu < \alpha, |\alpha-\nu|>1\}} (|\nu|+q) \delta^{|\alpha-\nu|} \right) \notag \\
& \leq C^{k} \left( \sum_{\nu < \alpha} \delta^{|\alpha - \nu|}   + 2 k C \sum_{\{\nu < \alpha, |\alpha-\nu|>1\}} \delta^{|\alpha-\nu|} \right). \label{prop:boundonyks2}
\end{align}
Note that
\begin{align*}
1 +  \sum_{\nu < \alpha} \delta^{|\alpha-\nu|}  = \sum_{\nu \leq \alpha} \delta^{|\alpha-\nu|}  =  
& \sum_{\nu_1 = 0 }^{\alpha_1} \sum_{\nu_2 = 0 }^{\alpha_2}\dots \sum_{\nu_q = 0 }^{\alpha_q}  \prod_{i=1}^{q} \delta^{(\alpha_i-\nu_i)}
\\& = \prod_{i=1}^{q} \left( \frac{1- \delta^{\alpha_i+1}}{1-\delta}\right) \\&  =  \prod_{i=1, \alpha_i > 0}^{ q } \left( \frac{1- \delta^{\alpha_i+1}}{1-\delta}\right)  
\end{align*}
and this shows that
\begin{align}
\label{taylorestimate1}
\sum_{\nu < \alpha} \delta^{|\alpha-\nu|}  \leq (1-\delta)^{-n(\alpha)} - 1, 
\end{align}
where $n(\alpha)$ be the number of non-zero coordinates of $\alpha$. Similarly
\begin{align}
\label{taylorestimate2}
\sum_{\{\nu < \alpha, |\alpha-\nu|>1\}} \delta^{|\alpha-\nu|}  =   \sum_{\nu \leq \alpha } \delta^{|\alpha-\nu|}  - 1 - n(\alpha) \delta \leq (1-\delta)^{-n(\alpha)} - 1 - 
n(\alpha) \delta.  
\end{align}
Since $\delta \in (0,1/2)$ and $n(\alpha)\leq q$, by Taylor's theorem we see that
\begin{align*}
 (1-\delta)^{-n(\alpha)} - 1 \leq q 2^{q+1} \delta \quad \textrm{ and } \quad \sum_{\nu \leq \alpha } \delta^{|\alpha-\nu|}  - 1 - n(\alpha) \delta \leq q(q+1)2^{q+1} \delta^2. 
\end{align*}
Using these estimates, \eqref{taylorestimate1}, \eqref{taylorestimate2} and \eqref{prop:boundonyks2} we get
\begin{align*}
\left\| Y_{\alpha}\right\|_{\infty} & \leq C^{k} \left( q 2^{q+1} \delta  + 2 k C q(q+1) 2^{q+1} \delta^2  \right)= C^{k} \delta k \left( C_0 q(q+1) 2^{q+3}   \right). 
\end{align*}
But $\delta$ satisfies \eqref{prop:conditionondelta} which shows \eqref{prop:boundonyks} and completes the proof of the proposition.

At the beginning of the proof, we had assumed that the eigenvalues $\lambda_1,\dots,\lambda_q$ of the matrix $[J \theta(h_{\textnormal{eq}})]$ are all real-valued. 
If that is not true then the invertible matrix $P$ that appears in \eqref{prop:jordancanonicalform} has complex entries. Let $\mathbb{C}$ be the field of complex numbers.
Define a map $\phi : \R^q \to \mathbb{C}^q$ by $\phi(h) = P(h - h_{\textnormal{eq}})$. 
The image of this map, denoted by $\phi(\R^q)$, sits as a $q$-dimensional real vector space in $\mathbb{C}^q$. 
The map $\phi$ is an infinitely differentiable isomorphism between $\R^q$ and $\phi(\R^q)$ and using this we can define derivatives of real-valued functions over 
$\phi(\R^q)$. As above, we can obtain an analytic solution $\beta$ of \eqref{kernelequation0} satisfying \eqref{initialconditionforbeta}, defined on some open set $U$ in 
$\phi(\R^q)$ containing $\bar{0}$. On $V = \phi^{-1}(U)$, the function $\Lambda$ defined by $\Lambda(h) = \beta(\phi(h))$ will then be an analytic solution to \eqref{kernelequation} 
satisfying \eqref{initialcondition}.
\end{proof}

The above proposition provides us with an analytic solution to \eqref{kernelequation} in a neighbourhood of $h_{\textnormal{eq}}$. Our next task is to extend it to a solution 
in $C^2(\hat{U}_{\textnormal{eq}}, \R^q_*)$ where $\hat{U}_{\textnormal{eq}}$ is an open set in $\R^q$ containing $U_{\textnormal{eq}}$.

Recall from Section \ref{density_regulation_mechanism} that for all $i,j \in Q$, $\beta_{ij}, \rho_i$ are functions in $C^2(\R^q_+, \R_+)$. 
Let $O \subset \R^q$ be the open set containing $\R^q_+$ defined by
\begin{align*}
O = \left\{ h \in \R^q : h_i > -1 \textrm{ for all } i= 1,\dots,q \right\}. 
\end{align*}
Then we can extend the functions $\beta_{ij}, \rho_i$ to functions $\hat{\beta}_{ij}, \hat{\rho}_i \in C^2(\R^q, \R_+)$ such that 
$\hat{\beta}_{ij}(h) = 0$ and $\hat{\rho}_i(h) = 0$ for all $h \notin O$. Moreover since each $\beta_{ij}$ is bounded, we can make sure that its extension 
$\hat{\beta}_{ij}$ is also bounded. For each $h \in \R^q$ let $\hat{A}(h) \in \mathbb{M}(q,q)$ be the matrix defined by \eqref{defninteractionmatrix} with 
$\beta_{ij}, \rho_i$ replaced by $\hat{\beta}_{ij}, \hat{\rho}_i$. Also let $\hat{\theta} \in C^2(\R^q,\R^q)$ be the function given by
\begin{align}
\label{defnthetahat}
\hat{\theta}(h) = \hat{A}(h)h \textrm{ for } h \in \R^q. 
\end{align}
Corresponding to $\hat{\theta}$ we can define the \emph{flow} map $\hat{\psi} \in C^2(\R^q \times \R_+,O)$ as the unique solution to the equation analogous to 
\eqref{flowequation}, with $\theta$ replaced by $\hat{\theta}$. 
Define the region of attraction of the fixed point $h_{\textnormal{eq}}$ as
\begin{align*}
\hat{U}_{\textnormal{eq}} = \left\{ h \in O : \lim_{t \to \infty} \hat{\psi}(h,t) = h_{\textnormal{eq}}  \right\}. 
\end{align*}
Then $\hat{U}_{\textnormal{eq}}$ is an open set in $\R^q$ (see Lemma 3.2 in \cite{Khalil}) containing $U_{\textnormal{eq}}$.
\begin{proposition}
\label{prop:extension}
There exists a solution $\Lambda \in C^2(\hat{U}_{\textnormal{eq}},\R^q_{+})$ of \eqref{kernelequation} satisfying \eqref{innerproductcondition} and \eqref{initialcondition}.
\end{proposition}
\begin{proof}
Suppose that $U \subset \hat{U}_{\textnormal{eq}}$ is any $\hat{\psi}$-invariant open set and the function $\Lambda \in C^2(U, \R^q)$ 
satisfies \eqref{kernelequation} and \eqref{initialcondition}. We first show that this function automatically satisfies \eqref{innerproductcondition} on $U$. 
Using \eqref{kernelequation} and the $\hat{\psi}$-invariance of $U$ we get 
\begin{align}
\label{lemma:eqnforlambda}
\frac{d }{ d t } \Lambda(\hat{\psi}(h,t)) & = [J \Lambda (\hat{\psi}(h,t))] \hat{\theta}(\hat{\psi}(h,t)) = - \hat{A}^T(\hat{\psi}(h,t)) \Lambda(\hat{\psi}(h,t)).	
\end{align}
Observe that
\begin{align*}
\frac{d }{ d t } \langle \hat{\psi}(h,t) , \Lambda(\hat{\psi}(h,t)) \rangle &= \left\langle \frac{d  }{d t} \hat{\psi}(h,t)  , \Lambda(\hat{\psi}(h,t)) \right\rangle 
+ \left\langle  \hat{\psi}(h,t)  ,\frac{d}{d t}\Lambda(\hat{\psi}(h,t))  \right\rangle\\
& =  \left\langle \hat{\theta}\left(\hat{\psi}(h,t)\right) , \Lambda(\hat{\psi}(h,t)) \right\rangle - \left\langle  \hat{\psi}(h,t)  , \hat{A}^T(\hat{\psi}(h,t)) \Lambda(\hat{\psi}(h,t)) \right\rangle\\
& = \left\langle \hat{\theta}\left(\hat{\psi}(h,t)\right) , \Lambda(\hat{\psi}(h,t)) \right\rangle - \left\langle \hat{A}(\hat{\psi}(h,t)) \hat{\psi}(h,t) , \Lambda(\hat{\psi}(h,t)) \right\rangle \\
& = 0,
\end{align*}
where the last equality holds due to \eqref{defnthetahat}. This shows that for any fixed $h \in U$ the function 
$ \langle \hat{\psi}(h,t) , \Lambda(\hat{\psi}(h,t)) \rangle$ is a constant function of time. Therefore \eqref{initialcondition} implies that for any $h \in U$
\begin{align*}
\langle h, \Lambda(h) \rangle = \lim_{t \to \infty}\langle \hat{\psi}(h,t) , \Lambda(\hat{\psi}(h,t)) \rangle  = \langle h_{\textnormal{eq}} , \Lambda(h_{\textnormal{eq}})\rangle = 1. 
\end{align*}
This proves that $\Lambda$ satisfies \eqref{innerproductcondition} on $U$. For any $h \in U$ and $0 \leq t \leq t_0$, let $\Phi(h,t,t_0)$ be the matrix defined in 
Lemma \ref{lemma:matrixpsi}. Since $\Lambda$ satisfies \eqref{lemma:eqnforlambda} we must have
\begin{align}
\label{reln:lambdapsi}
\Lambda(\hat{\psi}(h,t)) = \Phi(h,t,t_0)\Lambda(\hat{\psi}(h,t_0)). 
\end{align}

From Proposition \ref{prop:analyticsolution} we know that on some open set $V \subset \R^q$ containing $h_{\textnormal{eq}}$ we can find a solution 
$\bar{\Lambda} \in C^2(V,\R^q)$ that satisfies \eqref{kernelequation} along with \eqref{initialcondition}. Since $v_{\textnormal{eq}} \in \R^q_*$ 
(that is, it is positive component-wise) and $\bar{\Lambda}$ is a continuous function, by shrinking $V$ if necessary, we can ensure that the 
image of $V$ under $\bar{\Lambda}$ lies in $\R^q_*$. Since $V$ is open, there exists a $r \in (0,1)$ such that $B^q_r(h_{\textnormal{eq}}) \subset V$, where 
$B^q_r(h_{\textnormal{eq}})$ is the open ball in $\R^q$ centered at $h_{\textnormal{eq}}$ with radius $r$. As in the proof of Lemma \ref{compact_containment}, we can find a $\hat{\psi}$-invariant open set $W \subset B^q_r(h_{\textnormal{eq}})$ which contains $h_{\textnormal{eq}}$.

For each $n \in \N$ define an open set 
\[O_n = \{ h \in \hat{U}_{\textnormal{eq}}  : \hat{\psi}(h,n) \subset W\}.\]
Each $O_n$ is $\hat{\psi}$-invariant. Furthermore $W \subset O_1 \subset O_2 \dots $ and $\bigcup_{n=1}^{\infty} O_n = \hat{U}_{\textnormal{eq}}$.
Define $\lambda_n(h,t)$ for each $h \in O_n$ and $t \in [0,n)$ by
\begin{align}
\label{deflambdan}
 \lambda_n(h,t) = \Phi(h,t,n) \bar{\Lambda}(\hat{\psi}(h,n)). 
\end{align}
Observe that $\hat{\psi}(h,n) \in W \subset V$ and so $\bar{\Lambda}(\hat{\psi}(h,n))$ is well-defined and also $\bar{\Lambda}(\hat{\psi}(h,n)) \in \R^q_{*}$. 
Part (C) of Lemma \ref{lemma:matrixpsi} shows that $\lambda_n(h,t) \in \R^q_{*}$.
Since $\bar{\Lambda} \in C^2(V,\R^q_*)$, 
$\Phi(\cdot,\cdot,n) \in C^2(\hat{U}_{\textnormal{eq}} \times [0,n], \mathbb{M}_\R(q,q))$ (see Lemma \ref{lemma:matrixpsi}) and 
$\hat{\psi} \in C^2(\R^q \times \R_+, \R^q)$ we must have that $\lambda_n \in C^2(O_n \times [0,n) , \R^q_*)$.
Note that 
$\bar{\Lambda}$ satisfies \eqref{reln:lambdapsi} for all $h \in W$ and so for $0\leq t \leq t_0$
\begin{align}
\label{reln:lambdapsi2}
\bar{\Lambda}(\hat{\psi}(h,t)) = \Phi(h,t,t_0)\bar{\Lambda}(\hat{\psi}(h,t_0)). 
\end{align}
Therefore if $h \in W$, then for any $n \in \N$ and $t \geq 0$ we have
\begin{align}
\label{consistency1}
 \lambda_n(h,t) = \bar{\Lambda}(\hat{\psi}(h,t)).
\end{align}
Using parts (A) and (B) of Lemma \ref{lemma:matrixpsi}, \eqref{reln:lambdapsi2} and the semigroup property of $\hat{\psi}$ (similar to \eqref{semigroup_property}) 
we can also see that for any $h \in O_n$
\begin{align*}
\lambda_n(\hat{\psi}(h,t),0) & = \Phi(\hat{\psi}(h,t),0,n) \bar{\Lambda}\left(\hat{\psi}(\hat{\psi}(h,t),n) \right) \\
				& = \Phi(h,t,n+t) \bar{\Lambda} \left(\hat{\psi}(h,n+t) \right) \\
			        & = \Phi(h,t,n) \Phi(h,n,n+t) \bar{\Lambda}(\hat{\psi}(h,n+t)) \\
				& = \Phi(h,t,n)  \bar{\Lambda}(\hat{\psi}(h,n)) \\
				& = \lambda_n(h,t).
\end{align*}
Let $h \in O_n$ and $m \geq n$. Then $\hat{\psi}(h,n) \in W$. From part (A) of Lemma \ref{lemma:matrixpsi} and \eqref{reln:lambdapsi2} we can deduce that for any 
$t \in [0,n)$ 
\begin{align*}
\lambda_m(h,t) = \Phi(h,t,m) \bar{\Lambda}(\hat{\psi}(h,m))  &= \Phi(h,t,n) \Phi(h,n,m) \bar{\Lambda}( \hat{\psi}( h,m)) 
\\& = \Phi(h,t,n) \bar{\Lambda}(\hat{\psi}(h,n)) \\&= \lambda_n(h,t).
 \end{align*}
Hence if we define the map $\lambda : \hat{U}_{\textnormal{eq}} \times \R_+ \to \R^q_{*}$ by
\[ \lambda(h,t) = \lambda_n(h,t) \textrm{ if } (h,t) \in O_n \times [0,n),\]
then $\lambda$ is a well-defined function in $C^2(\hat{U}_{\textnormal{eq}} \times \R_+ , \R^q_{*})$ which satisfies 
\begin{align}
\label{cond1forlambdaht}
\lambda(h,t)  = \lambda(\hat{\psi}(h,t),0) \textrm{ for all } (h,t) \in \hat{U}_{\textnormal{eq}} \times \R_+.
\end{align}
From \eqref{deflambdan} and the definition of the matrix $\Phi$ we can see that
\begin{align}
\label{cond2forlambdaht}
\frac{d \lambda (h,t)}{ d t} = - \hat{A}^T(\hat{\psi}(h,t))\lambda (h,t). 
\end{align}

Define $\Lambda : \hat{U}_{\textnormal{eq}} \to \R^q_{*}$ by
\begin{align*}
\Lambda(h) = \lambda(h,0). 
\end{align*}
Then this map is in $C^2(\hat{U}_{\textnormal{eq}},\R^q_*)$ and \eqref{cond1forlambdaht} implies that for any $(h,t) \in \hat{U}_{\textnormal{eq}} \times \R_+$
\begin{align*}
\frac{d \lambda (h,t)}{ d t}  = \frac{d \lambda (\hat{\psi}(h,t),0)}{ d t} =  \frac{d }{ d t } \Lambda(\hat{\psi}(h,t)) = 
[J \Lambda (\hat{\psi}(h,t))] \frac{d \hat{\psi}(h,t)}{ d t} = [J \Lambda (\hat{\psi}(h,t))] \hat{\theta}(\hat{\psi}(h,t)). 
\end{align*}
Using \eqref{cond2forlambdaht} we obtain
\begin{align*}
 [J \Lambda (\hat{\psi}(h,t))] \hat{\theta}(\hat{\psi}(h,t)) = - \hat{A}^T(\hat{\psi}(h,t))\lambda (h,t) =  - \hat{A}^T(\hat{\psi}(h,t))\Lambda( \hat{\psi}(h,t)). 
\end{align*}
If we set $t=0$ then we see that $\Lambda$ is a solution to \eqref{kernelequation}.
Since $\bar{\Lambda}$ satisfies \eqref{initialcondition}, equation \eqref{consistency1} implies that $\Lambda$ will also satisfy it. 
We have already shown that such a solution of \eqref{kernelequation} will automatically satisfy \eqref{innerproductcondition} for all $h \in \hat{U}_{\textnormal{eq}}$. 
This completes the proof of the proposition. 
\end{proof}

\subsection{Fleming-Viot convergence}\label{section:fvconvergence}

In this section we will finally prove the main result of our paper, which is Theorem \ref{mainresult}. Let $\{ \mu^N(t): t \geq 0 \}$ be a $\mathcal{M}^q_{N,a}(E)$-valued process with generator $\mathbf{A}^N_l$ for some $l \in \{0,1,2,3\}$. 
As outlined at the end of Section \ref{mainresultsection}, we first \emph{extract} a $\mathcal{P}(E)$-valued process $\{ \nu^N(t): t \geq 0\}$ from the process $\{ \mu^N(t): t \geq 0 \}$.
This step requires a solution $\Lambda$ of \eqref{kernelequation} whose existence was shown in Section \ref{section:pdesolution}. We then show that as $N \to \infty$, we have $\nu^N \Rightarrow \nu$ where $\{\nu(t) : t \geq 0\}$ is an appropriately defined Fleming-Viot process. 
This convergence and Proposition \ref{prop:constancyofh_mainresult} prove Theorem \ref{mainresult}. Before we proceed we need some preliminary results.

For any set $A \subset \R^q_{+}$ define 
\begin{align*}
\mathcal{M}_F^q(E : A ) = \left\{ \mu \in \mathcal{M}_F^q(E) : H(\mu) \in A \right\}. 
\end{align*}
Note that if $A$ is a compact set then the set $\mathcal{M}_F^q(E : A )$ is also compact.

Recall the definition of the set $U_{\textnormal{eq}}$ from \eqref{defueq}. Let $\{F_N : N \in \N \}$ be a sequence of real-valued functions on $\mathcal{M}_F^q(E : U_{\textnormal{eq}} )$. 
We will say that this sequence belongs to class $o(N^{-m})$ for some $m \in \N_0$, if and only if for each compact $K \subset U_{\textnormal{eq}}$ we have
\begin{align*}
\limsup_{N \to \infty} \sup_{\mu \in \mathcal{M}_F^q(E : K)} N^m \left| F_N(\mu)\right| = 0.
\end{align*}
For two such sequences $\{F_N : N \in \N  \}$ and $\{G_N : N \in \N\}$, we say that $F_N(\mu) = G_N(\mu) + o(N^{-m})$ for all 
$\mu \in \mathcal{M}_F^q(E : U_{\textnormal{eq}})$ if and only if the sequence of functions $\{(F_N -G_N) : N \in \N\}$ is in the class $o(N^{-m})$.

From now on let $\Lambda \in C^2(\hat{U}_{\textnormal{eq}},\R^q_{*})$ be a function that satisfies \eqref{kernelequation}, \eqref{innerproductcondition} and \eqref{initialcondition} on 
some open set $\hat{U}_{\textnormal{eq}} \subset \R^q$ containing $U_{\textnormal{eq}}$. Such a function exists by Proposition \ref{prop:extension}. 
Define a continuous map $\Gamma : \mathcal{M}_F^q( E : U_{\textnormal{eq}} ) \to \mathcal{P}(E)$ by
\begin{align}
\label{defnGamma}
\Gamma(\mu) = \nu, 
\end{align}
where the measure $\nu$ is given by
\begin{align}
\label{relationofmuandnu}
\nu(S) = \sum_{i \in Q} \Lambda_i(h) \mu_i(S) \textrm{ for any } S \in \mathcal{B}(E), 
\end{align}
with $h = H(\mu)$ being the density vector corresponding to $\mu$. Note that for each $h \in U_{\textnormal{eq}}$, $\Lambda(h)$ is a
vector which is positive in each component and hence $\nu(S)\geq 0$ for all $ S \in \mathcal{B}(E)$. Since the function $\Lambda$ satisfies \eqref{innerproductcondition} 
we have
\[\nu(E) =  \sum_{i \in Q} \Lambda_i(h) \mu_i(E) =  \sum_{i \in Q} \Lambda_i(h) h_i = 1.\]
This shows that $\nu$ is a probability measure on $E$. 

Let $\Upsilon'$ be the class of functions in $C\left( \mathcal{M}_F^q(E : U_{\textnormal{eq}})\right)$ given by 
\begin{align}
\label{defnupsilon}
\Upsilon' =  & \left\{ F(\mu)  = \left( h_j \langle f, \mu_i \rangle - h_i \langle f, \mu_j \rangle  \right) L(\mu) : (h_1,\dots,h_q) = H(\mu) , \right. 
\\ & \left.  f \in C(E) \textrm{ , }  L \in C\left( \mathcal{M}_F^q(E : U_{\textnormal{eq}})\right) \textrm{ and } i,j \in Q   \right\}. \notag
\end{align}
Let $\Upsilon$ be the smallest algebra of functions in $C\left( \mathcal{M}_F^q(E : U_{\textnormal{eq}})\right)$ containing 
$\Upsilon'$. Observe that if $G \in C\left( \mathcal{M}_F^q(E : U_{\textnormal{eq}})\right)$ and $L \in \Upsilon$, then the product $G L$ is in $\Upsilon$. 
Given two functions 
$G_1,G_2 \in  C\left( \mathcal{M}_F^q(E : U_{\textnormal{eq}})\right)$ we say that $G_2(\mu) = G_1(\mu) + \Upsilon$ for all $\mu \in  \mathcal{M}_F^q( E : U_{\textnormal{eq}} )$ if and only 
if the function $(G_2-G_1)$ is in the class $\Upsilon$.

Let $F \in C(\mathcal{P}(E))$ be a function in the class $\mathcal{C}_0$ defined by \eqref{defc0}. Then $F$ has the form 
\begin{align}
\label{fvconvergence:Fform}
F(\nu) =  \prod_{j = 1}^{m} \langle f_j, \nu \rangle,
\end{align}
where $f_1,\dots,f_m \in \mathcal{D}_0$. Corresponding to $F$, define the functions $F_l, F_{lk} \in \mathcal{C}_0$ for all distinct $l,k \in Q$ by
\begin{align}
\label{fvconvergence:Fformlk}
F_{l}(\nu) = \prod_{j = 1, j \neq l}^{m} \langle f_j, \nu \rangle 
\ \textrm{ and } \ 
F_{lk}(\nu) = \prod_{j = 1, j \neq l,k}^{m} \langle f_j, \nu \rangle.
\end{align}
Using any $F \in \mathcal{C}_0$ we construct a function $\hat{F} \in \mathcal{C}^q_0$ as follows. We first extend the definition of $\Lambda$ 
to the whole of $\R^q$ by letting $\Lambda(h) = \bar{0}_q$ for all $h \notin \hat{U}_{\textnormal{eq}}$. If $F$ has the form \eqref{fvconvergence:Fform} then consider the function 
$\hat{F} : \mathcal{M}^q_F(E) \to \R$ given by 
\begin{align}
\label{fvconvergence:Fhatform}
\hat{F}(\mu) =  \prod_{j = 1}^{m} \left(\sum_{i \in Q} \Lambda_i(h) \langle f_j, \mu_i \rangle \right),
\end{align}
where $h = H(\mu)$. Due to \eqref{innerproductcondition}, the function $\hat{F}$ is in the class $\mathcal{C}^q_0$ defined by \eqref{classc0q}. 
The next result demonstrates how the action of various operators on functions of the form \eqref{fvconvergence:Fhatform} can be approximated.
\begin{proposition}
\label{prop:approximation}
Let $F \in \mathcal{C}_0$ have the form \eqref{fvconvergence:Fform}. Corresponding to $F$ let $\hat{F} \in \mathcal{C}^q_0$ have the form \eqref{fvconvergence:Fhatform} and 
for distinct $l,k \in Q$ let $F_l, F_{lk}$ be given by \eqref{fvconvergence:Fformlk}. Then 
for all $\mu \in \mathcal{M}^q_F(E:U_{\textnormal{eq}})$ with $h = H(\mu)$ and $\nu = \Gamma(\mu)$ we have the following.
\begin{itemize}
\item[(A)] Let $\mathbf{R}^N$ be the operator given by \eqref{defroperatorrn}. Then 
\begin{align}
\label{prop:approxforrn}
N \mathbf{R}^N \hat{F}(\mu) =   
 \sum_{1 \leq l \neq k \leq m} \gamma(h) \left( \langle f_lf_k, \nu \rangle - \langle f_l, \nu \rangle \langle f_k, \nu \rangle  \rangle\right)F_{lk}(\nu)+\Upsilon + o(1),
\end{align}
where 
\begin{align}
\label{defngammah} 
\gamma(h)= \frac{1}{2}\left[  \sum_{ i,j \in Q } \beta_{ij}(h) h_i (\Lambda_j(h))^2 + \sum_{i \in Q} \rho_i(h)h_i (\Lambda_i(h))^2 \right].
\end{align}
\item[(B)]Let $\mathbf{B}^N$ be the operator given by \eqref{maindefbn}. Then 
\begin{align}
\label{prop:approxforbn}
\mathbf{B}^N \hat{F}(\mu) =  \sum_{l = 1}^{m} \left(\sum_{i \in Q } \Lambda_i(h) h_i \langle B_i f_l, \nu \rangle \right) F_l(\nu) + \Upsilon  + o(1).
\end{align}
\item[(C)] Let $\mathbf{G}^N_1$ be the operator given by \eqref{defgn1}. Then
\begin{align}
\label{prop:approxforgn1}
\mathbf{G}^N_1 \hat{F}(\mu) & = 
\sum_{l = 1}^m \left[ \left( \langle  b^{s}(\cdot,h) f_l(\cdot), \nu \rangle - \langle  b^{s}(\cdot,h), \nu \rangle \langle f_l, \nu \rangle \right) \right. \\
& \left. + \left( \langle d^{s}(\cdot,h), \nu \rangle \langle f_l, \nu \rangle -  \langle d^{s}(\cdot,h) f_l(\cdot), \nu \rangle \right) \right] F_l(\nu)  + \Upsilon+ o(1), \notag
\end{align}
where for any $x \in E$ and $h \in \R^q_{+}$
\begin{align}
\label{operatorg1:defn_bs}
b^s(x,h) =  \sum_{i,j \in Q } b^{s}_{i j}(x,h)\Lambda_j(h)h_i \quad \textrm{ and } \quad d^s(x,h) =  \sum_{i \in Q } d^{s}_{i}(x,h)\Lambda_i(h)h_i.
\end{align}
\item[(D)] Let $\mathbf{G}^N_2$ be the operator given by \eqref{defgn2}. Then
\begin{align}
\label{prop:approxforgn2}
\mathbf{G}^N_2 \hat{F}(\mu) & =   \sum_{l = 1}^{m} \left( \sum_{i, j \in  Q} \beta_{ij}(h) \Lambda_j(h)h_i \langle C_{ij} f_l , \nu \rangle \right) F_l(\nu) + \Upsilon + o(1),
\end{align}
where the operators $C_{ij}$ are as in Assumption \ref{assmp:limitrelationshipondispersal}.
\item[(E)] Let $\mathbf{G}^N_3$ be the operator given by \eqref{defgn3}. Then
\begin{align}
\label{prop:approxforgn3}
\mathbf{G}^N_3 \hat{F}(\mu) & =  \sum_{l=1}^m \left( \sum_{i \in Q}  \kappa_i(h) \Lambda_i(h) 
\int_{E} \left( f_l(x) -  \langle f_l,\nu \rangle \right) \Theta_i(dx) \right) F_l(\nu) + \Upsilon+ o(1). 
\end{align}
\end{itemize}
\end{proposition}

\begin{proof}
For any $j \in Q$, let $e_j$ be the vector in $\R^q$ of the form $e_j = (0,\dots,0,1,0,\dots,0)$ with the $1$ at the $j$-th position.
Since $\Lambda \in C^2(\hat{U}_{\textnormal{eq}}, \R^q_{+})$ and $\hat{U}_{\textnormal{eq}}$ is an open set containing $U_{\textnormal{eq}}$, if $h \in U_{\textnormal{eq}}$, then 
using Taylor's theorem we can write
\begin{align*}
\Lambda_i \left( h \pm \frac{1}{N} e_j \right) = 
\Lambda_i(h) \pm  \frac{1}{N} \frac{\partial \Lambda_i (h)}{\partial h_j}+\frac{1}{2 N^2} \frac{ \partial ^2 \Lambda_i (h)}{ \partial h_j^2} + o(N^{-2})
\end{align*}
for any $i,j \in Q$.
But then for any $\mu \in \mathcal{M}_F(E : U_{\textnormal{eq}})$ and $x \in E$
\begin{align*}
\hat{F}\left( \mu \pm \frac{1}{N} \delta^j_x \right) & =  
\prod_{l = 1}^{m} \left( \sum_{i \in Q} \Lambda_i\left(h \pm \frac{1}{N} e_j\right) \langle f_l, \mu_i \rangle \pm 
\Lambda_j \left( h \pm \frac{1}{N} e_j \right) \frac{f_l(x)}{N}\right) \\
& = \prod_{l = 1}^{m} \left[ \sum_{i \in Q} \Lambda_i(h) \langle f_l, \mu_i \rangle   \pm  \frac{1}{N} \left( \sum_{i \in Q} \frac{\partial \Lambda_i (h)}{\partial h_j} \langle f_l, \mu_i \rangle + \Lambda_j(h)f_l(x)\right)\right. 
\\ & \left. + \frac{1}{N^2} \left( \frac{1}{2}  \sum_{i \in Q} \frac{\partial \Lambda^2_i (h)}{\partial h_j^2} \langle f_l, \mu_i \rangle  + 
\frac{\partial \Lambda_j (h)}{\partial h_j} f_l(x) \right)+ o(N^{-2}) \right] \\
& = \hat{F}(\mu) \pm \frac{1}{N} \sum_{l = 1}^{m} \chi^j_l(\mu,x) F_l(\nu) + \frac{1}{N^2} \sum_{l = 1}^{m} \phi^j_l(\mu,x) F_l(\nu) 
\\& +  \frac{1}{2 N^2} \sum_{1 \leq l  \neq  k \leq m} \chi^j_l(\mu,x)\chi^j_k(\mu,x) F_{lk}(\nu) + o(N^{-2}),   
\end{align*}
where $F_l, F_{lk}$ are as in \eqref{fvconvergence:Fformlk} and $ \chi^j_l(\mu,x) ,  \phi^j_l(\mu,x)$ are given by
\begin{align}
\label{defnchiandphi} 
 \chi^j_l(\mu,x) & = \sum_{i \in Q} \frac{\partial \Lambda_i (h)}{\partial h_j} \langle f_l, \mu_i \rangle + \Lambda_j(h)f_l(x) \notag \\
\textrm{ and } \  \phi^j_l(\mu,x) &= \frac{1}{2}  \sum_{i \in Q} \frac{\partial \Lambda^2_i (h)}{\partial h_j^2} \langle f_l, \mu_i \rangle  + \frac{\partial \Lambda_j (h)}{\partial h_j} f_l(x). 
\end{align}
On rearranging we obtain
\begin{align}
\label{simplifiedtaylorexpansion}
\hat{F}\left( \mu \pm \frac{1}{N} \delta^j_x \right) - \hat{F}(\mu) = 
&\pm \frac{1}{N} \sum_{l = 1}^{m} \chi^j_l(\mu,x) F_l(\nu) + \frac{1}{N^2} \sum_{l = 1}^{m} \phi^j_l(\mu,x) F_l(\nu) \\
& +  \frac{1}{2 N^2} \sum_{1 \leq l \neq k \leq m} \chi^j_l(\mu,x)\chi^j_k(\mu,x) F_{lk}(\nu) + o(N^{-2}).\notag
\end{align}

Therefore for any $\mu \in \mathcal{M}_F(E : U_{\textnormal{eq}})$ 
\begin{align*}
N \mathbf{R}^N \hat{F}(\mu) &= N^2 \sum_{ i,j \in Q } \int_{E} \beta_{ij}(h) \left( \hat{F} \left( \mu +\frac{1}{N} \delta_{x}^j \right) - \hat{F}(\mu) \right)\mu_i(dx) 
\\& +  N^2 \sum_{i \in Q} \int_{E} \rho_{i}(h) \left( \hat{F} \left( \mu -\frac{1}{N} \delta_{x}^{i} \right) - \hat{F}(\mu) \right)\mu_i(dx)\\
&= N \sum_{l=1}^m \left(  \sum_{ i,j \in Q } \beta_{ij}(h)  \langle \chi^j_l(\mu,\cdot ),\mu_i \rangle  -\sum_{i \in Q } \rho_i(h) \langle \chi^i_l(\mu,\cdot),\mu_i \rangle\right) F_{l}(\nu) \\
& + \sum_{l=1}^m \left(  \sum_{ i,j \in Q } \beta_{ij}(h)  \langle \phi^j_l(\mu,\cdot),\mu_i \rangle  
+ \sum_{i \in Q} \rho_i(h) \langle \phi^i_l(\mu,\cdot),\mu_i \rangle\right) F_{l}(\nu) \\
& + \frac{1}{2}  \sum_{1 \leq l \neq k \leq m}
\left(  \sum_{ i,j \in Q } \beta_{ij}(h)\langle  \chi^j_l(\mu,\cdot )\chi^j_k(\mu,\cdot ),\mu_i \rangle  
\right. \\ & \left.  + \sum_{i \in Q} \rho_i(h) \langle  \chi^i_l(\mu,\cdot )\chi^i_l(\mu,\cdot ),\mu_i \rangle\right) F_{lk}(\nu) + o(1). 
\end{align*}
But note that
\begin{align*}
& \sum_{ i,j \in Q } \beta_{ij}(h)  \langle \chi^j_l(\mu,\cdot),\mu_i \rangle  + \sum_{i \in Q} \rho_i(h) \langle \chi^i_l(\mu,\cdot),\mu_i \rangle  \\
& =   \sum_{r  \in Q } \langle  f_l , \mu_r\rangle  \left(  \sum_{j \in Q } \frac{\partial \Lambda_r(h)}{ \partial h_j} 
\left( \sum_{i \in Q }\beta_{ij}(h)h_i - \rho_j(h)h_j\right) + \sum_{j \in Q } \beta_{rj}(h) \Lambda_j(h) - \rho_r(h) \Lambda_r(h)\right) \\
& = \sum_{r  \in Q } \langle  f_l , \mu_r\rangle  \left(  \sum_{j \in Q } \frac{\partial \Lambda_r(h)}{ \partial h_j} \theta_j(h) 
+ \sum_{j \in Q } A_{rj}(h) \Lambda_j(h)\right),
\end{align*}
where the matrix $A(h)$ and the vector $\theta(h)$ are defined by \eqref{defninteractionmatrix} and \eqref{defnthetah}. Since the function $\Lambda$ satisfies 
\eqref{kernelequation}, the expression on the right is just $0$. Hence the formula for $N \mathbf{R}^N \hat{F}(\mu)$ simplifies to
\begin{align}
\label{approximationonNRN}
& N \mathbf{R}^N \hat{F}(\mu) \notag \\
&= \sum_{l=1}^m \left( \sum_{ i,j \in Q } \beta_{ij}(h)  \langle \phi^j_l(\mu,\cdot),\mu_i \rangle  
+ \sum_{i \in Q} \rho_i(h) \langle \phi^i_l(\mu,\cdot),\mu_i \rangle\right) F_{l}(\nu) \\
& + \frac{1}{2}  \sum_{1 \leq l \neq k \leq m}
\left(  \sum_{ i,j \in Q } \beta_{ij}(h)\langle  \chi^j_l(\mu,\cdot )\chi^j_k(\mu,\cdot ),\mu_i \rangle  
+ \sum_{i \in Q} \rho_i(h) \langle  \chi^i_l(\mu,\cdot )\chi^i_l(\mu,\cdot ),\mu_i \rangle\right) F_{lk}(\nu) + o(1). \notag
\end{align}

Equation \eqref{innerproductcondition} says that for all $h \in \hat{U}_{\textnormal{eq}}$ 
\[\sum_{i  \in Q}  h_i \Lambda_i(h) = 1.\]
Pick a $j \in Q$. Differentiating the above equation with respect to $h_j$ we get
\begin{align}
\label{derivativerelation1}
\sum_{i  \in Q }  h_i \frac{\partial \Lambda_i(h)}{\partial h_j} + \Lambda_j(h) = 0
\end{align}
and differentiating again with respect to $h_j$ we obtain
\begin{align}
\label{derivativerelation2}
\sum_{i  \in Q }  h_i \frac{\partial^2 \Lambda_i(h)}{\partial h_j^2} + 2\frac{\partial \Lambda_j(h)}{\partial h_j} = 0.
\end{align}

Recall that for any $\mu \in  \mathcal{M}_F(E : U_{\textnormal{eq}})$, $\nu = \Gamma(\mu)$ is given by \eqref{relationofmuandnu}. Using \eqref{innerproductcondition} 
one can verify that for any 
$f \in C(E)$ and $i \in Q$ 
\begin{align}
\label{mutonurelation1}
\langle f, \mu_i \rangle = h_i \langle f ,\nu \rangle + \sum_{j \in Q}\left( h_j \langle f, \mu_i \rangle - h_i \langle f, \mu_j \rangle  \right) \Lambda_j(h).  
\end{align}
But the second term on the right is a function in the class $\Upsilon$. Hence for all $\mu \in \mathcal{M}_F(E : U_{\textnormal{eq}})$
\begin{align}
\label{mutonurelation}
\langle f, \mu_i \rangle = h_i \langle f ,\nu \rangle  + \Upsilon. 
\end{align}
From the definitions of $\chi^j_l$ and $\phi^j_l$ (see \eqref{defnchiandphi}) it is immediate that for any 
$ j \in Q$, $l \in \{1,\dots,m\}$ and $x \in E$ we have the following relations. 
For all $\mu \in \mathcal{M}_F(E : U_{\textnormal{eq}})$
\begin{align*}
\chi^j_l(\mu,x) & =  \Lambda_j(h)f_l(x) +
\left( \sum_{i \in Q} h_i\frac{\partial \Lambda_i (h)}{\partial h_j} \right) \langle f_l, \nu \rangle  + \Upsilon \\
\textrm{ and } \ \phi^j_l(\mu,x) & = \frac{\partial \Lambda_j (h)}{\partial h_j} f_l(x) + 
\left( \frac{1}{2}  \sum_{i \in Q} h_i\frac{\partial \Lambda^2_i (h)}{\partial h_j^2} \right) \langle f_l, \nu \rangle   + \Upsilon .
\end{align*}
Using \eqref{derivativerelation1} and \eqref{derivativerelation2} we obtain
\begin{align}
\chi^j_l(\mu,x) & =  \Lambda_j(h)\left( f_l(x) - \langle f_l, \nu \rangle \right)+ \Upsilon \label{simpl1chi} \\
\textrm{ and } \ \phi^j_l(\mu,x) & = \frac{\partial \Lambda_j (h)}{\partial h_j} \left( f_l(x) - \langle f_l, \nu \rangle \right)   + \Upsilon .\label{simpl1phi}
\end{align}
Recall that the class $\Upsilon$ is invariant under multiplication by functions in $C(\mathcal{M}_F(E : U_{\textnormal{eq}}))$. It can be checked that 
for any $i, j \in Q$ and $l,k \in \{1,\dots,m\}$
\begin{align*}
\langle \chi^j_l(\mu,\cdot) \chi^j_k(\mu,\cdot) , \mu_i \rangle  = h_i (\Lambda_j(h))^2 
\left( \langle f_l f_k , \nu\rangle - \langle f_l, \nu\rangle \langle f_k, \nu\rangle\right) + \Upsilon
\end{align*}
and the function $\mu \mapsto  \langle \phi^j_l(\mu,\cdot), \mu_i \rangle$ belongs to class $\Upsilon$. Substituting these two relations in \eqref{approximationonNRN} 
proves part (A) of this proposition. 
 
Recall the definition of the operator $\mathbf{B}^n_i$ from Section \ref{migrationmechanism}. 
If $G(\nu) = \prod_{j = 1}^l \langle g_j,\nu \rangle \in \mathcal{C}_0$ then one can verify (see Section 2.2 in \cite{DawsonEcole}) 
that there is a constant $c$ (depending on $l$ and $g_1,\dots, g_l$) such that
\begin{align*}
\sup_{n \in \N} \sup_{\nu \in \mathcal{P}_{n,a}} \left( n
\left|\mathbf{B}^n_i G(\nu) - \sum_{j = 1}^{l} \langle B_i g_j,\nu \rangle \prod_{k = 1, k  \neq j}^{l} \langle g_k , \nu \rangle \right| \right) \leq c. 
\end{align*}
From the definition of the operator $\mathbf{B}^N$ and \eqref{mutonurelation} it is immediate that 
\begin{align*}
\mathbf{B}^N \hat{F}(\mu) &= \sum_{l = 1}^{m}  \left(\sum_{i \in Q } \Lambda_i(h) \langle B_i f_l, \mu_i \rangle \right) F_l(\nu) + o(1) \\
& =   \sum_{l = 1}^{m} \left(\sum_{i \in Q } \Lambda_i(h) h_i \langle B_i f_l, \nu \rangle \right) F_l(\nu) + \Upsilon  + o(1).
\end{align*}
This proves part (B) of the proposition.

Using \eqref{simplifiedtaylorexpansion}, \eqref{simpl1chi} and \eqref{mutonurelation} we get
\begin{align*}
\mathbf{G}^N_1 \hat{F}(\mu) &=  N \left( \sum_{i, j  \in Q} \int_{E}  b^{s}_{i j}(x,h)\left( \hat{F}\left( \mu + \frac{1}{N}\delta^j_{x}  \right) - \hat{F}(\mu)\right)  \mu_i(dx)  \right. \\& \left. + 
\sum_{i \in Q } \int_{E } d^{s}_{i}(x,h) \left( \hat{F}\left( \mu - \frac{1}{N}\delta^i_{x}  \right) - \hat{F}(\mu)\right)   \mu_i(dx)\right) \\
& = \sum_{l = 1}^m \left( \sum_{i, j \in Q } \left\langle  b^{s}_{i j}(\cdot,h) \chi^j_l(\mu, \cdot), \mu_i \right\rangle 
\right. \\ & \left.  
- \sum_{i \in Q } \left\langle d^{s}_{i}(\cdot,h)\chi^i_l(\mu, \cdot),\mu_i \right\rangle \right) F_l(\nu)  + \Upsilon+ o(1) \\
& = \sum_{l = 1}^m \left( \sum_{i, j \in Q } \Lambda_j(h) \left\langle  b^{s}_{i j}(\cdot,h) \left(f_l(\cdot) - \langle f_l, \nu \rangle \right), \mu_i \right\rangle \right. \\ & \left.  
- \sum_{i \in Q } \Lambda_i(h) \left\langle d^{s}_{i}(\cdot,h)\left(f_l(\cdot)  - \langle f_l, \nu \rangle \right),\mu_i \right\rangle \right) F_l(\nu) + \Upsilon+ o(1) \\
& = \sum_{l = 1}^m \left( \sum_{i, j \in Q } \Lambda_j(h) h_i\left\langle  b^{s}_{i j}(\cdot,h) \left(f_l(\cdot) - \langle f_l, \nu \rangle \right), \nu \right\rangle \right. \\ & \left.  
- \sum_{i \in Q } \Lambda_i(h)h_i \left\langle d^{s}_{i}(\cdot,h)\left(f_l(\cdot) - \langle f_l, \nu \rangle \right),\nu \right\rangle \right) F_l(\nu)  + \Upsilon+ o(1) \\
& = \sum_{l = 1}^m \left( \left( \langle  b^{s}(\cdot,h) f_l(\cdot), \nu \rangle - \langle  b^{s}(\cdot,h), \nu \rangle \langle f_l, \nu \rangle \right)
\right. \\ & \left. 
+ \left( \langle d^{s}(\cdot,h), \nu \rangle \langle f_l, \nu \rangle  -  \langle d^{s}(\cdot,h) f_l(\cdot), \nu \rangle \right) \right) F_l(\nu)  + \Upsilon+ o(1), 
\end{align*}
where the functions $b^s$ and $d^s$ are defined in \eqref{operatorg1:defn_bs}. This proves part (C).

Observe that Assumption \ref{assmp:limitrelationshipondispersal} implies that for any $f \in \mathcal{D}_0$ and $i,j \in Q$
\begin{align*}
 N \int_{E} p^N_{ij}(x) \int_{E} \left( f(y) - f(x) \right) \vartheta^N_{ij}(x,dy)\mu_i(dx) = \langle C_{ij} f, \mu_i \rangle + o(1) 
\end{align*}
and
\begin{align*}
\int_{E} p^N_{ij}(x) \int_{E} \left( f(y) - f(x) \right) \vartheta^N_{ij}(x,dy)\mu_i(dx) = o(1). 
\end{align*}
Using \eqref{simplifiedtaylorexpansion} , \eqref{defnchiandphi} and \eqref{mutonurelation} we obtain
\begin{align*}
&\mathbf{G}^N_2 \hat{F}(\mu) \\& =  N^2 \sum_{i, j \in Q } \beta_{ij}(h)  \int_{E }  p^N_{ij}(x) 
\left( \int_{E} \left( \hat{F} \left( \mu +\frac{1}{N} \delta^j_{y} \right) - \hat{F} \left( \mu +\frac{1}{N} \delta^j_{x} \right) \right) \vartheta^N_{ij}(x,dy) 
\right)\mu_i(dx)\\
& =  \sum_{l=1}^m \left( N \sum_{i, j \in Q } \beta_{ij}(h) \int_{E }  p^N_{ij}(x) 
\int_{E} \left( \chi^j_l(\mu,y) -  \chi^j_l(\mu,x) \right) \vartheta^N_{ij}(x,dy)  \mu_i(dx) \right) F_l(\nu) + o(1)\\
& =  \sum_{l=1}^m \left(  N \sum_{i, j \in Q } \beta_{ij}(h) \Lambda_j(h) \int_{E }  p^N_{ij}(x) 
\left[ \int_{E} \left( f_l(y)-  f_l(x) \right) \vartheta^N_{ij}(x,dy) \right] \mu_i(dx) \right) F_l(\nu) + o(1)\\
& = \sum_{l = 1}^{m} \left( \sum_{i, j \in  Q} \beta_{ij}(h) \Lambda_j(h) \langle C_{ij} f_l , \mu_i \rangle \right) F_l(\nu) + o(1) \\
& = \sum_{l = 1}^{m} \left( \sum_{i, j \in  Q} \beta_{ij}(h) \Lambda_j(h)h_i \langle C_{ij} f_l , \nu \rangle \right) F_l(\nu) + \Upsilon + o(1).
\end{align*}
This proves part (D).

Again using \eqref{simplifiedtaylorexpansion}, \eqref{simpl1chi} and \eqref{mutonurelation} we see that
\begin{align*}
\mathbf{G}^N_3 \hat{F}(\mu) & =   N \sum_{i=1}^{q}  \kappa_i(h) \int_{E} \left( \hat{F} \left( \mu +\frac{1}{N} \delta^i_{x} \right) - \hat{F}(\mu)  \right) \Theta_i(dx) \\
& = \sum_{l=1}^m \left( \sum_{i \in Q}  \kappa_i(h) \int_{E} \chi^i_l(\mu,x) \Theta_i(dx) \right) F_l(\nu) + o(1) \\
& = \sum_{l=1}^m \left( \sum_{i \in Q}  \kappa_i(h) \Lambda_i(h) \int_{E} \left( f_l(x) -  \langle f_l,\nu \rangle \right) \Theta_i(dx) \right) F_l(\nu) + \Upsilon+ o(1). 
\end{align*}
This proves part (E) and completes the proof of this proposition.
\end{proof}

\begin{remark}\label{remarkapproximation}
For any $(h,\nu) \in U_{\textnormal{eq}} \times \mathcal{P}(E)$ let $\zeta^R_F(h,\nu),\zeta^B_F(h,\nu),\zeta^{G,1}_F(h,\nu),\zeta^{G,2}_F(h,\nu)$ and $\zeta^{G,3}_F(h,\nu)$ 
be the first terms that appear on the right hand side of equations \eqref{prop:approxforrn}, \eqref{prop:approxforbn}, \eqref{prop:approxforgn1}, 
\eqref{prop:approxforgn2} and \eqref{prop:approxforgn3} respectively. Let $\zeta^{G,0}_F(h,\nu) = 0$ for all $(h,\nu) \in U_{\textnormal{eq}} \times \mathcal{P}(E)$.
Recall the definitions of the operators $\mathbf{A}_l$ for $l \in \{0,1,2,3\}$ from \eqref{main:gena0}, 
\eqref{main:gena1}, \eqref{main:gena2} and \eqref{main:gena3}. One can verify that for each $l \in \{0,1,2,3\}$, $F \in \mathcal{C}_0$ and $\nu \in \mathcal{P}(E)$
\begin{align}
\label{limitingspolittingofa}
\mathbf{A}_l F(\nu) = \zeta^B_F(h_{\textnormal{eq}},\nu) + \zeta^R_F(h_{\textnormal{eq}},\nu) + \zeta^{G,l}_F(h_{\textnormal{eq}},\nu).  
\end{align}
To check this relation observe that $\Lambda$ satisfies \eqref{initialcondition}. Furthermore for each $i \in Q$, 
$\theta_i(h_{\textnormal{eq}}) = \sum_{j \in Q} \beta_{ji}(h_{\textnormal{eq}}) h_{\textnormal{eq},j} - \rho_i(h_{\textnormal{eq}}) h_{\textnormal{eq},i} = 0$, which shows that the value of the function $\gamma$ 
(given by \eqref{defngammah}) at $h_{\textnormal{eq}}$ is equal to the constant $\gamma_{\textrm{smpl}}$ (given by \eqref{defnresamplingrate}).
\end{remark}
We now prove the main theorem of the paper.\\
\begin{proof}[Proof of Theorem \ref{mainresult}] 
Fix a $l \in \{0,1,2,3\}$ and let $\{ \mu^N(t) : t \geq 0 \}$ be a solution to the martingale problem for $\mathbf{A}^N_l$. Let 
$\{ h^N(t) = H \left( \mu^N(t)\right) : t \geq 0 \}$ be the 
corresponding density process. Since $\mu^N(0) \Rightarrow \mu(0)$ as $N \to \infty$ and $H(\mu(0)) \in U_{\textnormal{eq}}$ a.s.\ we must have that $h^N(0) \Rightarrow h(0)$ and 
$h(0) \in U_{\textnormal{eq}}$ a.s. It suffices to prove the theorem under the assumption that for all $N \in \N$, $h^N(0) \in K_0$ a.s.\ for some compact $K_0 \subset U_{\textnormal{eq}}$. 
 
By Lemma \ref{compact_containment}, we can find a bigger compact set $K \subset U_{\textnormal{eq}}$ containing $K_0$ such that if we define the stopping time 
$\sigma^N $ by
\begin{align*}
\sigma^N = \inf \left\{ t \geq 0 : h^N(t) \notin K \right\} 
\end{align*}
then $\sigma^N \Rightarrow \infty$ as $N \to \infty$. 

Let $F \in \mathcal{C}_0$ be a function of the form \eqref{fvconvergence:Fform} and let $\hat{F} \in \mathcal{C}^q_0$ have the form \eqref{fvconvergence:Fhatform}.
From \eqref{splittingofanl} and Proposition \ref{prop:approximation} we can conclude that for all $\mu \in \mathcal{M}_F(E : U_{\textnormal{eq}})$
\begin{align}
\label{fv:splittingofanl}
\mathbf{A}^N_l \hat{F}(\mu) = \zeta^B_F(h,\nu)+\zeta^R_F(h,\nu)+\zeta^{G,l}_F(h,\nu) + \Upsilon + o(1), 
\end{align}
where $h = H(\mu)$, $\nu \in \Gamma(\mu)$ and the continuous functions $\zeta^B_F,\zeta^R_F,\zeta^{G,0}_F,\zeta^{G,1}_F,\zeta^{G,2}_F,\zeta^{G,3}_F$ from 
$U_{\textnormal{eq}} \times \mathcal{P}(E)$ to $\R$ are defined in Remark \ref{remarkapproximation}.

The relation \eqref{fv:splittingofanl} implies that
\begin{align}
\label{tightness1}
\sup_{N \in \N} \sup_{\mu \in \mathcal{M}_F(E: K)} | \mathbf{A}^N_l \hat{F}(\mu) | < \infty. 
\end{align}
The function $\hat{F}$ belongs to the domain of $\mathbf{A}^N_l$ and the process $\{\mu^N(t) : t \geq 0\}$ is a solution to the martingale problem for $\mathbf{A}^N_l$. Hence 
\begin{align*}
\hat{F}(\mu^N(t)) - \hat{F}^N(\mu^N(0)) - \int_{0}^{t} \mathbf{A}^N_l \hat{F}(\mu^N(s))ds   
\end{align*}
is a martingale and using the optional sampling theorem we see that
\begin{align*}
m^N_F(t) = \hat{F}(\mu^N(t \wedge \sigma^N)) - \hat{F}^N(\mu^N(0)) - \int_{0}^{t \wedge \sigma^N} \mathbf{A}^N_l \hat{F}(\mu^N(s))ds   
\end{align*}
is also a martingale. Note that for all $t \geq 0$, $\mu^N(t \wedge \sigma^N)$ is in the set $\mathcal{M}_F(E : U_{\textnormal{eq}})$. Define a $\mathcal{P}(E)$-valued process by
\begin{align}
\label{fv:defnun}
 \nu^N(t) = \Gamma \left( \mu^N(t \wedge \sigma^N) \right) \textrm{ for } t \geq 0.
\end{align}
The form of the functions $F$ and $\hat{F}$ shows that $\hat{F}(\mu^N(t \wedge \sigma^N)) = F(\nu^N(t))$ for any $t \geq 0$. Hence the martingale $m^N_F$ can be rewritten as
\begin{align}
\label{fv:martn1}
m^N_F(t) = F(\nu^N(t))- F(\nu^N(0)) - \int_{0}^{t \wedge \sigma^N} \mathbf{A}^N_l \hat{F}(\mu^N(s))ds.   
\end{align}
The linear span of functions in the class $\mathcal{C}_0$ is a dense sub-algebra of $C\left( \mathcal{P}(E)\right)$ and for every $F \in \mathcal{C}_0$ we have the martingale 
relation \eqref{fv:martn1}. Theorems 3.9.1 and 3.9.4 in Ethier and Kurtz \cite{EK} along with the estimate \eqref{tightness1} imply that the sequence of 
processes $\{ \nu^N : N \in \N\}$ is tight in the space $D_{\mathcal{P}(E)}[0,\infty)$.

Let $t_N$ be a sequence satisfying the conditions of Theorem \ref{mainresult}. Pick a $T>0$. It is easy to see that 
\begin{align}
\label{fv:approximationforanl}
\sup_{t \in [0,T]} \left| \int_{t_N}^{t \wedge \sigma^N+t_N} \mathbf{A}^N_l \hat{F}(\mu^N(s))ds - \int_{0}^{t \wedge \sigma^N} \mathbf{A}^N_l \hat{F}(\mu^N(s))ds \right| \Rightarrow 0 \textrm{ as } N \to \infty. 
\end{align}
From parts (A) and (B) of Proposition \ref{prop:constancyofh_mainresult} we know that as $N \to \infty$
\begin{align}
\label{fv:constancyofh}
\sup_{t \in [0,T]} \left\| h^N(t+t_N) - h_{\textnormal{eq}}\right\|_1 \Rightarrow 0  
\end{align}
and for any $f \in C(E)$ and $i,j \in Q$
\begin{align}
\label{fv:inseparability1}
\sup_{t \in [0,T] } \left| h^N_j(t+t_N) \langle f,\mu^N_i(t+t_N) \rangle -  h^N_i(t+t_N)\langle f,\mu^N_j(t+t_N) \rangle \right| \Rightarrow 0.
\end{align}
Note that this also implies that if $L$ is a function in the class $\Upsilon$ then  
\begin{align}
\label{fv:inseparability2}
\sup_{t \in [0,T] } \left| L(\mu^N(t+t_N))  \right| \Rightarrow 0 \textrm{ as } N \to \infty.
\end{align}

We argued before that the sequence of processes $\{\nu^N : N \in \N\}$ is tight. Let $\{ \nu(t) : t \geq 0 \}$ be a limit point. Then along some sequence $k_N$, $\nu^N \Rightarrow \nu$ as $N \to \infty$. 
Since $\sigma^N \Rightarrow \infty$, using the continuous mapping theorem and \eqref{fv:constancyofh} we obtain that along the subsequence $k_N$
\begin{align}
\label{fv:convergenceofchi}
\sup_{t \in [0,T]} \left| \int_{t_N}^{t\wedge \sigma^N+ t_N} \zeta(h^N(s),\nu^N(s))ds - \int_{0}^{t} \zeta(h_{\textnormal{eq}},\nu(s))ds \right| \Rightarrow 0 \textrm{ as } N \to \infty,  
\end{align}
where $\zeta$ is any of the continuous functions $\zeta^B_F,\zeta^R_F,\zeta^{G,0}_F,\zeta^{G,1}_F,\zeta^{G,2}_F,\zeta^{G,3}_F$ defined in Remark \ref{remarkapproximation}. 
From \eqref{fv:splittingofanl}, \eqref{fv:approximationforanl}, \eqref{fv:inseparability2}, \eqref{fv:convergenceofchi} and \eqref{limitingspolittingofa} we get that along the subsequence $k_N$
\begin{align}
\label{fv:limitforanl}
\sup_{t \in [0,T]} \left| \int_{0}^{t \wedge \sigma^N} \mathbf{A}^N_l \hat{F}(\mu^N(s))ds - 
\int_{0}^{t} \mathbf{A}_l F(\nu(s)) ds \right| \Rightarrow 0 \textrm{ as } N \to \infty. 
\end{align}
Using \eqref{fv:limitforanl} and the continuous mapping theorem we can conclude that for any $F \in \mathcal{C}_0$, as $N \to \infty$,
the sequence of martingales $m^N_F$ (given by \eqref{fv:martn1}) converges in distribution along the subsequence $k_N$ to the martingale given by
\begin{align}
 F(\nu(t))- F(\nu(0)) - \int_{0}^{t } \mathbf{A}_l F(\nu(s))ds.   
\end{align}
This shows that $\{ \nu(t) : t \geq 0 \}$ is a solution to the martingale problem for $\mathbf{A}_l$.
Let $\pi \in \mathcal{P}(\mathcal{P}(E))$ be the distribution of $\Gamma(\mu(0))$. 
Since $\mu^N(0) \Rightarrow \mu(0)$ as $N \to \infty$ and $\Gamma$ is a continuous map we must also 
have that $\nu^N(0) \Rightarrow \nu(0)$, where $\nu(0)$ has distribution $\pi$.  
We argued in Section \ref{mainresultsection} that the martingale problem for each $\mathbf{A}_l$ is well-posed. Hence $\{ \nu(t) : t \geq 0\}$ is the unique solution to the martingale problem for $(\mathbf{A}_l,\pi)$ and thus $\nu^N \Rightarrow \nu$ as $N \to \infty$, along the entire sequence. Moreover the limiting process has sample paths in $C_{\mathcal{P}(E)}[0,\infty)$ almost surely.

Let $\{ \hat{\mu}^N(t) : t \geq 0\}$ be the process defined by \eqref{defnhatmun}. Pick a $i \in Q$ and $f \in C(E)$. From \eqref{mutonurelation1} for any $0 \leq t < \sigma^N-t_N$ we can write
\begin{align*}
\left\langle f , \hat{\mu}^N_i(t)\right\rangle &=
\left\langle f , \mu^N_i(t+t_N)\right\rangle \\
&= h^N_i(t+t_N) \langle f , \nu^N(t+t_N) \rangle  \\
& +  \sum_{j \in Q}\left( h^N_j(t+t_N) \langle f, \mu^N_i(t+t_N) \rangle - h^N_i(t+t_N) \langle f, \mu^N_j(t+t_N) \rangle  \right) \Lambda_j(h^N(t+t_N)). 
\end{align*}
Since $\nu^N \Rightarrow \nu$, $\sigma^N \Rightarrow \infty$ and $t_N \to 0$, \eqref{fv:constancyofh} , \eqref{fv:inseparability1} and the continuity of the sample paths of $\{ \nu(t) : t \geq 0\}$ imply that for any $T > 0$
\[ \sup_{t \in [0,T]} \left| \left\langle f , \hat{\mu}^N_i(t)\right\rangle - h_{\textnormal{eq},i} \langle f , \nu(t) \rangle\right| \Rightarrow 0 \textrm{ as } N \to \infty.\] 
This holds for any $f \in C(E)$ and $i \in Q$. Hence by Theorem 3.7.1 in Dawson\cite{DawsonEcole}, $\hat{\mu}^N \Rightarrow h_{\textnormal{eq}} \nu$ as $N \to \infty$ 
in the Skorohod topology on $D_{\mathcal{M}^q_F(E)}[0,\infty)$. This completes the proof of Theorem \ref{mainresult}.
\end{proof}

\begin{remark}
\label{remark:initialdistribution}
In the statement of Theorem \ref{mainresult} we did not specify how the initial distribution $\pi$ of the limiting Fleming-Viot process $\{ \nu(t) : t \geq 0\}$ is related to the distribution of $\mu(0)$.
However the above proof makes it clear that $\pi$ is the distribution of $\Gamma(\mu(0))$ where $\Gamma$ is the map defined by \eqref{defnGamma}. 
\end{remark}

\renewcommand {\theequation}{A.\arabic{equation}}
\appendix
\setcounter{equation}{0}

\section{Appendix.}

\begin{lemma}
\label{lemma:appendix1}
For each $h \in \R^q_{+}$ let $A(h) \in \mathbb{M}(q,q)$ be the matrix given by \eqref{defninteractionmatrix}. Suppose that Assumption \ref{mainassumptionsonr} is 
satisfied and let $h_{\textnormal{eq}} \in \R^q_{+}$ be the nonzero vector such that
\begin{align}
\label{eigenvalue0}
A(h_{\textnormal{eq}})h_{\textnormal{eq}} = \bar{0}_q. 
\end{align}
Then we have the following. 
\begin{itemize}
\item[(A)] The vector $h_{\textnormal{eq}}$ is in $\R^q_{*}$, that is, all its components are strictly positive. 
\item[(B)] The matrix $A(h_{\textnormal{eq}})$ has $0$ as an eigenvalue with algebraic multiplicity $1$. All the other eigenvalues of $A(h_{\textnormal{eq}})$ have strictly negative real parts.  
\item[(C)] There exists a unique vector $v_{\textnormal{eq}} \in \R^q_{*}$ such that $v_{\textnormal{eq}} A(h_{\textnormal{eq}}) = \bar{0}_q$ and $\langle v_{\textnormal{eq}}, h_{\textnormal{eq}} \rangle = 1$.
\item[(D)] Let $G(h_{\textnormal{eq}}) \in \mathbb{M}(q,q)$ be the matrix given by
\begin{align*}
G(h_{\textnormal{eq}})  =   \left( I_q - \frac{h_{\textnormal{eq}}}{\langle \bar{1}_q , h_{\textnormal{eq}} \rangle } \bar{1}^T_q \right) A(h_{\textnormal{eq}}). 
\end{align*}
Then the matrix $G(h_{\textnormal{eq}})$ has the same eigenvalues as the matrix $A(h_{\textnormal{eq}})$.
\item[(E)] Let $\bar{G}(h_{\textnormal{eq}}) \in \mathbb{M}(q-1,q-1)$ be the matrix defined by
\[\bar{G}_{ij}(h_{\textnormal{eq}}) = G_{ij}(h_{\textnormal{eq}}) - G_{iq}(h_{\textnormal{eq}})\textrm{ for all }i,j \in \{1,\dots,q-1\}.\]
Then the matrix $\bar{G}(h_{\textnormal{eq}})$ is stable, that is, all its eigenvalues have strictly negative real parts. 
\end{itemize}
\end{lemma}
\begin{proof}
Observe that all the non-diagonal entries of the matrix $A(h_{\textnormal{eq}})$ are nonnegative. Such matrices are sometimes referred to as Metzler-Leontief matrices in 
mathematical economics (see Section 2.3 in Seneta \cite{Seneta}). Their important property is that they can be transformed to a nonnegative matrix by adding a constant multiple of the identity matrix. This allows extensions of the Perron-Frobenius type results to such matrices. 

Note that the matrix $A(h_{\textnormal{eq}})$ is irreducible (part (C) of Assumption \ref{mainassumptionsonr}) and 
has $0$ as an eigenvalue with $h_{\textnormal{eq}}$ as a right eigenvector (see \eqref{eigenvalue0}). Theorem 2.6 in \cite{Seneta} proves parts (A),(B) and (C) of this lemma.

We now prove parts (D) and (E). Let $P$ and its inverse $P^{-1}$ be the matrices in $\mathbb{M}(q,q)$ given by
\begin{align*}
P= \left[
\begin{array}{cc}
I_{q-1} & \bar{1}_{q-1} \\
\bar{0}^T_{q-1} & 1 \\
\end{array}
\right] 
\hspace{5pt} \textrm{ and } \hspace{5pt} 
P^{-1}= \left[
\begin{array}{cc}
I_{q-1} & -\bar{1}_{q-1} \\
\bar{0}^T_{q-1} & 1 \\
\end{array}
\right].
\end{align*}
Observe that $\bar{1}^T_q G(h_{\textnormal{eq}}) = \bar{0}_q$ and 
\begin{align}
\label{relationofgandgbar}
P^T G(h_{\textnormal{eq}}) (P^T)^{-1}= \left[
\begin{array}{cc}
\bar{G}(h_{\textnormal{eq}}) & v \\
\bar{0}^T_{q-1} & 0 \\
\end{array}
\right], 
\end{align}
where $v$ is some vector in $\R^{q-1}$. Let  $L = \textrm{Diag}(h_{\textnormal{eq}})$ and $Q= P^{-1} L^{-1}$. 
Note that $Q h_{\textnormal{eq}} = P^{-1} L^{-1} h_{\textnormal{eq}} = P^{-1} \bar{1}_q = \bar{0}_q$ and hence
\begin{align*}
 Q \left( I_q - \frac{h_{\textnormal{eq}}}{\langle \bar{1}_q , h_{\textnormal{eq}} \rangle } \bar{1}^T_q \right) &= Q - \frac{1}{\langle \bar{1}_q , h_{\textnormal{eq}} \rangle } [ Q  h_{\textnormal{eq}} ] \bar{1}^T_q  = Q.
\end{align*}
This shows that $Q G(h_{\textnormal{eq}}) Q^{-1}  = Q A(h_{\textnormal{eq}}) Q^{-1}$. Hence the matrices $G(h_{\textnormal{eq}})$ and $A(h_{\textnormal{eq}})$ are similar and have the same eigenvalues. 
Part (B) of this lemma and \eqref{relationofgandgbar} imply that the matrix $\bar{G}(h_{\textnormal{eq}})$ is stable. This completes the proof of this lemma.
\end{proof}


Let $\hat{A}$, $\hat{\theta}$, $\hat{\psi}$ and $\hat{U}_{\textnormal{eq}}$ be as defined in Section \ref{section:pdesolution} (just prior to Proposition \ref{prop:extension}).
\begin{lemma}
\label{lemma:matrixpsi}
Fix a $t_0 > 0$. For $(h,t )\in \hat{U}_{\textnormal{eq}} \times [0,t_0]$ consider the following matrix equation 
\begin{align}
\label{matrixvaluedequation}
\Phi(h,t,t_0) = I_q + \int_{t}^{t_0} \hat{A}^T(\hat{\psi}(h,u)) \Phi(h,u,t_0)du, 
\end{align}
where $I_q$ is the $q\times q$ identity matrix. 
This equation has a unique solution in $C^2(\hat{U}_{\textnormal{eq}} \times [0,t_0], \mathbb{M}_\R(q,q))$ that satisfies the following 
for any $h \in \hat{U}_{\textnormal{eq}}$, $s \geq 0$ and $0\leq t \leq t_0$.
\begin{itemize}
\item[(A)] If $s \in [t,t_0]$ then $\Phi(h,t,t_0) =\Phi(h,t,s) \Phi(h,s,t_0).$
\item[(B)] For any $s \geq 0$, $\Phi( \hat{\psi}(h,s),t, t_0) = \Phi(h,t+s, t_0+s)$.
\item[(C)] If $v_0 \in \R^q_{*}$ then $\Phi(h,t,t_0) v_0 \in \R^q_{*}$. 
\end{itemize}
\end{lemma}
\begin{proof}
The function $\hat{\psi}$ is in $C^2(\R^q \times \R_+, \R^q)$ and the matrix-valued function $(h,t) \mapsto \hat{A}^T(\hat{\psi}(h,t))$ is in 
$C^2(\R^q \times \R_+, \mathbb{M}(q,q))$. Standard existence and uniqueness results for ordinary differential equations guarantee that there is a unique solution for 
\eqref{matrixvaluedequation} in the class $C^2(\hat{U}_{\textnormal{eq}} \times [0,t_0], \mathbb{M}_\R(q,q))$.

Part (A) of the lemma is just the Chapman-Kolmogorov property (see Proposition 2.12 in Chicone \cite{Chicone}). 
Note that due to the semigroup property for $\hat{\psi}$ (similar to \eqref{semigroup_property}) both $\Phi( \hat{\psi}(h,s),t, t_0)$ and 
$\Phi(h,t+s, t_0+s)$ satisfy the same equation for $t \in [0,t_0]$. Hence by uniqueness of solutions, part (B) is immediate.

We now prove part (C). Note that only the diagonal elements of the matrix $\hat{A}^T( \hat{\psi}(h,t))$ can be negative. For $t \in [0,t_0]$ let 
\[ c(t) = -\min_{1\leq i \leq q} \hat{A}_{ii}(\hat{\psi}(h,t)).\]
Fix a $v_0 \in \R^q_{*}$ and define 
\begin{align*}
L(t) = \exp \left( \int_{t_0 - t}^{t_0} c(s)ds \right) \Phi(h,t_0- t,t_0) v_0.  
\end{align*}
Then
\begin{align*}
\frac{d L(t)}{d t} =  \left( c(t_0 -t) I_{q} + \hat{A}^T(\hat{\psi}(h,t_0-t)) \right)L(t).
\end{align*}
But the matrix $\left( c(t_0 -t) I_{q} + \hat{A}^T(\hat{\psi}(h,t_0-t)) \right)$ has all entries positive for any $t \in [0,t_0]$. 
Since $L(0) = v_0 \in \R^q_*$, for any $t \in [0,t_0]$ we have 
$d L(t)/d t \in \R^q_+.$
Therefore $L(t) \in \R^q_{*}$ and this proves part (C). 
\end{proof}


\bibliographystyle{amsplain}
\bibliography{references}

\begin{thebibliography}{10}

\bibitem{Alee}
W.~Allee.
\newblock {\em Animal aggregations : A study in general sociology}.
\newblock University of Chicago Press, Chicago, USA, 1931.

\bibitem{AAWW}
S.~J. Altschuler, S.~B. Angenent, Y.~Wang, and L.~F. Wu.
\newblock On the spontaneous emergence of cell polarity.
\newblock {\em Nature}, 454:886--889, 2008.

\bibitem{Ball}
K.~Ball, T.~G. Kurtz, L.~Popovic, and G.~Rempala.
\newblock Asymptotic analysis of multiscale approximations to reaction
  networks.
\newblock {\em The Annals of Applied Probability}, 16(4):1925--1961, 2006.

\bibitem{AAWWref9}
A.~Butty, N.~Perrinjaquet, A.~Petit, M.~Jaquenoud, J.~Segall, K.~Hofmann,
  C.~Zwahlen, and M.~Peter.
\newblock A positive feedback loop stabilizes the guanine-nucleotide exchange
  factor cdc24 at sites of polarization.
\newblock {\em EMBO Journal}, 21:1565--1576, 2002.

\bibitem{Petzold}
Y.~Cao, D.~T. Gillespie, and L.~R. Petzold.
\newblock {The slow-scale stochastic simulation algorithm}.
\newblock {\em The Journal of Chemical Physics}, 122(1), Jan. 2005.

\bibitem{Chicone}
C.~Chicone.
\newblock {\em Ordinary differential equations with applications}, volume~34 of
  {\em Texts in Applied Mathematics}.
\newblock Springer-Verlag, New York, 1999.

\bibitem{DawsonEcole}
D.~A. Dawson.
\newblock {\em \'{E}cole d'\'{E}t\'e de {P}robabilit\'es de {S}aint-{F}lour
  {XXI}---1991}, volume 1541 of {\em Lecture Notes in Mathematics}.
\newblock Springer-Verlag, Berlin, 1993.
\newblock Papers from the school held in Saint-Flour, August 18--September 4,
  1991, Edited by P. L. Hennequin.

\bibitem{DKgen}
P.~Donnelly and T.~G. Kurtz.
\newblock Genealogical processes for {F}leming-{V}iot models with selection and
  recombination.
\newblock {\em The Annals of Applied Probability}, 9(4):1091--1148, 1999.

\bibitem{DK99}
P.~Donnelly and T.~G. Kurtz.
\newblock Particle representations for measure-valued population models.
\newblock {\em The Annals of Probability}, 27(1):166--205, 1999.

\bibitem{DN}
D.~G. Drubin and W.~J. Nelson.
\newblock Origins of cell polarity.
\newblock {\em Cell}, 84:335--344, 1996.

\bibitem{EK81}
S.~N. Ethier and T.~G. Kurtz.
\newblock The infinitely-many-neutral-alleles diffusion model.
\newblock {\em Advances in Applied Probability}, 13(3):429--452, 1981.

\bibitem{EK}
S.~N. Ethier and T.~G. Kurtz.
\newblock {\em Markov processes : {C}haracterization and {C}onvergence}.
\newblock Wiley Series in Probability and Mathematical Statistics: Probability
  and Mathematical Statistics. John Wiley \& Sons Inc., New York, 1986.

\bibitem{EK93}
S.~N. Ethier and T.~G. Kurtz.
\newblock {F}leming-{V}iot processes in population genetics.
\newblock {\em Siam {J}ournal on {C}ontrol and {O}ptimization}, 31(2):345--386,
  1993.

\bibitem{EwensBook}
W.~J. Ewens.
\newblock {\em Mathematical population genetics. {I}}, volume~27 of {\em
  Interdisciplinary Applied Mathematics}.
\newblock Springer-Verlag, New York, second edition, 2004.

\bibitem{FengPoisson}
S.~Feng.
\newblock {\em The {P}oisson-{D}irichlet distribution and related topics}.
\newblock Probability and its Applications (New York). Springer, Heidelberg,
  2010.
\newblock Models and asymptotic behaviors.

\bibitem{rdpde1}
P.~C. Fife.
\newblock {\em Mathematical aspects of reacting and diffusing systems},
  volume~28 of {\em Lecture Notes in Biomathematics}.
\newblock Springer-Verlag, Berlin, 1979.

\bibitem{FV}
W.~Fleming and M.~Viot.
\newblock Some measure-valued {M}arkov processes in population genetics theory.
\newblock {\em Indiana University Mathematics Journal}, 28:817--843, 1979.

\bibitem{AAWWref3}
A.~Gierer and H.~Meinhardt.
\newblock A theory of biological pattern formation.
\newblock {\em Kybernetik}, 12:30--39, 1972.

\bibitem{GUPTA}
A.~Gupta.
\newblock Stochastic model for cell polarity.
\newblock {\em The Annals of Applied Probability}, 22(2):827--859, 2012.

\bibitem{AAWWref14}
J.~E. Irazoqui, A.~S. Gladfelter, and D.~J. Lew.
\newblock Scaffold-mediated symmetry breaking by cdc42p.
\newblock {\em Nature Cell Biology}, 5:1062--1070, 2003.

\bibitem{Joffe}
A.~Joffe and M.~M{\'e}tivier.
\newblock Weak convergence of sequences of semimartingales with applications to
  multitype branching processes.
\newblock {\em Advances in Applied Probability}, 18(1):20--65, 1986.

\bibitem{HWKang}
H.-W. Kang and T.~G. Kurtz.
\newblock Separation of time-scales and model reduction for stochastic reaction
  networks.
\newblock {\em The Annals of Applied Probability (to appear)}, 2012.

\bibitem{kat}
G.~S. Katzenberger.
\newblock Solutions of a stochastic differential equation forced onto a
  manifold by a large drift.
\newblock {\em The Annals of Probability}, 19(4):1587--1628, 1991.

\bibitem{Khalil}
H.~K. Khalil.
\newblock {\em Nonlinear systems}.
\newblock Macmillan Publishing Company, New York, 1992.

\bibitem{kimura}
M.~Kimura.
\newblock Solution of a process of random genetic drift with a continuous
  model.
\newblock {\em Proceedings of the National Academy of Sciences},
  41(3):144--150, 1955.

\bibitem{KC64}
M.~Kimura and J.~Crow.
\newblock The number of alleles that can be maintained in a finite population.
\newblock {\em Genetics}, 49:725--738, 1964.

\bibitem{King75}
J.~F.~C. Kingman.
\newblock Random discrete distribution.
\newblock {\em Journal of the Royal Statistical Society. Series B.}, 37:1--22,
  1975.
\newblock With a discussion by S. J. Taylor, A. G. Hawkes, A. M. Walker, D. R.
  Cox, A. F. M. Smith, B. M. Hill, P. J. Burville, T. Leonard and a reply by
  the author.

\bibitem{LotkaBook}
A.~Lotka.
\newblock {\em Elements of Physical Biology}.
\newblock The Williams and Watkins company, Baltimore, 1925.

\bibitem{Moran}
P.~A.~P. Moran.
\newblock Random processes in genetics.
\newblock {\em Mathematical Proceedings of the Cambridge Philosophical
  Society}, 54(01):60--71, 1958.

\bibitem{Nisbet}
R.~M. Nisbet and W.~S.~C. Gurney.
\newblock {\em Modeling fluctuating populations}.
\newblock Wiley, 1982.

\bibitem{Oelschlager}
K.~Oelschl{\"a}ger.
\newblock On the derivation of reaction-diffusion equations as limit dynamics
  of systems of moderately interacting stochastic processes.
\newblock {\em Probability Theory and Related Fields}, 82(4):565--586, 1989.

\bibitem{Seneta}
E.~Seneta.
\newblock {\em Non-negative matrices and {M}arkov chains}.
\newblock Springer Series in Statistics. Springer, New York, 2006.

\bibitem{AAWWref7}
M.~Sohrmann and M.~Peter.
\newblock Polarizing without a c(l)ue.
\newblock {\em Trends Cell Biology}, 13:526--533, 2003.

\bibitem{Takaku}
T.~Takaku, K.~Ogura, H.~Kumeta, N.~Yoshida, and F.~Inagaki.
\newblock Solution structure of a novel cdc42 binding module of bem1 and its
  interaction with ste20 and cdc42.
\newblock {\em Journal of Biological Chemistry}, 285(25):19346--19353, 2010.

\bibitem{Thieme}
H.~R. Thieme.
\newblock {\em Mathematics in population biology}.
\newblock Princeton Series in Theoretical and Computational Biology. Princeton
  University Press, Princeton, NJ, 2003.

\bibitem{Varah}
J.~M. Varah.
\newblock A lower bound for the smallest singular value of a matrix.
\newblock {\em Linear Algebra and Applications}, 11:3--5, 1975.

\bibitem{Verhulst}
P.~Verhulst.
\newblock Notice sur la loi que la population poursuit dans son accroissement.
\newblock {\em Correspondance Math\'{e}matique et Physique}, 10:113--121, 1838.

\bibitem{Volterra}
V.~Volterra.
\newblock Fluctuations in the abundance of a species considered mathematically.
\newblock {\em Nature}, 118:558--560, 1926.

\bibitem{Weiner}
O.~Weiner, P.~Neilsen, G.~Prestwich, M.~Kirschner, L.~Cantley, and H.~Bourne.
\newblock A ptdinsp(3)- and rho gtpase-mediated positive feedback loop
  regulates neutrophil polarity.
\newblock {\em Nature Cell Biology}, 4(5):509--13, 2002.

\bibitem{Wright}
S.~Wright.
\newblock Evolution in {M}endelian populations.
\newblock {\em Genetics}, 16(2):97--159, 1931.

\end{thebibliography}


\end{document}